\documentclass[11pt,letterpaper,reqno]{amsart}
\usepackage{amssymb,amsmath,amsfonts,amscd,amssymb,amsmath,amsbsy,amsthm}
\usepackage[all]{xy}
\usepackage{mathrsfs}
\usepackage{latexsym}
\usepackage[pagebackref]{hyperref}
\usepackage{graphicx,tikz}
\usepackage{TooYoung}
\usepackage{enumerate}
\usepackage{TooYoung}
\usepackage{enumerate}

\newtheorem{Thm}{Theorem}

\newtheorem{Corol}[Thm]{Corollary}
\numberwithin{equation}{section}
\newtheorem{Th}{Theorem}[section]
\newtheorem{Lemma}[Th]{Lemma}
\newtheorem{Coro}[Th]{Corollary}
\newtheorem{Prop}[Th]{Proposition}
\newtheorem{Eg}[Th]{Example}

\theoremstyle{definition}
\newtheorem{Def}[Th]{Definition}

\theoremstyle{remark}
\newtheorem{Rmk}[Th]{Remark}

\begin{document}

\title%
[A Pieri type formula for motivic Chern classes of Schubert cells]%
{A Pieri type formula for motivic Chern classes of Schubert cells in Grassmannians}

\author{Neil J.Y. Fan}
\address[Neil J.Y. Fan]{Department of Mathematics, 
Sichuan University, Chengdu, Sichuan 610065, P.R. China}
\email{fan@scu.edu.cn}

\author{Peter L. Guo}
\address[Peter L. Guo]{Center for Combinatorics, LPMC, 
Nankai University, Tianjin 300071, P.R. China}
\email{lguo@nankai.edu.cn}

\author{Changjian Su}
\address[Changjian Su]{Yau Mathematical Sciences Center, Tsinghua University, Beijing, China}
\email{changjiansu@mail.tsinghua.edu.cn}

\author{Rui Xiong}
\address[Rui Xiong]{Department of Mathematics and Statistics, University of Ottawa, 150 Louis-Pasteur, Ottawa, ON, K1N 6N5, Canada}
\email{rxion043@uottawa.ca}

% \author[R.~Xiong]{Rui Xiong}
% \address[Rui Xiong]{Department of Mathematics and Statistics, University of Ottawa, 150 Louis-Pasteur, Ottawa, ON, K1N 6N5, Canada}
% \email{rxion043@uottawa.ca}

\maketitle
{}

\def\Fun{\operatorname{\mathsf{Fun}}}
\def\CSM{c_{\mathrm{SM}}}
\def\MC{\operatorname{MC}}
\def\SM{s_{\mathrm{SM}}}
\def\SMC{\operatorname{SMC}}
\def\H{{H}}
\def\S{\mathcal{S}}
\def\Haff{\widehat{\H}}
\def\ev{\operatorname{ev}}
\def\T{{T}}
\def\bT{{\overline{T}}}
\def\A{\mathcal{A}_{\hbar}}
\def\Gr{\operatorname{Gr}}
\def\Fl{\operatorname{Fl}}
\def\GL{\operatorname{\it GL}}
\def\S{\mathcal{S}}
\def\pt{\mathsf{pt}}
\def\wdd{\operatorname{wd}}
\def\htt{\operatorname{ht}}

\def\ftail#1{\mathop{\makebox{\(\textnormal{\raisebox{0.125pc}{\tiny\textbullet}\rule{-0.4pc}{0pc}\texttt{[}}#1\textnormal{\texttt{|}}\)}}\nolimits}%
\def\fhead#1{\mathop{\makebox{\(\textnormal{\texttt{|}}#1\textnormal{\raisebox{0.125pc}{\tiny\textbullet}\rule{-0.4pc}{0pc}\texttt{]}}\)}}\nolimits}%
\def\xftail#1{\mathop{\makebox{\(\textnormal{\texttt{[}}#1\textnormal{\raisebox{0.125pc}{\tiny\textbullet}\rule{-0.425pc}{0pc}\texttt{|}}\)}}\nolimits}%
\def\xfhead#1{\mathop{\makebox{\(   \textnormal{\raisebox{0.125pc}{\tiny\textbullet}\rule{-0.4pc}{0pc}\texttt{|}}
#1\textnormal{\texttt{]}}\)}}\nolimits}
\def\ttail#1{\mathop{\makebox{\(\textnormal{\texttt{[}}#1\textnormal{\texttt{|}}\)}}\nolimits}%
\def\hhead#1{\mathop{\makebox{\(\textnormal{\texttt{|}}#1\textnormal{\texttt{]}}\)}}\nolimits}%
\def\ff#1{\mathop{\makebox{\({\textnormal{\texttt{[}}#1\textnormal{\texttt{]}}}\)}}\nolimits}
\def\f#1{\mathop{\makebox{\({\textnormal{\texttt{(}}#1\textnormal{\texttt{)}}}\)}}\nolimits}

%\subsection*{Acknowledgement}

\begin{abstract}
We prove a Pieri formula for motivic Chern classes of Schubert cells in the  equivariant  K-theory of Grassmannians, which is described in terms of ribbon operators on partitions.
Our approach is to transform  the Schubert calculus over Grassmannians to the calculation in a certain  affine Hecke algebra.
As a consequence, we derive a Pieri formula for Segre motivic   classes of Schubert cells in Grassmannians. We  apply the Pieri formulas to establish a relation between motivic Chern classes and Segre motivic classes, extending a well-known relation  between the classes of structure sheaves and ideal sheaves.  As another application, we find a symmetric power series  representative for  the class  of the dualizing sheaf  of a Schubert variety.
\end{abstract}

\setcounter{tocdepth}{1}
\tableofcontents

% The limit to CSM classes can be found in \cite{AMSS19}. 

\section{Introduction}

The primary goal of this paper is to build a Pieri type formula for motivic  Chern classes of Schubert cells in the equivariant K-theory of Grassmannians. The notion of motivic Chern classes was defined by Brasselet, Sch\"urmann and Yokura \cite{BSY} as a generalization of the total $\lambda$ classes of the cotangent bundle to singular varieties. Its equivariant version was constructed by Feh\'er, Rim\'anyi, and Weber \cite{FRW}, see also Aluffi, Mihalcea, Sch\"urmann and Su \cite{AMSS19}.
Motivic Chern classes specialize to other important characteristic classes for singular varieties in homology, including for example  Chern classes by MacPherson \cite{MacPherson}, Todd classes by Yokura \cite{Yokura}, and $L$-classes by Cappell and Shaneson \cite{CS}.

We shall focus on (equivariant) motivic Chern classes of Schubert cells in the Grassmannian  $\Gr(k, n)$ or the (complete) flag variety  $\Fl(n)$. 
Motivic Chern classes of Schubert cells are equivalent to K-theoretic stable envelopes introduced by Okounkov \cite{Okounkov}, see \cite{AMSS19,FRW,K22,KW}, and can be computed via the Demazure--Lusztig operators \cite{AMSS19}.
They are also related, for example, to 
generalized oriented cohomology \cite{LSZZ}, 
mixed Hodge modules \cite{BSY,AMSS}, 
Tarasov--Varchenko weight functions \cite{FRW,Rimanyi}, Casselman basis \cite{AMSS19}, and Whittaker functions \cite{MSA}. We refer the readers to \cite{AMSS22} for several  `folklore properties' about  motivic Chern classes of Schubert cells. The Chevalley formula for  motivic Chern classes has been obtained recently in \cite{MNS2}.
%(such as, some specializations, how to recover the corresponding  Chern--Schwartz--MacPherson classes), see \cite{AMSS22}.

Our main result is a Pieri type formula for equivariant motivic Chern classes of Schubert cells in $\Gr(k, n)$, which will be described in Theorem \ref{ThmA} below. {Building upon this, we obtain a Pieri formula for equivariant Segre motivic classes of Schubert cells, the dual basis of motivic Chern classes with respect to the  Poincar\'e pairing, see Theorem \ref{ThmB}.}  

The Pieri formulas lead to two applications. First, guided by the similarity of the Pieri formulas in Theorem \ref{ThmA} and  Theorem \ref{ThmB}, we deduce a surprising relation between equivariant motivic Chern classes and  Segre motivic classes, see Theorem  \ref{ThmD}. This generalizes a well-known relation concerning the classes of structure sheaves and ideal sheaves for Grassmannians. Second, we apply Theorem \ref{ThmA}  to find a symmetric power series representative for the class of the dualizing sheaf of an opposite  Schubert variety, see 
Theorem \ref{ThmE}. The symmetric power series is constructed by means of the image of stable  Grothendieck polynomials under the omega involution investigated first by Lam and Pylyavskyy \cite{LP} from a combinatorial point of view. 

To state Theorem \ref{ThmA}, we need to introduce the notion of \emph{ribbon Schubert  operators}  on partitions inside the $k\times (n-k)$ rectangle
$(n-k)^k$. We do
not distinguish a partition $\lambda=(\lambda_1\geq \lambda_2\geq \cdots\geq  \lambda_k\geq 0)$ with its Young diagram, a left-justified array with $\lambda_i$ boxes in row $i$. 
A \emph{ribbon } (also called rim hook, or border strip) is a nonempty connected skew shape $\eta=\mu/\lambda$, containing no $2\times 2$ squares $\young{&\\&}$. 
The \emph{head} (resp.,  \emph{tail}) of $\eta$ is the box at its top right  (resp., lower left) end. 
The \emph{height} $\htt(\eta)$  (resp.,  \emph{width} $\wdd(\eta)$)   of $\eta$ is the number of rows (resp., columns) of $\eta$. 
%For example, let $n=9$,  $k=5$,  $\lambda=(4,2,1,1,0)$, and   $\mu=(4,4,3,1,0)$. Then $\eta=\mu/\lambda$ is a   ribbon with head   at $(2,4)$ and tail   at $(3,2)$. Clearly, we have  $\htt(\eta)=2$ and $\wdd(\eta)=3$.
%\[
%\young{&&&\\
%&&F|*(yellow)&|>|*(yellow)\\
%&<|*(yellow)&J|*(yellow)\\
%&[]\\|l|}.
%\]
For a ribbon $\eta=\mu/\lambda$  with head in row $i$ and tail in row $j$, let
\begin{equation}\label{POIU}
\mathtt{h}(\eta)=\mu_i+k+1-i,\qquad\mathtt{t}(\eta)=\lambda_j+k+1-j.
\end{equation}

Fix a partition $\lambda\subseteq(n-k)^k$. For $i=1,\ldots,k$,  the (head-valued) \emph{ribbon   Schubert  operator}, denoted by $\fhead{i}\to \lambda$, is defined as follows
\begin{equation}\label{POIU-1s}
\fhead{i}\to \lambda = 
t_{c}\cdot \lambda+
\sum_{\mu}
\left(\hbar-(p-q)\,t_{\mathtt{h}(\mu/\lambda)}\right)\,p^{\htt(\mu/\lambda)-1}q^{\wdd(\mu/\lambda)-1}\cdot\mu,
\end{equation}
where the sum is taken over all $\mu\subseteq (n-k)^k$ such that $\mu/\lambda$ is a ribbon with head in row $i$, $c=\lambda_i+k+1-i$, and $p, q,\hbar$ are parameters. %{\color{red}In the notation, we have suppressed the dependence on the parameters $p, q$  and $\hbar$. }
The ribbon Schubert operators may be regarded as generalizations of the classical Schur operators \cite{Fomin,FG},  the K-theory operators   \cite{LP}, and the ribbon Schur operators   \cite{Lam}.
Below is an example of the ribbon Schubert operator for $n=7$, $k=3$, $\lambda=(3,1,0)$, and $i=2$.
% $$\xfhead{3}\to 
% \young{&&u|\\\\\\l|}
% =t_3\young{&&u|\\\\{\Ythicker\wall{r}}||\\l|}
% +(1+y)(1-t_5)\young{&&u|\\&A|*(yellow)\,\,\,\,\,\bullet\\&V|*(yellow)3\\l|}
%     -y(1+y)(1-t_7)\young{&&A|*(yellow)\,\,\,\,\,\bullet\\&F|*(yellow)&J|*(yellow)\\&V|*(yellow)3\\l|}.
% $$
% $$\fhead{2}\to 
% \young{&&u|\\\\\\l|}
% =t_4\young{&&u|\\{\Ythicker\wall{r}}||\\\\l|}
% +(1+y)\left((1-t_5)\young{&&u|\\&*(yellow)\,\,\,2\bullet\\\\l|}
% +(1-t_5)\young{&&u|\\&A|*(yellow)\,\,\,2\bullet\\&V|*(yellow)\\l|}
% -y(1-t_5)\young{&&u|\\&A|*(yellow)\,\,\,2\bullet\\&H|*(yellow)\\<|*(yellow)&J|*(yellow)}\right).
% $$
\begin{align*}
\fhead{2}\to \young{&&&u|\\\\|l|}
&=t_3\cdot\young{&&&u|\\{\Ythicker\wall{r}}||\\l|}
+
(\hbar-(p-q)t_4)\cdot\young{&&&u|\\&||*(yellow)\,\,\,2\bullet\\|l|}+
(\hbar-(p-q)t_5)q\cdot\young{&&&u|\\&|<|*(yellow)&>|*(yellow)\,\,\,2\bullet\\|l|}\\
&\quad+
(\hbar-(p-q)t_4)pq\cdot\young{&&&u|\\&|A|*(yellow)\,\,\,2\bullet\\<|*(yellow)&J|*(yellow)}+
(\hbar-(p-q)t_5)pq^2\cdot\young{&&&u|\\&F|*(yellow)&>|*(yellow)\,\,\,2\bullet\\<|*(yellow)&J|*(yellow)}.
\end{align*}

% {\color{red}To prove Theorem \ref{ThmA}, we define an affine Hecke algebra structure *** Then we  transform the expansion of the product on the left-hand side of \eqref{ThmA-PPP} to the calculation of certain product in the affine Hecke algebra****}

For $0\leq r\leq k$, let $c_r(\mathcal{V}^\vee)\in K_T(\Gr(k,n))$ be the $r$-th equivariant Chern class of the dual tautological bundle  over $\Gr(k,n)$. 
%{\color{red}Note that under the equivariant Chern character, the leading term of $c_r(\mathcal{V}^\vee)$ is the equivariant Chern class in cohomology. }
%{\color{red}Moreover, $c_r(\mathcal{V}^\vee)$ corresponds to the $r$-th elementary symmetric polynomial after identifying the structure sheaves with double symmetric Grothendieck polynomials.} 

\begin{Thm}\label{ThmA}
For $\lambda\subseteq (n-k)^k$, let  $\MC_y(Y(\lambda)^\circ)\in K_T(\Gr(k,n))[y]$ be the equivariant motivic Chern class indexed by  the opposite Schubert cell $Y(\lambda)^\circ$ in $\Gr(k, n)$. Set 
 $(p, q, \hbar)=(1, -y, 1+y)$. 
Then, for $0\leq r\leq k$, we have 
\begin{align}\label{ThmA-PPP}
c_r(\mathcal{V}^\vee)\cdot \MC_y(Y(\lambda)^\circ)
=\sum_{1\leq i_1<\cdots<i_r\leq k}
\fhead{i_r}\to \cdots \to \fhead{i_1}\to 
\MC_y(Y(\lambda)^\circ).
\end{align}
Here, the notation 
\[
\fhead{i}\to \MC_y(Y(\lambda)^\circ)  = 
t_{c}\cdot \MC_y(Y(\lambda)^\circ) +(1+y)
\sum_{\mu}
\left(1-\,t_{\mathtt{h}(\mu/\lambda)}\right)(-y)^{\wdd(\mu/\lambda)-1}\cdot\MC_y(Y(\mu)^\circ) 
\]
is defined in  the same manner as  \eqref{POIU-1s} (in the case $(p, q, \hbar)=(1, -y, 1+y)$), except that here we replace  $\lambda$ by $\MC_y(Y(\lambda)^\circ)$, and   $\mu$ by    $\MC_y(Y(\mu)^\circ)$.
\end{Thm}

The strategy to prove Theorem \ref{ThmA} is to transform the Schubert calculus over Grassmannians to the determination of products in a certain affine Hecke algebra. This will be done through Sections \ref{sec3}--\ref{sec5}.

For an example to illustrate 
Theorem \ref{ThmA}, let us 
take $n=5$, $k=2$, $\lambda=(2,0)$, and $r=2$. By \eqref{ThmA-PPP}, we see that
\[
c_2(\mathcal{V}^\vee)\cdot \MC_y(Y(2,0)^\circ)=\fhead{2}\to \fhead{1}\to \MC_y(Y(2,0)^\circ).
\]
The following depicts the procedure of  
$\fhead{2}\to \fhead{1}\to \lambda=(2,0)$.
$$
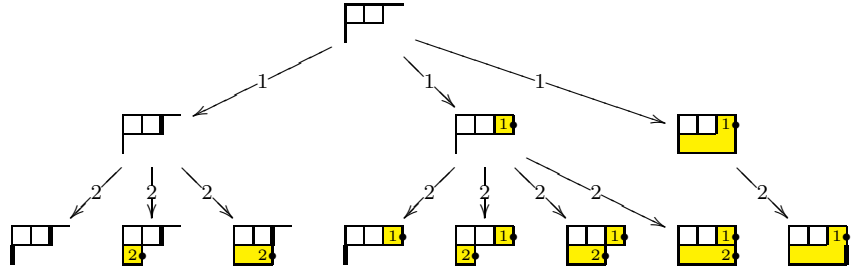
\xymatrix@!=1.5pc{
&&&
\young{&&u|\\l|}
\ar[dll]|{{1}}
\ar[dr]|{{1}}
\ar[drrr]|{{1}}
&&&%\makebox[0pc]{\fbox{via $\fhead{i}$}}
\\
&
\young{&&{\Ythicker\wall{l}}|u|\\l|}
\ar[dl]|{{2}}
\ar[dr]|{{2}}
\ar[d]|{{2}}
&&&
\young{&&*(yellow)\,\,\,1\bullet\\l|}
\ar[dl]|{{2}}
\ar[dr]|{{2}}
\ar[d]|{{2}}
\ar[drr]|{{2}}
&&
\young{&&A|*(yellow)\,\,\,1\bullet\\
<|*(yellow)&=|*(yellow)&|J|*(yellow)}
\ar[dr]|{{2}}
\\
\young{&&{\Ythicker\wall{l}}|u|\\{\Ythicker\wall{l}}|l|}
&
\young{&&{\Ythicker\wall{l}}|u|\\
*(yellow)\,\,\,2\bullet}
&
\young{&&{\Ythicker\wall{l}}|u|\\
<|*(yellow)&>|*(yellow)\,\,\,2\bullet}
&
\young{&&*(yellow)\,\,\,1\bullet\\{\Ythicker\wall{l}}|l|}
&
\young{&&*(yellow)\,\,\,1\bullet\\
*(yellow)\,\,\,2\bullet}
&
\young{&&*(yellow)\,\,\,1\bullet\\
<|*(yellow)&>|*(yellow)\,\,\,2\bullet}
&
\young{&&*(yellow)\,\,\,1\bullet\\
<|*(yellow)&=|*(yellow)&>|*(yellow)\,\,\,2\bullet}
&
\young{&&A|*(yellow)\,\,\,1\bullet\\
<|*(yellow)&=|*(yellow)&{\Ythicker\wall{r}}|J|*(yellow)}}$$
%-----------------------------
% diagram via [i|*.
%-----------------------------
%
So, with the abbreviation 
$M_{\lambda}=\MC_y(Y(\lambda)^\circ)$, we obtain that
\begin{align*}
c_2(\mathcal{V}^\vee)\cdot M_{(2,0)} 
&=  
t_1t_4\cdot M_{(2,0)}   
+(1+y)(1-t_2)t_4\cdot M_{(2,1)}
- y(1+y)(1-t_3)t_4\cdot M_{(2,2)}\\
&\quad +(1+y)t_1(1-t_5)\cdot M_{(3,0)} 
+ (1+y)^2(1-t_2)(1-t_5)\cdot M_{(3,1)}\\
&\quad- y(1+y)^2(1-t_3)(1-t_5)\cdot M_{(3,2)}
 + y^2(1+y)^2(1-t_4)(1-t_5)\cdot M_{(3,3)}\\
&\quad+y^2(1+y)t_4(1-t_5)\cdot M_{(3,3)}.
\end{align*}

Based on Theorem \ref{ThmA}, we can derive a Pieri formula for Segre motivic classes, which live in the localized equivariant K-theory of the Grassmannian, denoted by $\mathbb{K}_T(\Gr(k,n))$. The formulation requires {tail-valued} ribbon Schubert operators, denoted by $\xfhead{i}\to \lambda $, which has the same formulation as $\fhead{i}\to \lambda$ defined in \eqref{POIU-1s}, except that the variable $t_{\mathtt{h}(\mu/\lambda)}$ is replaced by $t_{\mathtt{t}(\mu/\lambda)}$ (the subscript $\mathtt{t}(\mu/\lambda)$ is defined in \eqref{POIU}).

\begin{Thm}\label{ThmB}
Set $(p, q, \hbar)=(1, -y, 1+y)$. For $\lambda\subseteq (n-k)^k$ and $0\leq r\leq k$, the equivariant  Segre motivic  class $\SMC_y(Y(\lambda)^\circ)\in \mathbb{K}_T(\Gr(k,n))$
 %K_T(\Gr(k,n))[y]$ 
satisfies the following Pieri formula:
\begin{align}\label{ThmB-QQQ}
c_r(\mathcal{V}^\vee)\cdot \SMC_y(Y(\lambda)^\circ)
=\sum_{1\leq i_1<\cdots<i_r\leq k}
\xfhead{i_r}\to \cdots \to \xfhead{i_1}\to 
\SMC_y(Y(\lambda)^\circ).
\end{align}
\end{Thm}

Since $\MC_y(Y(\lambda)^\circ)|_{y=0} = [\mathcal{I}_{\partial Y(\lambda)}]$ and 
$\SMC_y(Y(\lambda)^\circ)|_{y=0} = [\mathcal{O}_{Y(\lambda)}]$, 
setting $y=0$ in Theorem \ref{ThmA} and \ref{ThmB}, we obtain Pieri formulas for the classes of ideal sheaves and structure sheaves, whose non-equivariant version was due to Lenart \cite{Lenart1} (see Remark \ref{JHG678}).

\begin{Corol}\label{coroC}
The non-equivariant  motivic Chern classes and  non-equivariant Segre motivic classes enjoy completely the same Pieri formula, which is obtained respectively from \eqref{ThmA-PPP} and  \eqref{ThmB-QQQ} by
replacing $\fhead{i}$ and  $\xfhead{i}$ with the operator  $\hhead{i}$. 
More precisely, $\hhead{i}$ is obtained from \eqref{POIU-1s} by letting  $(p, q, \hbar)=(1, -y, 1+y)$ and 
$t=0$:
\begin{equation*} 
\hhead{i}\to \lambda = 
(1+y)\sum_{\mu/\lambda=\eta}
(-y)^{\wdd(\eta)-1}\cdot\mu
\end{equation*}
with $\mu$ ranging  over  partitions in  $(n-k)^k$  such that $\mu/\lambda$ is a ribbon with head in row $i$.
\end{Corol}

The similarity between the Pieri formulas in Theorems \ref{ThmA} and \ref{ThmB} indicates that
there should be some relation between   motivic Chern classes and Segre motivic classes.  Recall that
the total $\lambda$-class of the cotangent bundle in 
$K_T(\Gr(k,n))[y]$ is 
$$\lambda_y(\mathscr{T}^\vee_{\Gr(k,n)})
=\sum_{r\geq 0}y^r[\mathsf{\Lambda}^r\mathscr{T}^\vee_{\Gr(k,n)}].$$
%See Subsection \ref{MCCCC} for more information. 
%Comparing the Pieri formulas in Theorems \ref{ThmA} and \ref{ThmB}, we obtain the following result. 
% We denote $\det\mathcal{V}=\mathsf{\Lambda}^{k}\mathcal{V}$, where $\mathcal{V}$ is the tautological bundle.
% Note that we have $[\det\mathcal{V}]=1-\mathcal{O}_{Y(\square)}$ in non-equivariant. 

\begin{Thm}\label{ThmD}
Let $\lambda\subseteq (n-k)^k$. 
Over $\mathbb{K}_T(\Gr(k,n))$, we have 
\begin{equation}\label{eq:thmMCSMC}
\lambda_y(\mathscr{T}^\vee_{\Gr(k,n)})\cdot \frac{(1-[\mathcal{O}_{Y(\square)}])\cdot \SMC_y(Y(\lambda)^\circ)}{1-[\mathcal{O}_{Y(\square)}]|_{\lambda}}=
 \MC_y(Y(\lambda)^\circ),
 \end{equation}
where $[\mathcal{O}_{Y(\square)}]$ is the class of the structure sheaf for a single box $\square$, and   $\cdot |_{\lambda}$ is the restriction map (see Section \ref{sec:relMCSMC} for detailed information). 
%Restricting to $K(\Gr(k,n))[y]$, 
In the non-equivariant case, we get
\begin{equation}\label{nonequismc}
\lambda_y(\mathscr{T}^\vee_{\Gr(k,n)})\cdot (1-[\mathcal{O}_{Y(\square)}])\cdot \SMC_y(Y(\lambda)^\circ)=
\MC_y(Y(\lambda)^\circ).
\end{equation}
\end{Thm}

% We do not have (and expect) a geometric proof of   \eqref{eq:thmMCSMC}. 
Setting $y=0$ in  \eqref{nonequismc} 
%(and by Remark \ref{RRR-N}),  
%CS: already mentioned this after Theorem B
we obtain the following  relation in the ordinary K-theory due to Buch \cite[\textsection 8]{Buch}:
$$(1-[\mathcal{O}_{Y(\square)}])\cdot [\mathcal{O}_{Y(\lambda)}]=
[\mathcal{I}_{\partial Y(\lambda)}].$$
This was generalized to the equivariant K-theory, see, for example, \cite[Prop. 4.2]{BCMP},
$$\frac{(1-[\mathcal{O}_{Y(\square)}])\cdot [\mathcal{O}_{Y(\lambda)}]}{1-[\mathcal{O}_{Y(\square)}]|_{\lambda}}
=[\mathcal{I}_{\partial Y(\lambda)}],$$
which is an immediate consequence of \eqref{eq:thmMCSMC} by setting $y=0$.

As another application of Theorem \ref{ThmA}, we find a symmetric power series representative of the class $[\omega_{Y(\lambda)}]\in K(\Gr(k,n))$ of the dualizing sheaf of the opposite Schubert variety $Y(\lambda)$. To be specific, let $G_{\lambda}(x)$ be the stable Grothendieck polynomial, which is a symmetric power series introduced by Fomin and Kirillov \cite{FK-1} as the limit of the Grothendieck polynomials for Grassmannian permutations. 
%When restricted to the first $k$ variables,
%$G_{\lambda}(x_1,\ldots,x_k)$ is a polynomial representative of the class $[\mathcal{O}_{Y(\lambda)}]$ of the structure sheaf for $Y(\lambda)$.
%The geometric meaning of $G_{\lambda}(x)$ is 
%$G_{\lambda}%(x_1,\ldots,x_k,0,0,\ldots)
%\text{ represents }[\mathcal{O}_{Y(\lambda)}]\in K_T(\Gr(k,n)).$$
% {\color{red} 
% where the left-hand side is symmetric in $x_1,\ldots,x_k$, thus defines a class over $K_T(\Gr(k,n))$. 
% Explanation, symmetric, so it is class of $K(Gr(k,n))$.}
Buch \cite{Buch} showed that $G_\lambda(x)$ can be interpreted as the generating function of set-valued tableaux of shape $\lambda$.
Let $J_\lambda(x)$ be the image of $(-1)^{|\lambda|}G_{\lambda}(-x_1,-x_2,\ldots)$ under the omega involution which sends  elementary symmetric functions $e_r(x)$ to  homogeneous symmetric functions $h_r(x)$. 
As proved by Lam and  Pylyavskyy  \cite[\textsection 9.7]{LP}, $J_{\lambda}(x)$ 
can be combinatorially generated by 
weak set-valued tableaux of shape $\lambda$.   

Define an algebra homomorphism 
\begin{equation}
\rho\colon \hat{\Lambda}\longrightarrow K(\Gr(k,n)),\qquad e_r(x) \longmapsto 
\begin{cases}
c_r(\mathcal{V}^\vee), & r\leq k,\\
0, & r> k.
\end{cases}    
\end{equation}
from the ring
$\hat{\Lambda}$ of symmetric power series  to $K(\Gr(k,n))$.
The following theorem reveals the geometric meaning of $J_\lambda(x)$.

%We define a homomorphism $\epsilon$ in \eqref{epsilo} from the ring of symmetric power series to $K(\Gr(k,n))$ by sending elementary symmetric functions to Chern classes.  
%The details will be given in  Section \ref{sec:dualsheaf}.

\begin{Thm}\label{ThmE}
Let $\lambda\subseteq (n-k)^k$, and $\lambda'$ be its conjugate partition.
Then
\begin{equation}
\rho\big((1-G_{\square}(x))^nJ_{\lambda'}(x)\big)=[\omega_{Y(\lambda)}].    
\end{equation}
\end{Thm}

This paper is organized as follows. In Section \ref{sec:geom}, we provide the necessary geometric background. 
In Section \ref{sec3}, we define an affine Hecke algebra $\Haff_n$ depending on parameters $p$, $q$ and $\hbar$, and demonstrate that   Schubert classes, classes of ideal sheaves and structure sheaves, Chern--Schwartz--MacPherson classes, and motivic Chern classes can be realized as basis in particular representations of $\Haff_n$. In Section
\ref{sec4}, we prove the equivalence between Schubert calculus over Grassmannians and the computation of certain products in $\Haff_n$. 
In Section \ref{sec5}, we show that the products in $\Haff_n$ that we need to calculate can be formulated in terms of the ribbon Schubert operators 
$\fhead{i}$. Section \ref{sec-66} is devoted to proofs of the theorems listed in the introduction. In Section \ref{sec67}, by taking specific values of $p$, $q$ and $\hbar$,  we further exhibit Pieri formulas for 
Schubert classes, classes of ideal sheaves and structure sheaves, Chern--Schwartz--MacPherson classes, and  Segre--MacPherson classes. From this point of view, our approach provides a unified way to deal with Pieri type formulas. The appendix section proves the equivalences between different types of ribbon Schubert operators. 
 
\subsection*{Acknowledgements}
We are grateful to Cristian Lenart and  Leonardo Mihalcea
for valuable discussions and suggestions. N.F. was supported by
the National Natural Science Foundation of China (No. 12071320) and Sichuan Science and Technology Program (No. 2023ZYD0012), P.G. was  supported by the National Natural Science Foundation of China (Nos. 11971250, 12371329), and 
R.X. acknowledges the partial support from the NSERC Discovery grant RGPIN-2015-04469, Canada. 

\section{Geometric background}\label{sec:geom}

This section lays down the geometric preliminaries.
%that we are concerned with in this paper. 
Here and throughout, we assume that  $G=\GL_n(\mathbb{C})$,  $P=\binom{\GL_k\quad\,\,*\,\,\quad}{\phantom{\GL_k}\GL_{n-k}}$ is a standard parabolic subgroup, and  $B\subset P$ is the Borel subgroup consisting of upper triangular matrices of $G$. The Grassmannian  $\Gr(k,n)$  of $k$-dimensional subspaces in $\mathbb{C}^n$  is a homogeneous variety isomorphic to $G/P$. The {flag variety} $\Fl(n)$ is the variety of (complete) flags 
$\{0\}=V_0\subsetneq V_1\subsetneq\cdots \subsetneq V_n=\mathbb{C}^n$
of $\mathbb{C}^n$, which is isomorphic to
$G/B$.

\subsection{Schubert cells and  varieties}\label{FEDC}

Let $T$ be the torus of diagonal matrices of $G$. The associated Weyl group $\S_n=N_G(T)/T$ is the symmetric group of permutations of $\{1,2, \ldots, n\}$. For $1\leq i<n$, let $s_i=(i, i+1)$ be the simple transposition interchanging   $i$ and $i+1$. Then $\{s_1,\ldots, s_{n-1}\}$ constitutes a generating set of $\S_n$,   satisfying $s_i^2=1$ and the braid relations: $s_is_j=s_js_i$ for $|i-j|>1$, $s_is_{i+1}s_i=s_{i+1}s_is_{i+1}$.
Each   $w\in \S_n$ can be decomposed into a product of simple transpositions, and its length $\ell(w)$   is the minimum $\ell$ such that $w=s_{i_1}\cdots s_{i_\ell}$.
Combinatorially,   $\ell(w)$ equals the number of   inversions of $w$:
\[
\ell(w)=|\{(i, j)\colon 1\leq i<j\leq n, \,w(i)>w(j)\}|.
\]
For $u,w\in \S_n$, we say $u\leq w$ in the {\it Bruhat order}   if $u$ can be obtained by erasing some factors of a reduced decomposition of $w$.

Let $\S_k\times \S_{n-k}$ be the maximal  parabolic subgroup  of $\S_n$   generated by $\{s_1,\ldots, s_{n-1}\}\setminus \{s_k\}$. That is,  $w\in \S_n$ belongs to $\S_k\times \S_{n-k}$ if and only if $\{w(i)\colon 1\leq i\leq k\}=\{1,2,\ldots, k\}$.
Each  $w\in \S_n$ admits a unique \emph{parabolic decomposition} 
\begin{equation}\label{eq:parabolicdecomposition}
w=w_{\lambda}\,v,
\end{equation}
where  $v\in \S_k\times \S_{n-k}$, and $w_\lambda$ is a {\it Grassmannian permutation} indexed by a partition $\lambda$ inside $(n-k)^k$ (which is the
minimal coset representative of the left coset of $\S_k\times \S_{n-k}$ with respect to $w$). 
To be precise, $w_\lambda$ is the permutation obtained from $w$ by 
rearranging the first $k$ values   as well as  the last $n-k$ values in increasing order, with
 $\lambda$ being  determined by
\begin{equation}\label{GRWS-12}
w_{\lambda}(k+1-i)=\lambda_i+k+1-i \qquad \text{for   $i=1,\ldots,k$}.
\end{equation}
The correspondence in \eqref{GRWS-12} describes  a  bijection between partitions inside $(n-k)^k$
and Grassmannian permutations of $\S_n$ 
with descent at position $k$.
Notice that $\ell(w_\lambda)=|\lambda|$, which is the number of boxes in $\lambda$, and that $\ell(w)=\ell(w_\lambda)+\ell(v)=|\lambda|+\ell(v)$.
We shall adopt the notation $\Gr(w)=\lambda$. For example, 
for $k=3$ and $w=4165273\in \S_7$,
we have $w_\lambda=1462357$ with $\Gr(w)=\lambda=(3,2,0)$, and $v=213 6475\in \S_3\times \S_4$.

The Grassmannian $\Gr(k,n)=G/P$ is a disjoint union of {\it Schubert cells }indexed by partitions inside $(n-k)^k$. To each $\lambda\subseteq (n-k)^k$, the associated { Schubert cell} is  $X(\lambda)^\circ = Bw_{\lambda}P/P$. 
The {\it Schubert variety} $X(\lambda)=\overline{X(\lambda)^\circ}$ is its closure.
The {\it opposite} Schubert cell and variety are defined accordingly as 
$Y(\lambda)^\circ = B^-w_{\lambda}P/P$ and $Y(\lambda) = \overline{Y(\lambda)^\circ}$,
where $B^-$ is the opposite Borel subgroup consisting of lower triangular matrices of $G$. 

The flag variety $\Fl(n)=G/B$ has a decomposition into Schubert cells indexed by permutations of $\S_n$. The Schubert cell indexed by $w\in \S_n$ is 
$X(w)^\circ = BwB/B$,
whose closure is  the Schubert variety $X(w) = \overline{X(w)^\circ}$.
Similarly, the {opposite Schubert cell  and variety} in   $G/B$ are 
$Y(w)^\circ = B^-wB/B$ and $ Y(w) = \overline{Y(w)^\circ}$.
Let
\begin{align}\label{eq:pi:GB->GP}
    \pi\colon G/B\longrightarrow G/P
\end{align}
 be the natural projection.
 % {\color{red}sending ***}. 
Then
\begin{gather*}
\pi(X(w)^\circ) = X(\Gr(w))^\circ,\qquad
\pi(X(w)) = X(\Gr(w)),
\\[5pt]
\pi(Y(w)^\circ) = Y(\Gr(w))^\circ,\qquad
\pi(Y(w)) = Y(\Gr(w)).
\end{gather*}

\subsection{Equivariant cohomology and Schubert classes}

This subsection is mainly a brief overview of the equivariant cohomology of $\Gr(k,n)$ or $\Fl(n)$, as well as the basis of Schubert classes of Schubert varieties. The readers are referred to \cite{AF,Brion1997}
for more details.

Let   $\mathbb{E}_N=(\mathbb{C}^N\setminus 0)^n$ equipped with the action of torus  $T$. 
For a nonsingular $T$-variety  $X$, the {equivariant cohomology} of $X$ is a graded ring whose $i$-th component is
$$H^{i}_T(X):=H^{i}(\mathbb{E}\times^T X),$$
where $\mathbb{E}=\mathbb{E}_N$ for $N\gg 0$.  
A $T$-equivariant vector bundle $\mathcal{V}$ over $X$ induces a bundle $\mathcal{V}'=\mathbb{E}\times^T \mathcal{V}$ over $\mathbb{E}\times^T X$. For $i\geq 0$,   \emph{the $i$-th equivariant Chern class} is defined by 
$$c_i(\mathcal{V}):=c_i(\mathcal{V}')\in 
H^{2i}(\mathbb{E}\times^T X)=H^{2i}_T(X). $$
A $T$-equivariant subvariety $Y\subset X$ of codimension $d$ defines a subvariety $Y'=\mathbb{E}\times^T Y\subset \mathbb{E}\times^T X$ of the same codimension, its \emph{equivariant fundamental class} is
$$[Y] := [Y']\in H^{2d}(\mathbb{E}\times^T X)=H_T^{2d}(X).$$
% If we further assume $X$ to be proper, the Poincar\'e pairing
% $$\left<\alpha,\beta\right>=\int_{X} \alpha\cup \beta \in H^\bullet_T(\pt)$$
% is nondegenerate. 

% Write $\widehat{T}$ for its character group:
% $$\widehat{T}=\operatorname{Hom}_{\mathsf{AlgGrp}}(T,\mathbb{C}^\times) = \mathbb{Z}z_1\oplus \cdots\oplus \mathbb{Z}z_n,$$
% where $z_i$ is the character of $T$ of the $i$-th projection
% $$T\ni\operatorname{diag}(z_1,\ldots,z_n)\longmapsto z_i\in \mathbb{C}^\times. $$

Denote by $\pt=\operatorname{Spec}\mathbb{C}$   a single $\mathbb{C}$-point.
Following the historical convention in Schubert calculus, 
we write 
\begin{align*}
t_i & = c_1(\mathbb{C}_{-z_i})\in H^2_T(\pt), 
\end{align*}
where $\mathbb{C}_{-z_i}$ is the bundle corresponding to the character
$$T\ni\operatorname{diag}(z_1,\ldots,z_n)\longmapsto z_i^{-1}\in \mathbb{C}^\times. $$
Then $H^\bullet_T(\pt)$ is naturally isomorphic to the ring of polynomials in $t_1,\ldots, t_n$:
\[
H^\bullet_T(\pt)
=\mathbb{Q}[t_1,\ldots,t_n].
\]

Now we consider the case when $X$ is $\Gr(k,n)$ or $\Fl(n)$.  
Over the Grassmannian, we have the \emph{tautological bundle} $\mathcal{V}$ whose total space is 
$\{(V,v)\in \Gr(k,n)\times \mathbb{C}^n\colon v\in V\}$. 
Let  $\mathcal{V}^\vee$ be the  \emph{dual} tautological bundle of $\mathcal{V}$. 
For $0\leq r\leq k$, let
\begin{align}
c_r(\mathcal{V}^\vee) \in H^{2r}_T(G/P)
\end{align}
be the $r$-th equivariant Chern class of $\mathcal{V}^\vee$. 
Similarly, over the flag variety, we have the \emph{tautological flag}
$0= \mathcal{V}_0\subsetneq \mathcal{V}_1\subsetneq \cdots \subsetneq \mathcal{V}_n=\mathcal{O}^{\oplus n}$,
where the total space of $\mathcal{V}_i$  is 
$\{(V_\bullet,v)\in \Fl(n)\times \mathbb{C}^n\colon v\in V_i\}$. 
For $i=1,\ldots,n$, let
\begin{align}\label{X-123}
x_i = c_1\big((\mathcal{V}_i/\mathcal{V}_{i-1})^{\vee}\big)\in H^2_T(G/B)
\end{align}
be the first equivariant Chern class of $(\mathcal{V}_i/\mathcal{V}_{i-1})^{\vee}$.
Denote by
\begin{align}\label{eq:pullcoh}
\pi^*\colon H^\bullet_T(G/P)\longrightarrow
% \mathop{\lhook\joinrel\xrightarrow{\quad}} 
H^\bullet_T(G/B)
\end{align}
the pullback algebra homomorphism induced by \eqref{eq:pi:GB->GP}. 
By the Whitney product formula of Chern classes, it is not hard to obtain that
$$\pi^*c_r(\mathcal{V}^\vee)=e_r(x_1,\ldots,x_k)\in H^\bullet_T(G/B),$$
where 
\[
e_r(x_1,\ldots, x_k)=\sum_{1\leq i_1<\cdots<i_r\leq k} x_{i_1}\cdots x_{i_r}
\]
is the $r$-th elementary symmetric polynomial in $x_1,\ldots, x_k$.

\begin{Th}[Borel \cite{Borel}]\label{Borelthm:inj}
We have 
\begin{enumerate}[\rm(i)]
\setlength{\itemsep}{1ex}
    \item The algebra $H_T^\bullet(G/B)$ is generated by $x_1,\ldots,x_n$ over $H_T^\bullet(\pt)$; 
    
    \item The algebra $H_T^\bullet(G/P)$ is generated by $c_1(\mathcal{V}^\vee),\ldots,c_k(\mathcal{V}^\vee)$ over $H_T^\bullet(\pt)$;
    
    \item The pullback $\pi^*$ defined in \eqref{eq:pullcoh} is an injective algebra homomorphism. 
\end{enumerate}
\end{Th}

Schubert varieties are $T$-equivariant,  and the corresponding equivariant fundamental classes
\begin{align*}
[X(\lambda)], \,\,  [Y(\lambda)] &\in H^\bullet_T(G/P), \qquad\lambda\subseteq (n-k)^k,\\
[X(w)],\,\,  [Y(w)] & \in H^\bullet_T(G/B),\qquad w\in \S_n,
\end{align*}
are    known as equivariant \emph{Schubert classes}. 
 We next collect some properties about Schubert classes. 
The  \emph{BGG Demazure operator} over $H_T^\bullet(G/B)$ \cite{BGG} is defined by 
\begin{align}\label{partiali}
\partial_i =\frac{\operatorname{id}-s_i}{x_i-x_{i+1}},  
\end{align}
where $s_i$ is the right Weyl group action exchanging $x_i$ and $x_{i+1}$. 
% To be precisely, for any $\gamma\in H_T^\bullet(G/B)$,  $\partial_i\gamma$ is the unique class such that 
% $$(x_i-x_{i+1})\partial_i \gamma = \gamma-s_i\gamma.$$
% % Let us first define the \emph{Demazure operator} on a polynomial $f$ in $x_1,\ldots,x_n$, 
% % \begin{align*}
% % \partial_if=\frac{f-s_if}{x_i-x_{i+1}},
% % \end{align*}
% % where $s_if$ is obtained by exchanging $x_i$ and $x_{i+1}$. 
% This operator can be geometrically constructed by push-pull, thus in particular is well-defined over $H_T^\bullet(G/B)$, see [REF]. 
%These operators are geometrically induced by push-pull \cite{AF}. 
These operators satisfy  the relations $\partial_i^2  = 0$,
$\partial_i\partial_j=\partial_j\partial_i$ for $|i-j|>1$, and $\partial_i\partial_{i+1}\partial_i=\partial_{i+1}\partial_i\partial_{i+1}$.
Hence, for $w\in \S_n$, one may define $\partial_w=\partial_{i_1}\cdots\partial_{i_\ell}$ for any reduced decomposition $w=s_{i_1}\cdots s_{i_\ell}$. 

We also recall the {restriction map}. 
Denote $e=1\cdot B/B\in G/B$. Since $e$ is a torus fixed point, we have the restriction map
$$\cdot |_e: H_T^\bullet(G/B)\longrightarrow H_T^\bullet(e)\cong H_T^\bullet(\pt). $$
This map satisfies 
$$x_i|_e=t_i,\qquad t_i|_e=t_i.$$ 

\begin{Prop}[see for example {\cite{AF}}]\label{SchRep1}
We have
\begin{enumerate}[\rm(i)]
\setlength{\itemsep}{1ex}

\item $\big\{[Y(w)]\colon w\in \S_n\big\}$ forms a basis of $H^\bullet_T(G/B)$ over $H^\bullet_T(\pt)$; 

\item $[Y(w)] = \partial_{w^{-1}w_0}[Y(w_0)]$;

\item $[Y(w)]\big|_e\neq 0$ if and only if $w=\operatorname{id}$.

\end{enumerate}
\end{Prop}

Let us turn to the Grassmannian case. 
Recall   the \emph{Poincar\'e pairing} 
$$\left<\alpha,\beta\right>=\int_{G/P} \alpha\cup \beta \in H^\bullet_T(\pt)$$
for $\alpha,\beta\in H_T^\bullet(G/P)$. 
Denote by $w_0^L$ the left Weyl group action on $H^\bullet_T(G/P)$ for the longest element $w_0=n\cdots 21$. It satisfies 
$$w_0^Lc_r(\mathcal{V}^\vee)=c_r(\mathcal{V}^\vee) \quad (1\leq r\leq k),\qquad
w_0^Lt_i=t_{n+1-i}\quad (1\leq i\leq n).$$
For   $\lambda\subseteq (n-k)^k$, the dual partition 
$\overline{\lambda}=(\overline{\lambda}_1,\ldots, \overline{\lambda}_k) $ of $\lambda$ is complement of $\lambda$ in the rectangle $(n-k)^k$, that is, 
\begin{align}\label{eq:lambdabar}
    \lambda_i+\overline{\lambda}_{k+1-i}=n-k,\quad \text{for  $i=1,\ldots,k$}.
\end{align}

\begin{Prop}[see for example {\cite{AF}}]\label{SchRep1:Gr}
Let  ${\lambda}\subseteq (n-k)^k$, and $\overline{\lambda}$ be its dual. Then
\begin{enumerate}[\rm(i)]
\setlength{\itemsep}{1ex}

\item $\pi^*[Y(\lambda)] = [Y(w_\lambda)]$;

\item $[X(\lambda)]$ is dual to $[Y(\lambda)]$ with respect to the Poincar\'e pairing; 

\item $[X(\lambda)]=w_0^L [Y(\overline{\lambda})]$. 
\end{enumerate}
\end{Prop}

\begin{Rmk}
The Schubert classes $[Y(w)]$ can be represented by the double Schubert polynomials $\mathfrak{S}_w(x, t)$  due to  Lascoux and Sch\"utzenberger \cite{LS}, which are defined by
$$\mathfrak{S}_w(x, t)=\partial_{w^{-1}w_0}\prod_{i+j\leq n}(x_i-t_j).$$
%The Schubert polynomial $\mathfrak{S}_w(x, t)$
% represents the Schubert class $[Y(w)]$. 
In the case when $w=w_{\lambda}$ is a Grassmannian permutation, $\mathfrak{S}_{w_\lambda}(x, t)$ reduces to the  double Schur polynomial $s_\lambda(x,t)$, representing the  Schubert class $[Y(\lambda)]$. 
% We denote 
% $$s_{\lambda}(x,t)=\mathfrak{S}_{w_{\lambda}}(x,t).$$
\end{Rmk}

\subsection{Equivariant K-theory and classes of structure/ideal sheaves}
  
Let $X$ be a nonsingular $T$-variety. 
The $T$-equivariant K-group $K_T(X)$  of $X$ is the Grothendieck group of the category of $T$-equivariant coherent sheaves on $X$. 
To be specific, 
$K_T(X)$ is generated by the symbols $[\mathcal{F}]$ for all $T$-equivariant coherent sheaves modulo the relation $[\mathcal{F}]=[\mathcal{F}_1]+[\mathcal{F}_2]$ for all short exact sequences $0\to \mathcal{F}_1\to \mathcal{F}\to \mathcal{F}_2\to 0$ of equivariant coherent sheaves. 
With the assumption that  $X$ is non-singular, there is a well-defined product structure such that $[\mathcal{F}]\otimes [\mathcal{G}]=[\mathcal{F}\otimes \mathcal{G}]$ when $\mathcal{F}$ is flat. For a reference of the equivariant K-theory, see for example \cite{CG}.

For any $T$-equivariant vector bundle $\mathcal{V}$ of rank $k$ over $X$, we denote the $r$-th \emph{equivariant Chern class} by $c_r(\mathcal{V})\in K_T(X)$, in the sense of algebraic cobordism \cite{LM}. 
They are characterized by
$$\sum_{r=0}^k q^rc_r(\mathcal{V}) = \sum_{r=0}^k (1+q)^{k-r}(-q)^r [\mathsf{\Lambda}^r\mathcal{V}^\vee]\in K_T(X)[q],$$
where $q$ is a formal variable. 
For example, for a line bundle $\mathcal{L}$, the first Chern class is 
$c_1(\mathcal{L})=1-[\mathcal{L}]^{-1}$, see \cite[Example 1.1.5]{LM}. 
Note that $c_r(\mathcal{V})$ satisfies the 
Whitney product formula \cite[Prop 4.1.15 (3)]{LM}
and the splitting principle \cite[Remark 4.1.12]{LM}. Moreover, under the Chern character, the lowest degree component of $c_r(\mathcal{V})$ is the usual Chern class in cohomology. 

% Example 1.1.5.

% For any $T$-subvariety $Y$, the structure sheaf $\mathcal{O}_Y$ is $T$-equivariant, thus defines a class $[\mathcal{O}_Y]\in K_T(X)$.

Let 
\begin{align*}
t_i & = c_1(\mathbb{C}_{-z_i})
=1-[\mathbb{C}_{z_i}]\in K_T(\pt).
% \\\tau_i & = [\mathbb{C}_{-z_i}] \in K_T(\pt).
\end{align*}
As before, $\mathbb{C}_{-z_i}$ is the bundle corresponding to the character
$$T\ni\operatorname{diag}(z_1,\ldots,z_n)\longmapsto z_i^{-1}\in \mathbb{C}^\times. $$
% Note that $t_i$ and $\tau_i$ are related by $t_i=1-\tau_i^{-1}$. 
% Then $K_T(\pt)$ is isomorphic to the ring of Laurent polynomials in $\tau_1,\ldots, \tau_n$: 
% \begin{align*}
% K_T(\pt)
% =\mathbb{Q}[\tau_1^{\pm1},\ldots,\tau_n^{\pm 1}].
% \end{align*}
Then $K_T(\pt)$ is isomorphic to 
\begin{align*}
K_T(\pt)
=\mathbb{Q}\big[t_1,\tfrac{-t_1}{1-t_1},\ldots,t_n,\tfrac{-t_n}{1-t_n}\big].
\end{align*}

We turn to the case when $X$ is $\Gr(k,n)=G/P$ or $\Fl(n)=G/B$. 
Similar to   the equivariant cohomology,  consider
 the $r$-th equivariant Chern class of $\mathcal{V}^\vee$
\begin{align}\label{UHGR}
 c_r(\mathcal{V}^\vee)\in K_T(G/P),
% \\
% \varepsilon_r &= [\mathsf{\Lambda}^r \mathcal{V}^\vee]\in K_T(G/P).
\end{align}
where  $\mathcal{V}$ is the tautological bundle over $G/P$. 
For $i=1,\ldots,n$, write
\begin{align}\label{X-12345}
x_i &= c_1\big((\mathcal{V}_i/\mathcal{V}_{i-1})^\vee\big)
=1-[\mathcal{V}_i/\mathcal{V}_{i-1}]\in K_T(G/B).
% \\\chi_i &= [(\mathcal{V}_i/\mathcal{V}_{i-1})^\vee]\in K_T(G/B).
\end{align}
% Note that %for K-theory $x_i$ and $\chi_i$ are related by 
% $x_i=1-\chi_i^{-1}$. 
Denote by
\begin{align}\label{eq:pullK}
\pi^*\colon K_T(G/P)\longrightarrow
% \mathop{\lhook\joinrel\xrightarrow{\quad}} 
K_T(G/B)
\end{align}
the algebra homomorphism induced by \eqref{eq:pi:GB->GP}.  
By the Whitney product formula, we have
\begin{align*}
    \pi^*c_r(\mathcal{V}^\vee) & =e_r(x_1,\ldots,x_k)\in K_T(G/B).
    % \\
    % \pi^*\epsilon_r & =e_r(\chi_1,\ldots,\chi_k)\in K_T(G/B).
\end{align*}

\begin{Th}[\cite{KK}]\label{th:KBorelinj} We have 
\begin{enumerate}[\rm(i)]
\setlength{\itemsep}{1ex}

    \item The algebra $K_T(G/B)$ is generated by $x_1,\ldots,x_n$ over $K_T(\pt)$; 

    \item The algebra $K_T(G/P)$ is generated by $c_1(\mathcal{V}^\vee),\ldots,c_k(\mathcal{V}^\vee)$ over $K_T(\pt)$;
    
    \item The map $\pi^*$ defined in \eqref{eq:pullK} is an injective algebra homomorphism. 
\end{enumerate}
\end{Th}

In equivariant K-theory, we  consider the classes of structure sheaves 
\begin{align*}
[\mathcal{O}_{X(w)}],\,[\mathcal{O}_{Y(w)}]\,
& \in K_T(G/B),\qquad w\in \S_n,\\
[\mathcal{O}_{X(\lambda)}],\,[\mathcal{O}_{Y(\lambda)}]\,
& \in K_T(G/P),\qquad \lambda\subseteq (n-k)^k.
\end{align*}
Denote  
$$\mathcal{I}_{\partial X(w)}=\mathcal{O}_{X(w)}(-\partial X(w))
\subseteq \mathcal{O}_{X(w)},$$ 
which is the pushforward of the ideal sheaf of $\partial X(w)=X(w)\setminus X(w)^\circ$. 
In the  same manner, we can define $\mathcal{I}_{\partial Y(w)}$, $\mathcal{I}_{\partial X(\lambda)}$, and $\mathcal{I}_{\partial Y(\lambda)}$. Their corresponding classes are
\begin{align*}
[\mathcal{I}_{\partial X(w)}],\,[\mathcal{I}_{\partial Y(w)}] 
& \in K_T(G/B),\qquad w\in \S_n,\\
[\mathcal{I}_{\partial X(\lambda)}],\,[\mathcal{I}_{\partial Y(\lambda)}]
& \in K_T(G/P),\qquad \lambda\subseteq (n-k)^k.
\end{align*}

For $1\leq i\leq n-1$, 
  the \emph{Demazure operators} on $K_T(G/B)$\cite{Demazure}  are
\begin{align}
\pi_i
=\frac{(1-x_{i+1})\operatorname{id}-(1-x_i)s_i}{x_i-x_{i+1}}, \qquad 
\widehat{\pi}_i
=(1-x_{i})\frac{\operatorname{id}-s_i}{x_i-x_{i+1}}=\pi_i-\operatorname{id}.\label{pihat}
\end{align}
% Recall that $x_i=1-\chi_i^{-1}$.
Note that $\pi_i$'s are geometrically induced by push-pull, see \cite{KK}.
These operators satisfy 
$\pi_i^2 = \pi_i$, $\widehat{\pi}_i^2 = -\widehat{\pi}_i $, as well as the braid relations.
Thus $\pi_w$ and $\widehat{\pi}_w$ for   $w\in \S_n$ are well defined. 
We similarly denote 
$$\cdot |_e: K_T(G/B)\longrightarrow K_T(e)\cong K_T(\pt), \qquad 
x_i|_e=t_i,\qquad t_i|_e=t_i.$$

\begin{Prop}[\cite{GK}]\label{SchRep2}
We have
\begin{enumerate}[\rm(i)]
\setlength{\itemsep}{1ex}

% \item
% {\color{red}$[\mathcal{O}_{X(w)}], [\mathcal{O}_{Y(w)}]$?}

\item $\big\{[\mathcal{I}_{\partial Y(w)}]\colon w\in \S_n\big\}$ forms a basis of $K_T(G/B)$ over $K_T(\pt)$; 

\item $[\mathcal{I}_{\partial Y(w)}]= \widehat{\pi}_{w^{-1}w_0}[\mathcal{I}_{\partial Y(w_0)}]$; 

\item $[\mathcal{I}_{\partial Y(w)}]\big|_e\neq 0$ if and only if $w=\operatorname{id}$. 
% {\color{red} need the definition of restriction map?}
\end{enumerate}
\end{Prop}

We now consider the case of $\Gr(k,n)=G/P$. 
Recall   the \emph{Poincar\'e pairing} 
$$\left<\alpha,\beta\right>=\chi\big(G/P,\alpha\otimes \beta) \in K_T(\pt)$$
for $\alpha,\beta\in K_T(G/P)$, where $\chi$ is induced by the pushforward to a point. 
Recall that the left Weyl group   action  $w_0^L$ of    $w_0$  on $K_T(G/P)$ satisfies
\begin{equation}\label{wyleaCT}
w_0^Lc_r(\mathcal{V}^\vee)=c_r(\mathcal{V}^\vee) \quad (1\leq r\leq k),\qquad 
w_0^Lt_i=t_{n+1-i}\quad (1\leq i\leq n).
\end{equation}

\begin{Prop}[\cite{GK}]\label{SchRep2:Gr}
Let  ${\lambda}\subseteq (n-k)^k$, and $\overline{\lambda}$ be its dual. Then
\begin{enumerate}[\rm(i)]
\setlength{\itemsep}{1ex}

\item $\pi^*[\mathcal{I}_{\partial Y(\lambda)}] = \displaystyle\sum_{\begin{subarray}{c}
w\in \S_n\\
\Gr(w)=\lambda
\end{subarray}} [\mathcal{I}_{\partial Y(w)}]$ and $\pi^*[\mathcal{O}_{Y(\lambda)}]=[\mathcal{O}_{Y(w_{\lambda})}]$; 

\item $[\mathcal{I}_{\partial X(\lambda)}]$ is dual to $[\mathcal{O}_{Y(\lambda)}]$ with respect to the Poincar\'e pairing; 

\item $[\mathcal{I}_{\partial Y(\lambda)}]$ is dual to $[\mathcal{O}_{X(\lambda)}]$ with respect to the Poincar\'e pairing; 

\item $[\mathcal{I}_{\partial Y(\lambda)}]=w_0^L
[\mathcal{I}_{\partial X(\overline{\lambda})}]$;

\item $[\mathcal{O}_{Y(\lambda)}]=w_0^L [\mathcal{O}_{X(\overline{\lambda})}]$.
\end{enumerate}
\end{Prop}

The above proposition  can be illustrated by  the following diagram:
$$\xymatrix@!C=6pc{
[\mathcal{I}_{\partial Y(\lambda)}] 
\ar@{<->}[r]^{\text{Weyl action}}
\ar@{<->}[dr]|{\phantom{\rule[-0.25pc]{0pc}{1pc}
\text{Poincar\'e dual}}}& \ar@{<->}[dl]|{\rule[-0.25pc]{0pc}{1pc}\text{Poincar\'e dual}}
[\mathcal{I}_{\partial X(\lambda)}]\\
[\mathcal{O}_{Y(\lambda)}] 
\ar@{<->}[r]_{\text{Weyl action}}& 
[\mathcal{O}_{X(\lambda)}]
}$$

\begin{Rmk}\label{rmk:GrothenConv}
The classes of structure sheaves $[\mathcal{O}_{Y(w)}]$ can be represented by the double Grothendieck polynomials $\mathfrak{G}_w(x,t)$ also defined  by  Lascoux and Sch\"utzenberger \cite{LS-2}: for   $w\in \S_n$, 
$$\mathfrak{G}_w(x,t)
% =\pi_{w^{-1}w_0}\prod_{i+j\leq n}\bigg(1-\frac{\tau_i}{\chi_i}\bigg)
=\pi_{w^{-1}w_0}\prod_{i+j\leq n}\frac{x_i-t_j}{1-t_i}.$$ 
When $w=w_{\lambda}$ is a Grassmannian permutation, $\mathfrak{G}_{w_\lambda}(x, t)$ represents the class $[\mathcal{O}_{Y(\lambda)}]$. 
% If we understand each $\mathfrak{G}_w(x,y)$ as a power series in $x$'s and $t_i$'s, then   
% $$\mathfrak{G}_w(x,t)=\mathfrak{S}_{w}(x,t)+(\text{higher degree terms})\in \mathbb{Z}[\![x_1,x_2,\ldots,t_1,t_2,\ldots]\!].$$
% % the lowest degree component of $\mathfrak{G}_w(x,t)$ is the Schubert polynomial $\mathfrak{S}_{w}(x,t)$. 
\end{Rmk}

\subsection{Chern--Schwartz--MacPherson classes and Segre--MacPherson class}
For any algebraic variety $X$ over $\mathbb{C}$, let $\Fun(X)$ denote the space of constructible functions over $X$ with values in $\mathbb{Z}$. For any proper morphism $f\colon X\to Y$, we define 
$f_*\colon \Fun(X)\to \Fun(Y)$ such that for any $\varphi\in \Fun(X)$
% {\color{red} what is $\varphi$?}
$$(f_*\varphi)(y)=\sum_{q\in \operatorname{im}\varphi}\chi_c\big(f^{-1}(y)\cap \varphi^{-1}(q)\big)q,$$
where $\chi_c$ stands for the Euler characteristic of compact support. It was conjectured by Grothendieck and Deligne and confirmed by MacPherson \cite{MacPherson} that there is a natural transformation
$$c_*\colon \Fun(-)\longrightarrow H_\bullet(-)$$
from the functor of rational constructible functions to Borel--Moore homology such that 
when $X$ is nonsingular, the constant constructible function $\mathbf{1}_X$ is sent to the total Chern class of tangent bundle, namely,
$$c_*(\mathbf{1}_X) = c(\mathscr{T}_X)\cap [X]. $$
The class $c_*(\mathbf{1}_X)$ (possibly for singular $X$)  was later shown to coincide with a class defined earlier by   Schwartz \cite{Sch65a,Sch65b}.
This transformation is extended to the equivariant setting by Ohmoto \cite{Ohmoto}. 

Assume further that $X$ is a complete nonsingular $T$-variety. We can identify the Borel--Moore homology with the cohomology through Poincar\'e duality. 
Denote by
$$\mathbb{H}_T(-) = H^\bullet_T(-)_{\mathbb{F}}$$
the localized equivariant cohomology of $X$, where $\mathbb{F}=\mathbb{Q}(t_1,\ldots,t_n)$ is the fraction field of $H^\bullet_T(\pt)$.
For a $T$-equivariant constructible subset $W$, we denote the \emph{Chern--Schwartz--MacPherson (CSM) class} and 
\emph{Segre--MacPherson class} by
\begin{align*}
\CSM(W)= 
c_*(\mathbf{1}_W) \in \mathbb{H}_T(X),\qquad 
\SM(W)= \frac{c_*(W)}{c(\mathscr{T}_{X})}\in \mathbb{H}_T(X),
\end{align*}
where $\mathbf{1}_W$ is the characteristic function of $W$, and $c(\mathscr{T}_{X})$ is the total Chern class of the tangent bundle of $X$. 
% Note that here we identify cohomology with its homology using Poincar\'e duality. 

Restricting to the case $X=G/B$ or $G/P$, one can define 
\begin{align*}
\CSM(X(w)^\circ),\,\CSM(Y(w)^\circ),\,
\SM(X(w)^\circ),\,\SM(Y(w)^\circ)
&\in \mathbb{H}_T(G/B),\\[5pt]
\CSM(X(\lambda)^\circ),\,\CSM(Y(\lambda)^\circ),\,
\SM(X(\lambda)^\circ),\,\SM(Y(\lambda)^\circ)
&\in \mathbb{H}_T(G/P).
\end{align*}
For $1\leq i\leq n-1$, 
 the \emph{Demazure--Lusztig operator} on $\mathbb{H}_T(G/B)$ is 
\begin{align}\label{tihat}
\mathcal{T}_i=\partial_i-s_i.
\end{align}
They  satisfy $\mathcal{T}_i^2 = \operatorname{id}$ and the braid relations, and so  $\mathcal{T}_w$ for   $w\in \S_n$ is well defined.

\begin{Prop}[{\cite{AM,AMSS17}}]\label{OPUY}
We have
\begin{enumerate}[\rm(i)]
\setlength{\itemsep}{1ex}
\item $\{\CSM(Y(w)^\circ)\colon w\in \S_n\}$ is a basis of $\mathbb{H}_T(G/B)$ over $\mathbb{H}_T(\pt)$; 
\item $\CSM(Y(w)^\circ)= \mathcal{T}_{w^{-1}w_0}\CSM(Y(w_0)^\circ)$; 
\item $\CSM(Y(w)^\circ)\big|_e\neq 0$ if and only if $w=\operatorname{id}$.
\end{enumerate}
\end{Prop}

Let us denote 
\begin{align}
\pi_*\colon H^\bullet_T(G/B) \longrightarrow H^\bullet_T(G/P)
\end{align}
the pushforward induced by \eqref{eq:pi:GB->GP}.

\begin{Prop}[{\cite{AMSS17}}]\label{SchRep3}
Let  ${\lambda}\subseteq (n-k)^k$, and $\overline{\lambda}$ be its dual. Then
\begin{enumerate}[\rm(i)]
\setlength{\itemsep}{1ex}
\item $\pi_*\CSM(Y(w)^\circ) = \CSM(Y(\Gr(w))^\circ)$; 
 % \item $\pi^*\SM(Y(\lambda)^\circ) = \displaystyle\sum_{w\in \S_n\atop \Gr(w)=\lambda} \SM(Y(w)^\circ)$;
\item $\CSM(X(\lambda)^\circ)$ is dual to $\SM(Y(\lambda)^\circ)$ with respect to the Poincar\'e pairing;
%\item $\CSM(Y(\lambda)^\circ)$ is dual to $\SM(X(\lambda)^\circ)$ under the Poincar\'e pairing;
\item $\CSM(X(\lambda)^\circ)=w_0^L
\CSM(Y(\overline{\lambda})^\circ)$;
\item $\SM(X(\lambda)^\circ)=w_0^L
\SM(Y(\overline{\lambda})^\circ)$.
\end{enumerate}
\end{Prop}

The relations in Proposition \ref{SchRep3} can be summarized as follows:
$$\xymatrix@!C=6pc{
\CSM(Y(\lambda)^\circ)
\ar@{<->}[r]^{\text{Weyl action}}
\ar@{<->}[dr]|{\phantom{\rule[-0.25pc]{0pc}{1pc}
\text{Poincar\'e dual}}}& \ar@{<->}[dl]|{\rule[-0.25pc]{0pc}{1pc}\text{Poincar\'e dual}}
\CSM(X(\lambda)^\circ)\\
\SM(Y(\lambda)^\circ)
\ar@{<->}[r]_{\text{Weyl action}}& 
\SM(X(\lambda)^\circ)
}$$

% \begin{Rmk}
% As pointed out in \cite{AM}, 
% \begin{align*}
% \CSM(Y(w)^\circ) & =[Y(w)]+(\text{higher degree classes})\in H_T(G/B) {\color{red}or \mathbb{H}},\\
% \CSM(Y(\lambda)^\circ) & =[Y(\lambda)]+(\text{higher degree classes})\in H_T(G/P) {\color{red}or \mathbb{H}}.
% \end{align*}
% Similar relations  hold for {\color{red}Segre--MacPherson}   classes. 
% \end{Rmk}

% \begin{Prop}[\cite{AM}]
% The lowest degree component of
% $\CSM(Y(w)^\circ)$ or $\SM(Y(w)^\circ)$ is $[Y(w)]$. 
% \end{Prop}

\subsection{Motivic Chern classes and Segre motivic classes }\label{MCCCC}

Let $X$ be an algebraic variety. Define the \emph{Grothendieck  group}
$G_0(X)$ to be the group generated by $[Z\stackrel{\varphi}\to X]$ for all morphisms $\varphi\colon Z\to X$ modulo the relation 
$$[Z\stackrel{\varphi}\to X]=[U\stackrel{\varphi}\to X]+[Z\setminus U\stackrel{\varphi}\to X]$$
for an open subset $U\subseteq Z$. 
For a proper morphism $f\colon X\to Y$,  define $f_*:G_0(X)\to G_0(Y)$ by 
$$f_*[Z\stackrel{\varphi}\to X]=[Z\stackrel{\varphi}\to X\stackrel{f}\to Y].$$
By \cite{BSY}, there exists a unique natural transformation
$$\MC_y\colon G_0(-)\longrightarrow K(-)[y]$$
such that when $X$ is nonsingular, the identity morphism is sent to the total $\lambda$-class of cotangent bundle
$$\MC_y([X\stackrel{\operatorname{id}}\to X])=\lambda_y(\mathscr{T}^\vee_X):=\sum_{r\geq 0}y^r[\mathsf{\Lambda}^r\mathscr{T}^\vee_X]. $$
Here $y$ is a formal variable.
The equivariant version was established in \cite{FRW,AMSS19}. 

Assume that $X$ is a complete nonsingular $T$-variety. 
Denote 
$$\mathbb{K}_T(X)=K_T(X)_{\mathbb{F}},$$
where $\mathbb{F}=\mathbb{Q}(t_1,\ldots,t_n,y)$ is the fraction field of $K_T(\pt)[y]$.
Let $W$ be an equi-dimensional $T$-equivariant constructible subset. 
The equivariant  \emph{motivic Chern class} and 
\emph{Segre motivic class} of $W$ are 
\begin{align*}
\MC_y(W) & =\MC_y([W\hookrightarrow X])\in \mathbb{K}_T(X),\\[5pt]
\SMC_y(W)&= (-y)^{\dim W}\frac{\mathcal{D}(\MC_y(W))}{\lambda_y(\mathscr{T}^\vee_X)}\in \mathbb{K}_T(X),
\end{align*}
where $\mathcal{D}$ is induced by the Serre duality functor $R\mathscr{H}{\it \!om}(-,\omega_X[\dim X])$.

Restricting to the case $X=G/B$ or $G/P$, we can consider equivariant motivic Chern classes and Segre motivic classes of Schubert cells:
\begin{align*}
\MC_y(X(w)^\circ),\MC_y(Y(w)^\circ),
% \qquad\qquad\nonumber\\
\SMC_y(X(w)^\circ),\SMC_y(Y(w)^\circ)
&\in \mathbb{K}_T(G/B),\\[5pt]
\MC_y(X(\lambda)^\circ),\MC_y(Y(\lambda)^\circ),
% \qquad\qquad\nonumber\\
\SMC_y(X(\lambda)^\circ),\SMC_y(Y(\lambda)^\circ)
&\in \mathbb{K}_T(G/P).
\end{align*}
For $1\leq i\leq n-1$,   the \emph{Demazure--Lusztig operator} on $\mathbb{K}_T(G/B)$ is
\begin{align}\label{scriptti}
\mathfrak{T}_i
% =\big(1+y\tfrac{\chi_{i+1}}{\chi_i}\big)\pi_i-\operatorname{id}
=\big(1+y\tfrac{1-x_{i}}{1-x_{i+1}}\big)\pi_i-\operatorname{id}.
\end{align}
The operators satisfy 
$(\mathfrak{T}_i+1)(\mathfrak{T}_i+y)=0$ and the braid relations, and thus $\mathfrak{T}_w$ for  $w\in \S_n$ is well defined. 

\begin{Prop}[{\cite{AMSS19}}]\label{SchRep4}
We have
\begin{enumerate}[\rm(i)]
\setlength{\itemsep}{1ex}
\item $\{\MC_y(Y(w)^\circ)\colon w\in \S_n\}$ forms  a basis of $\mathbb{K}_T(G/B)$ over $\mathbb{K}_T(\pt)$; 
\item $\MC_y(Y(w)^\circ)= \mathfrak{T}_{w^{-1}w_0}\MC_y(Y(w_0)^\circ)$; 
\item $\MC_y(Y(w)^\circ)\big|_e\neq 0$ if and only if $w=\operatorname{id}$.
\end{enumerate}
\end{Prop}

Similarly, denote by
\begin{align}
\pi_*\colon K_T(G/B) \longrightarrow K_T(G/P)
\end{align}
the pushforward induced by \eqref{eq:pi:GB->GP}.

\begin{Prop}[{\cite{AMSS19,MNS}}]\label{SchRep4:Gr}

We have
\begin{enumerate}[\rm(i)]
\setlength{\itemsep}{1ex}
\item $\pi_*\MC_y(Y(w)^\circ) = (-y)^{\binom{k}{2}+\binom{n-k}{2}-\ell(w)+|\Gr(w)|} \MC_y(Y(\Gr(w))^\circ)$;

% \item 
% % $\pi_*\left(\sum_{\Gr(w)=\lambda}\MC_y(Y(w)^\circ)\right)
% % =\left(\sum_{v\in \S_k\times \S_{n-k}}(-y)^{\ell(v)}\right)\MC_y(Y(\lambda)^\circ)$. 
% $\pi^*\SMC_y(Y(\lambda)^\circ) = 
% \displaystyle\sum_{w\in \S_n\atop \Gr(w)=\lambda} \SMC_y(Y(w)^\circ);
% $

\item $\MC_y(X(\lambda)^\circ)$ is dual to $\SMC_y(Y(\lambda)^\circ)$ with respect to the Poincar\'e pairing;
% \item $\MC_y(Y(\lambda)^\circ)$ is dual to $\SMC_y(X(\lambda)^\circ)$ under the Poincar\'e pairing;
\item $\MC_y(X(\lambda)^\circ)=w_0^L \MC_y(Y(\overline{\lambda})^\circ)$;
\item $\SMC_y(X(\lambda)^\circ)=w_0^L \SMC_y(Y(\overline{\lambda})^\circ)$.
\end{enumerate}
\end{Prop}

We use the following diagram to illustrate Proposition \ref{SchRep4:Gr}:
$$\xymatrix@!C=8pc{
\MC_y(Y(\lambda)^\circ)
\ar@{<->}[r]^{\text{Weyl action}}
\ar@{<->}[dr]|{\phantom{\rule[-0.25pc]{0pc}{1pc}
\text{Poincar\'e dual}}}& \ar@{<->}[dl]|{\rule[-0.25pc]{0pc}{1pc}\text{Poincar\'e dual}}
\MC_y(X(\lambda)^\circ)\\
\SMC_y(Y(\lambda)^\circ)
\ar@{<->}[r]_{\text{Weyl action}}& 
\SMC_y(X(\lambda)^\circ)
}$$

\begin{Rmk}\label{RRR-N}
By \cite{AMSS19} and \cite{MNS}, if we specialize $y=0$, then we have 
$$
\MC_y(Y(w)^\circ)|_{y=0} = [\mathcal{I}_{\partial Y(w)}],\qquad 
\SMC_y(Y(w)^\circ)|_{y=0} = [\mathcal{O}_{Y(w)}],$$
and
$$
\MC_y(Y(\lambda)^\circ)|_{y=0} = [\mathcal{I}_{\partial Y(\lambda)}],\qquad 
\SMC_y(Y(\lambda)^\circ)|_{y=0} = [\mathcal{O}_{Y(\lambda)}].$$
% By \cite{AMSS22}, $\CSM(Y(w)^\circ)$ 
\end{Rmk}

% \begin{Prop}[{\cite{AMSS19},\cite{MNS}}]
% We have 
% $$
% \MC_y(Y(w)^\circ)|_{y=0} = [\mathcal{I}_{\partial Y(w)}],\qquad 
% \SMC_y(Y(w)^\circ)|_{y=0} = [\mathcal{O}_{Y(w)}].$$
% \end{Prop}

\section{Affine Hecke algebras and Schubert representations}\label{sec3}

In this section, we introduce an affine Hecke algebra $\Haff_n$,  and define the  Schubert representation of $\Haff_n$. We illustrate how motivic Chern classes, as well as three other classes in Section \ref{sec:geom}, are incorporated naturally in this setting. 
%Some properties are also investigated.  

Let
$\mathcal{A}=\mathbb{Q}[p,q]$ be the ring   of polynomials over $\mathbb{Q}$ in two parameters  $p$ and $q$. 
The (type $A$) \emph{Hecke algebra} $\H_n$ over $\mathcal{A}$  is the associative algebra with generators  $\T_i$ ($1\leq i \leq n-1$), subject to the relations 
    \begin{align}
    (\T_i+p)&(\T_i-q)=0,\label{HeckeRelation1}\\[5pt]
\T_i\T_j&=\T_j\T_i,\quad |i-j|>1,\label{BraidRelation1}\\[5pt]
\T_i\T_{i+1}\T_i&=\T_{i+1}\T_i\T_{i+1}. \label{BraidRelation2} 
    \end{align}
Due to the braid relations in \eqref{BraidRelation1} and \eqref{BraidRelation2}, one can define $\T_w=\T_{{i_1}}\cdots \T_{{i_{\ell(w)}}}$ 
for $w\in \S_n$,
where $w=s_{i_1}\cdots  s_{i_{\ell(w)}}$ is any 
reduced decomposition of $w$.
Then $\H_n$ is a free $\mathcal{A}$-module with basis
$\{\T_w\colon w\in \S_n\}$. 
The relation in \eqref{HeckeRelation1}
can be rewritten as 
$\T_i^2=-(p-q)\T_i+pq$. 
The basis elements satisfy the following multiplication rule
\begin{align*}
\T_i\T_w=\left\{
  \begin{array}{ll}
    \T_{s_iw}, & \hbox{if $\ell(s_iw)>\ell(w)$,} \\[5pt]
    -(p-q)\T_w+pq\T_{s_iw}, & \hbox{if $\ell(s_iw)<\ell(w)$.}
  \end{array}
\right.
\end{align*}
Denote $\bT_i=\T_i+p-q$. 
It is direct  to check that $\bT_i$'s satisfy
\begin{align}
% \label{quadraticRelation}
%\label{barHeckeRelation1}
(\bT_i+q)&(\bT_i-p)=0,\\[5pt]
\label{barBraidRelation1}
\bT_i\bT_j&=\bT_j\bT_i,\quad |i-j|>1,\\[5pt]
\label{barBraidRelation2}
\bT_i\bT_{i+1}\bT_i&=\bT_{i+1}\bT_i\bT_{i+1}. 
\end{align}
Thus, we may write  
$\bT_w=\bT_{{i_1}}\cdots \bT_{{i_{\ell(w)}}}$
for any reduced decomposition  $w=s_{i_1}\cdots  s_{i_{\ell(w)}}$ of $w\in \S_n$.  Then $\{\bT_w\colon w\in \S_n\}$ is also an $\mathcal{A}$-basis of $H_n$.

\begin{Def}\label{def:affineHecke}
% For a parameter $\hbar$, 
Let $\A=\mathcal{A}[\hbar]$. 
The \emph{affine Hecke algebra}  $\Haff_n$ over $\A$ is the algebra generated by $\T_1,\ldots, \T_{n-1}$
and  $x_1,\ldots, x_n$,  with   relations  \eqref{HeckeRelation1}--\eqref{BraidRelation2} together with
\begin{align}
x_ix_j&=x_jx_i,\\[5pt]
\label{Lebniz1}
\T_ix_j&=x_j\T_i,\quad j\neq i,i+1, \\[5pt]
\label{Lebniz2}
\T_ix_i &= x_{i+1}\T_i+(\hbar-(p-q) x_{i}),\\[5pt]
\label{Lebniz3}
\T_ix_{i+1} &= x_{i}\T_i-(\hbar-(p-q) x_{i}).
\end{align}
\end{Def}

% If we denote $X_i=\hbar-(p-q)x_i$, 
% then it is not hard to check $\T_iX_j=X_j\T_i$ for $j\neq i,i+1$ and $\T_i X_{i+1}\T_i=pq X_i$. It specializes to the Bernstein relation of type $A$, see \cite{Lusztig1}. 

% \begin{Rmk}
%     If we set $p=q=1$, we get the graded affine Hecke algebra. If we set $p=1$ and $\hbar=0$, we get the affine Hecke algebra, see \cite{Lusztig2}. {\color{red} (CS: for the affine Hecke, we need to invert $x_i$. for the graded affine Hecke, we only need to work with the polynomials. How to get a compatible version of this definition?)}
% \end{Rmk}

Notice that $\Haff_n$ can be viewed as a free
$\A$-module
$$\Haff_n=\A[x_1,\ldots,x_n]\otimes_{\mathcal{A}}\H_n,$$
including  $\A[x_1,\ldots,x_n]$ and $\H_{n}$ as subalgebras.  
For $\alpha=(\alpha_1,\ldots, \alpha_n)\in \mathbb{Z}_{\geq 0}^n$, write $x^\alpha=x_1^{\alpha_1}\cdots x_n^{\alpha_n}$. Then both 
\[
\big\{x^\alpha \T_w\colon  w\in \S_n, \,\alpha\in \mathbb{Z}_{\geq 0}^n\big\}
\quad \text{and}\quad  
\big\{x^\alpha \bT_w\colon  w\in \S_n, \,\alpha\in \mathbb{Z}_{\geq 0}^n\big\}\]
are $\A$-basis of $\Haff_n$.

% {\color{red}
% By\eqref{Lebniz2} and \eqref{Lebniz3}, it is routine to check that 
% \begin{equation}\label{EQQQ-1}
% \bT_i x_i=x_{i+1}\T_i+\hbar \quad \text{and}\quad  x_i\bT_i=\T_ix_{i+1}+\hbar,
% \end{equation}
% and  that 
% \begin{align}
% \bT_i(\hbar-(p-q)x_i)=(\hbar-(p-q)x_{i+1})\T_i,\label{EQQQ-2}\\[5pt]
% (\hbar-(p-q)x_i)\bT_i=\T_i(\hbar-(p-q)x_{i+1}).\label{EQQQ-3}
% \end{align}
% We remark that relations \eqref{EQQQ-2} and \eqref{EQQQ-3}
% are essentially the relations for the usual affine Hecke algebra [REF] after denoting $X_i=\hbar-(p-q)x_i$. 
% }{\color{red} CS: not sure what is the exact meaning of this.}

\begin{Rmk}
Here are some remarks about $\H_n$  and $\Haff_n$.
\begin{itemize}
    \item Our definition of    $\H_n$   is a special case of generic Hecke algebras in type $A$, see \cite[\textsection 7.1]{Hump}, {\it loc. cit}, $a=q-p$ and $b=pq$. 
    The classical Hecke algebra \cite{IM95,KL} is the case with $p=1$  and with $q$ inverted. 
    Another popular convention is the case when $q=-v$ and $p=-v^{-1}$ \cite{Seo}. 
    
    \item The basis element $\bT_w$ is the image of $\T_w$ under Goldman involution introduced by Iwahori \cite{Iwahori}. 
    For $p=-v^{-1}$ and $q=-v$, it coincides with the image under the bar involution in  \cite{Seo}. For  $p=1$,   $\bT_w$ is the same as $q_w\bT_w$ in \cite{KL}. 

    \item The affine Hekce algebra   $\Haff_n$ in Definition \ref{def:affineHecke} is a generalization of graded affine Hecke algebras defined by Drinfeld \cite{Drin} (extended to any type by Lusztig \cite{Lusztig2}). 
    If we denote $X_i=\hbar-(p-q)x_i$, then it can be checked that
    $\T_iX_j=X_j\T_i$ for $j\neq i,i+1$ and 
    $\T_i X_{i+1}\T_i=pq X_i$, specializing the Bernstein relations of type $A$ \cite{Lusztig1}. 
\end{itemize}
\end{Rmk}

%Actually, $\Haff_n$ is a free $\A$-module. 
% and any element $P\in \Haff_n$ can be uniquely written as 
% $$P=\sum_{w\in \S_n} P_w(x_1,\ldots,x_n) \bT_w$$
% for some $P_w\in \A[x_1,\ldots,x_n]$. 
% We call such a form by the \emph{left-justified form} of $P$. 
%We remark that $\H_n$ is a subalgebra of $\Haff_n$, and we can naturally identify $\H_m$ (resp. $\Haff_m$) as a subalgebra of $\H_n$ (resp. $\Haff_n$) for $m\leq n$. Here is the diagram of the inclusion
%$$\xymatrix{
%\H_1\ar@{..}[d]|{\shortmid\cap}\ar@{..}[r]|{\subseteq}& 
%\H_2\ar@{..}[d]|{\shortmid\cap}\ar@{..}[r]|{\subseteq}&
%\cdots\ar@{..}[d]|{\vdots}\ar@{..}[r]|{\subseteq}&
%\H_n\ar@{..}[d]|{\shortmid\cap}
%\\
%\Haff_1\ar@{..}[r]|{\subseteq}& 
%\Haff_2\ar@{..}[r]|{\subseteq}& 
%\cdots\ar@{..}[r]|{\subseteq}& 
%\Haff_n}$$

Taking specific values of the parameters $p, q, \hbar$, the relations satisfied by $T_i$ and $x_j$ in \eqref{HeckeRelation1}--\eqref{BraidRelation2} as well as in  \eqref{Lebniz1}--\eqref{Lebniz3} are compatible with the relations satisfied by the operators appearing in  Section \ref{sec:geom}.  The details are listed in the following proposition. 

\begin{Prop}\label{THaffineHeckeEg} 
Via direct computations, one can check that
\begin{enumerate}[\rm(i)]
\setlength{\itemsep}{1ex}

\item The operator  $\partial_i$   in \eqref{partiali} corresponds to $(p,q,\hbar)=(0,0,1)$; 

\item 
The operator $\widehat{\pi}_i$ in \eqref{pihat}   corresponds to $(p,q,\hbar)=(1,0,1)$;

% Generally, $\widehat{\pi}^\beta_i$ and $x_j$ satisfy the relations of $\Haff_n$ with $(p,q,\hbar)=(-\beta,0,1)$;

% \item $\widehat{\pi}_i$ and $\chi_j$ satisfy the relations of $\Haff_n$ with $(p,q,\hbar)=(1,0,0)$;

% {\color{red}
% \item $\pi_i$ and $\chi_j^{-1}$ satisfy the relations of $\Haff_n$ with $(p,q,\hbar)=(0,1,0)$;}

\item The operator $\mathcal{T}_i$   in \eqref{tihat}  corresponds to $(p,q,\hbar)=(1,1,1)$;

% Generally, $\mathcal{T}_i^\hbar$ and $x_j$ satisfy the relations of $\Haff_n$ with $(p,q,\hbar)=(1,1,\hbar)$;

\item The operator $\mathfrak{T}_i$  in \eqref{scriptti} corresponds to $(p,q,\hbar)=(1,-y,1+y)$. 

% \item $\mathfrak{T}_i$ and $\chi_j$ satisfy the relations of $\Haff_n$ with $(p,q,\hbar)=(1,-y,0)$. 

% {\color{red}
% \item $\mathfrak{T}_i$ and $\chi_j^{-1}$ satisfy the relations of $\Haff_n$ with $(p,q,\hbar)=(-y,1,0)$. }

\end{enumerate}
\end{Prop}

% \begin{proof}
% All can be checked via direct computations.
% \end{proof}

% \subsection{Schubert representations}

Motivated by Proposition \ref{THaffineHeckeEg}, we define the Schubert  representation of  $\Haff_n$ as follows. 

\begin{Def}\label{DDD-1}
Let $\mathbb{V}$ be a representation of $\Haff_n$ over  $\mathbb{F}=\mathbb{Q}(t_1,\ldots,t_n)$. 
We say that $\mathbb{V}$ is a \emph{Schubert representation} if 
% \begin{enumerate}[\rm(1)]
\begin{itemize}
\item There exists a distinguished element $\mathcal{Y}_{w_0}\in \mathbb{V}$ ($w_0=n\cdots 21$ is the longest permutation), such that the elements, called \emph{generic Schubert classes}, defined by 
$$\mathcal{Y}_{w}=\T_{w^{-1}w_0}\cdot\mathcal{Y}_{w_0}\in \mathbb{V},\qquad w\in \S_n,$$
form  a basis of $\mathbb{V}$ over $\mathbb{F}$.

\item There exists  an $\mathbb{F}$-linear \emph{evaluation map} $\ev\colon \mathbb{V}\to \mathbb{F}$ such that 
$$\ev(\mathcal{Y}_w)\neq 0\iff w = \operatorname{id}.$$

\item There exist \emph{equivariant parameters} $t_1,\ldots,t_n\in \mathbb{F}$, such that for any $\mathcal{Y}_w\in\mathbb{V}$,
\begin{equation}\label{UYTR}
  \ev(x_i \cdot \mathcal{Y}_w)=t_i\ev(\mathcal{Y}_w).  
\end{equation}
\end{itemize}
% \end{enumerate}
\end{Def}

Combining  Proposition \ref{THaffineHeckeEg} and Propositions \ref{SchRep1}, \ref{SchRep2}, \ref {OPUY} and \ref{SchRep4}, we illustrate  in the following example  concrete Schubert representations that  we need.

\begin{Eg}\label{ThSchubertcls}
Choose 
\[
\mathbb{V} = \mathbb{H}_T(G/B),\quad
\mathbb{F} = \mathbb{H}_T(\pt),\quad 
\ev=-|_{e},
\]
or 
\[
\mathbb{V} = \mathbb{K}_T(G/B),\quad
\mathbb{F} = \mathbb{K}_T(\pt),\quad 
\ev=-|_{e}.
\]
We have the following setups:
\begin{align*}
(p,q,\hbar) &= (0,0,1), 
&& \T_i=\partial_i,
&& \text{$x_i$ in \eqref{X-123}},
&& \mathcal{Y}_{w}=[Y(w)]\in \mathbb{H}_T(G/B),\\[5pt]
(p,q,\hbar) &= (1,0,1),
&& \T_i=\hat{\pi}_i,
&& \text{$x_i$ in \eqref{X-12345}},
&& \mathcal{Y}_{w}=[\mathcal{I}_{\partial Y(w)}]\in \mathbb{K}_T(G/B),\\[5pt]
(p,q,\hbar) &= (1,1,1),
&& \T_i=\mathcal{T}_i,
&& \text{$x_i$ in \eqref{X-123}},
&& \mathcal{Y}_{w}=\CSM(Y(w)^\circ)\in \mathbb{H}_T(G/B),\\[5pt]
(p,q,\hbar) &= (1,-y,1+y),
&& \T_i=\mathfrak{T}_i, 
&& \text{$x_i$ in \eqref{X-12345}},
&& \mathcal{Y}_{w}=\MC_y(Y(w)^\circ)\in \mathbb{K}_T(G/B).
\end{align*}
\end{Eg}

In the remaining of this section, we clarify the equivalence  between the product $ \bT_w f$ and the action $f\cdot \mathcal{Y}_w$ for $f\in \A[x_1,\ldots,x_n]$.
Its Grassmannian analogue will be our main task in Section \ref{sec4}.

\begin{Th}\label{Prodform}
Let $\mathbb{V}$ be a Schubert representation of $\Haff_n$.
For $w\in \S_n$ and $f\in \A[x_1,\ldots,x_n]$, assume that
\begin{equation}\label{00PPO}
    \bT_w f=\sum_{u\in \S_n} P_{u, w}(x_1,\ldots,x_n)\,\bT_u\in \Haff_n,
\end{equation}
where $P_{u,w}\in \A[x_1,\ldots,x_n]$.
Then 
$$f\cdot \mathcal{Y}_u = \sum_{w\in \S_n}P_{u, w}(t_1,\ldots,t_n)\, \mathcal{Y}_w\in \mathbb{V}.$$
\end{Th}

The proof of Theorem \ref{Prodform}
is based on the following property.

\begin{Prop}\label{Reconstruction}
Let $\mathbb{V}$ be a Schubert representation of $\Haff_n$. 
For  $u, v\in \S_n$, 
\begin{equation}\label{EEE-1}
\ev\left(\bT_u\cdot\mathcal{Y}_v\right) = 
\begin{cases}
\ev\left(\mathcal{Y}_{\operatorname{id}}\right), & u= v,\\[5pt]
0, & \text{otherwise}.
\end{cases}
\end{equation}
Moreover, for   $\mathcal{Y}\in \mathbb{V}$, we have
\begin{equation}\label{EEE-2}
\mathcal{Y}=\frac{1}{\ev(\mathcal{Y}_{\operatorname{id}})}\sum_{w\in \S_n} \ev\left(\bT_w \cdot\mathcal{Y}\right)\mathcal{Y}_w.
\end{equation}
\end{Prop}
{ The proof of the above proposition needs a classical identity in Hecke algebras \cite{KL}, and here we include  a proof    for completeness.} 
Denote by
$\epsilon\colon \H_n\to \mathcal{A}$ the operation  of taking the coefficient of $\T_{w_0}$. 

\begin{Lemma}[]\label{LLL-1}
For   $u,v\in \S_n$, 
$$\epsilon(\bT_u\T_v)=
\begin{cases}
1, & uv=w_0,\\[5pt]
0, & \text{otherwise}.
\end{cases}$$
\end{Lemma}

\begin{proof}
We use induction on $\ell(u)$. 
This is clear for $u=\operatorname{id}$. We next consider $\ell(u)\geq 1$. 
By definition, it is easily seen  that for $w\in \S_n$,
\[\bT_w=T_w+\sum_{y<w} a_{y, w} T_y,\]
where $a_{y, w}\in \mathbb{Q}[p,q]$, and $<$ is the Bruhat order.   %{\color{red}(definition  of Bruhat order?) }
So the assertion is true  when  $\ell(u)\leq \ell(w_0)-\ell(v)$.
Now we are left with  the case when $\ell(u)>\ell(w_0)-\ell(v)$. In this case, we have $uv\neq w_0$, and so we need to verify $\epsilon(\bT_u\T_v)=0$.
Notice that for $w\in \S_n$ and $1\leq i\leq n-1$,
\begin{align}\label{NNNN}
\bT_i\T_w=\left\{
  \begin{array}{ll}
    pq\T_{s_iw}, & \hbox{$\ell(s_iw)<\ell(w)$,} \\[5pt]
    (p-q)\T_w+\T_{s_iw}, & \hbox{$\ell(s_iw)>\ell(w)$.}
  \end{array}
\right.
\end{align}
Since $\ell(u)\geq 1$, we may pick an index $i$ such that $\ell(us_i)<\ell(u)$. By \eqref{NNNN},
$$\epsilon(\bT_u\T_v)=
\epsilon(\bT_{us_i}\bT_i\T_{v})
= \begin{cases}
pq\epsilon(\bT_{us_i}\T_{s_iv}), & \ell(s_iv)<\ell(v),\\[5pt]
\epsilon(\bT_{us_i}\T_{s_iv})
+(p-q)\epsilon(\bT_{us_i}\T_{v}), & \ell(s_iv)>\ell(v).
\end{cases}$$
Since  $us_is_iv=uv\neq w_0$, we have $\epsilon(\bT_{us_i}\T_{s_iv})=0$ by induction. To conclude the proof, we need to check that $\epsilon(\bT_{us_i}\T_{v})=0$. 
Suppose to the contrary that $\epsilon(\bT_{us_i}\T_{v})$ is non-zero, that is, $us_iv=w_0$ by induction. 
However, this would lead to
$$\ell(v)=\ell(s_iv)-1=\ell(w_0)-\ell(u)-1,$$
contradicting  the assumption that $\ell(v)>\ell(w_0)-\ell(u)$. 
\end{proof}

\begin{proof}[Proof of Theorem \ref{Reconstruction}]
For $u, v\in \S_n$, it follows from Definition \ref{DDD-1} that
$$\ev(\bT_u\cdot \mathcal{Y}_v) = \ev(\bT_u\T_{v^{-1}w_0}\cdot \mathcal{Y}_{w_0})
=\epsilon(\bT_u\T_{v^{-1}w_0})\ev(\mathcal{Y}_{\operatorname{id}}).$$
By Lemma \ref{LLL-1}, $\epsilon(\bT_u\T_{v^{-1}w_0})$ equals $1$ if $u=v$, and vanishes otherwise. 
This proves \eqref{EEE-1}. 
To verify \eqref{EEE-2}, write 
$\mathcal{Y}=\sum_{w\in \S_n} c_w \mathcal{Y}_w$.
By \eqref{EEE-1},
\[
\ev(\bT_u\cdot \mathcal{Y})=\sum_{w\in \S_n} c_w\ev(\bT_u\cdot \mathcal{Y}_w)=  c_u \ev(\mathcal{Y}_{\operatorname{id}}),
\]
which justifies  \eqref{EEE-2}.
\end{proof}

\begin{proof}[Proof of Theorem \ref{Prodform}]
By \eqref{EEE-2}, we see that 
\[
f\cdot \mathcal{Y}_u 
 = 
\frac{1}{\ev(\mathcal{Y}_{\operatorname{id}})}
\sum_{w\in \S_n}\ev(\bT_w f\cdot \mathcal{Y}_u)\mathcal{Y}_w,
\]
which,  along with
\eqref{00PPO}, leads to 
\begin{equation}\label{VFR}
 f\cdot \mathcal{Y}_u  =
\frac{1}{\ev(\mathcal{Y}_{\operatorname{id}})}
\sum_{w\in \S_n}\left(\sum_{v\in \S_n} \ev(P_{v,w}(x_1,\ldots,x_n)\bT_v\cdot \mathcal{Y}_u)\right)\mathcal{Y}_w.   
\end{equation}
In view of \eqref{UYTR} and \eqref{EEE-1}, one can simplify  
\eqref{VFR} as
\begin{align*}
f\cdot \mathcal{Y}_u 
&   =
\frac{1}{\ev(\mathcal{Y}_{\operatorname{id}})}
\sum_{w\in \S_n} P_{u,w}(t_1,\ldots,t_n)\ev(\mathcal{Y}_{\operatorname{id}})\mathcal{Y}_w\\[5pt]
& = 
\sum_{w\in \S_n} P_{u, w}(t_1,\ldots,t_n)\mathcal{Y}_w. \qedhere
\end{align*}
\end{proof}

\section{Schubert calculus on Grassmannians}\label{sec4}

The aim of this section is to relate the Schubert calculus (in particular, the Pieri formula case) over Grassmannians to the computation  in the affine Hecke algebra  $\Haff_n$. The result is described  in Theorem \ref{ThPProdForm}, which can be regarded as a Grassmannian analogue of  Theorem \ref{Prodform}.

\subsection{Main result}
We start with some definitions and notation.

%This is done in Theorem \ref{ThPProdForm}, which can be regarded as a Grassmannian analogue of  Theorem \ref{Prodform}.

%A crucial structure we will introduce is as follows. 

\begin{Def}
Let $\mathbb{V}$ be a Schubert representation of $\Haff_n$ with basis $\{\mathcal{Y}_{w}\colon w\in\S_n\}$. 
For  $\lambda\subseteq (n-k)^k$, the \emph{generic Grassmannian Schubert class} is defined by
\begin{equation}\label{DEFGT}
\mathcal{Y}_{\lambda} = \sum_{v\in \S_k\times \S_{n-k}} p^{\ell(v)}\mathcal{Y}_{w_{\lambda}v}=
\sum_{\begin{subarray}{c}
w\in \S_n\\
\mathrm{Gr}(w)=\lambda
\end{subarray}} p^{\ell(w)-|\lambda|}\mathcal{Y}_w.  
\end{equation}
Here, recall that $w_\lambda$ is defined in \eqref{eq:parabolicdecomposition}.  
\end{Def}

By taking specific values of $(p,q,\hbar)$  as illustrated in Example \ref{ThSchubertcls}, we  obtain  the concrete realizations of generic
Grassmannian Schubert classes.

\begin{Eg}\label{specialvalue}
\begin{itemize}
    \item[(i)]    Schubert classes: $(p,q,\hbar) = (0,0,1)$. In this case, by Example \ref{ThSchubertcls}, we have $
\mathcal{Y}_{w} = [Y(w)]$.
Since $p=0$, the only basis element appearing in  \eqref{DEFGT} is $\mathcal{Y}_{w_\lambda}$, and so we have
$\mathcal{Y}_{\lambda}=\mathcal{Y}_{w_\lambda}=[Y(w_{\lambda})]$. Along with  Proposition \ref{SchRep1:Gr}(i),  we see that 
$$\mathcal{Y}_{\lambda} = [Y(w_{\lambda})]=\pi^*[Y(\lambda)]. $$

\item[(ii)] Classes of ideal sheaves:  $(p,q,\hbar) = (1,0,1)$. In this case,  we have $\mathcal{Y}_w =[\mathcal{I}_{\partial Y(w)}]$. So, by \eqref{DEFGT} and Proposition \ref{SchRep2:Gr}(i), 
$$\mathcal{Y}_{\lambda}=\sum_{
\begin{subarray}{c}
w\in \S_n\\
\mathrm{Gr}(w)=\lambda
\end{subarray}} \mathcal{Y}_w = \sum_{
\begin{subarray}{c}
w\in \S_n\\
\mathrm{Gr}(w)=\lambda
\end{subarray}}[\mathcal{I}_{\partial Y(w)}]=\pi^*[\mathcal{I}_{\partial Y(\lambda)}]. $$

\item[(iii)] CSM classes:  $(p,q,\hbar) = (1,1,1)$. In this case, we have $\mathcal{Y}_w=\CSM(Y(w)^\circ)$. It follows from  \eqref{DEFGT} that
$$\mathcal{Y}_{\lambda} =\sum_{
\begin{subarray}{c}
w\in \S_n\\
\mathrm{Gr}(w)=\lambda
\end{subarray}} \mathcal{Y}_w = 
\sum_{
\begin{subarray}{c}
w\in \S_n\\
\mathrm{Gr}(w)=\lambda
\end{subarray}}\CSM(Y(w)^\circ)=\CSM\left(\bigcup_{
\begin{subarray}{c}
w\in \S_n\\
\mathrm{Gr}(w)=\lambda
\end{subarray}}Y(w)^\circ\right). $$

\item[(iv)] Motivic Chern classes: $(p,q,\hbar) = (1,-y,1+y)$. In this case, we have $\mathcal{Y}_w=\MC_y(Y(w)^\circ)$. It follows from  \eqref{DEFGT} that
$$\mathcal{Y}_{\lambda} = \sum_{
\begin{subarray}{c}
w\in \S_n\\
\mathrm{Gr}(w)=\lambda
\end{subarray}} \MC_y(Y(w)^\circ)=\MC_y\left(\bigcup_{
\begin{subarray}{c}
w\in \S_n\\
\mathrm{Gr}(w)=\lambda
\end{subarray}} Y(w)^\circ\right). $$
\end{itemize}
\end{Eg}

For a nonnegative integer $m$, the \emph{quantum number} refers to
$$[m]=\frac{p^m-q^m}{p-q},\qquad [m]!=[m][m-1]\cdots [1].$$  {Here we follow the usual convention   $[0]!=1$. }

\begin{Def}\label{YHN5766}
The \emph{quantum symmetrizer} is defined as the following element in $\H_n$:
\begin{equation}
\Sigma_k^{n} = \frac{1}{[k]![n-k]!}\sum_{v\in \S_k\times \S_{n-k}} q^{d-\ell(v)}\bT_v,    
\end{equation}
where $d=\binom{k}{2}+\binom{n-k}{2}=\ell(v_0)$ with $v_0=k \cdots 1\, n \cdots (k+1)$ being the longest permutation in $\S_k\times \S_{n-k}$. 
With this notation, for  $\lambda\subseteq (n-k)^k$, we define the \emph{generic Grassmannian Hecke operator} by
\begin{equation}
 \bT_{\lambda} = \bT_{w_{\lambda}}\Sigma_k^n.   
\end{equation}
More concretely,
\[
\bT_{\lambda} = \frac{1}{[k]![n-k]!}\sum_{v\in \S_k\times \S_{n-k}} q^{d-\ell(v)}\bT_{w_{\lambda}v}=\frac{1}{[k]![n-k]!}\sum_{
\begin{subarray}{c}
w\in \S_n\\
\mathrm{Gr}(w)=\lambda
\end{subarray}} q^{d-\ell(w)+|\lambda|}\bT_{w}.
\]

\end{Def}

\begin{Rmk}
\begin{itemize}
\item {The structures  in Definition \ref{YHN5766} were essentially investigated  by Deodhar \cite{Deodhar} in his study of  the parabolic analogue  of Kazhdan--Lusztig polynomials.}

\item { The coefficient  in $\Sigma_k^n$ is a normalization due to the following  identity (see for example \cite[Chatper III (1.3)]{Mac}):
\begin{equation}\label{eq:qiden}
\frac{1}{[k]![n-k]!}\sum_{v\in \S_k\times \S_{n-k}} q^{d-\ell(v)}p^{\ell(v)}=1.
\end{equation}}
\end{itemize}
\end{Rmk}

Our main result in this section can now be stated as follows. 

\begin{Th}\label{ThPProdForm}
Let $\mathbb{V}$ be a Schubert representation of $\Haff_n$ with basis $\{\mathcal{Y}_{w}\colon w\in\S_n\}$. 
Assume that  $f\in \A[x_1,\ldots,x_n]$ is symmetric under $\S_k\times \S_{n-k}$, that is, $f$ is symmetric both in $x_1,\ldots, x_k$ and in $x_{k+1},\ldots, x_n$. 
\begin{itemize}
    \item[(i)] For $\mu\subseteq (n-k)^k$, 
the   product $\bT_\mu f$ in $\Haff_n$ can be  expressed as
\begin{equation}\label{BH-1}
\bT_\mu f=\sum_{\lambda\subseteq (n-k)^k} P_{\lambda, \mu}(x_1,\ldots,x_n)\,\bT_{\lambda},
\end{equation}
where $P_{\lambda, \mu}(x_1,\ldots,x_n)\in \A[x_1,\ldots,x_n].$

\item[(ii)] For $\lambda\subseteq (n-k)^k$,
  the action $f\cdot \mathcal{Y}_\lambda$
in  $\mathbb{V}$ is determined by \eqref{BH-1}, that is,
\begin{equation}\label{BH-2}
f\cdot \mathcal{Y}_\lambda = \sum_{\mu\subseteq (n-k)^k}P_{\lambda,  \mu}(t_1,\ldots,t_n) \,\mathcal{Y}_{\mu}.
\end{equation}
\end{itemize}
\end{Th}

% {\color{red}
% Combining Remark \ref{RRR-11} and the following proposition, it will be clear that  one can transform the computation of  structure constants over Grassmannians to 
% ****
% }
 
% Before proving this theorem, we see some examples, which allow us to transform the computation of coefficients over Grassmannians to those of $\mathcal{Y}_{\lambda}$. 

For our purpose of deriving Pieri type formulas, we pay attention to the case of elementary symmetric polynomials $f=e_r(x_1,\ldots, x_k)$. We  adopt the abbreviation $e_r(x_{[k]}):=e_r(x_1,\ldots, x_k)$. Based upon the observations in  Examples \ref{ThSchubertcls} and \ref{specialvalue}, one may transform the determination of Pieri formulas to the computation of the action in \eqref{BH-2}, as explained below.

\begin{Prop}
\label{SchRep:Grr}
%Let $P_{\lambda,\mu}(t)$ be the coefficient in %\eqref{BH-2}. 
\begin{itemize}
\item[(i)]  Schubert classes: $(p,q,\hbar) = (0,0,1)$. In this case, $\mathcal{Y}_{\lambda} =\pi^*[Y(\lambda)]$.
Suppose that 
\[
c_r(\mathcal{V}^\vee)\cdot [Y(\lambda)]=\sum_{\mu\subseteq (n-k)^k} P_{\lambda,  \mu}^{(1)}(t)\,  [Y(\mu)].
\]
Applying the injection $\pi^*$ to both sides yields that 
\[
e_r(x_{[k]})\cdot \mathcal{Y}_\lambda = \sum_{\mu\subseteq (n-k)^k}P_{\lambda,  \mu}^{(1)}(t)\, \mathcal{Y}_{\mu}.
\]

\item[(ii)] Classes of ideal sheaves:  $(p,q,\hbar) = (1,0,1)$. 
In this case, $\mathcal{Y}_{\lambda}=\pi^*[\mathcal{I}_{\partial Y(\lambda)}]$. 
Suppose that 
\[
c_r(\mathcal{V}^\vee)\cdot \mathcal{I}_{\partial Y(\lambda)}=\sum_{\mu\subseteq (n-k)^k} P_{\lambda,  \mu}^{(2)}(t)\,  \mathcal{I}_{\partial Y(\mu)}.
\]
Applying the injection $\pi^*$ to both sides yields that 
\[
e_r(x_{[k]})\cdot \mathcal{Y}_\lambda = \sum_{\mu\subseteq (n-k)^k}P_{\lambda,  \mu}^{(2)}(t)\, \mathcal{Y}_\mu.
\]

\item[(iii)] CSM classes:  $(p,q,\hbar) = (1,1,1)$.
In this case, 
\[
\mathcal{Y}_{\lambda}=\sum_{
\begin{subarray}{c}
w\in \S_n\\
\mathrm{Gr}(w)=\lambda
\end{subarray}}\CSM(Y(w)^\circ).
\]
Using Proposition \ref{SchRep3}(i), we obtain that
\[
\pi_*\mathcal{Y}_{\lambda}
=k!(n-k)!\cdot  \CSM(Y(\lambda)^\circ).
\]
Suppose that  
\[
e_r(x_{[k]})\cdot \mathcal{Y}_\lambda = \sum_{\mu\subseteq (n-k)^k}P_{\lambda,  \mu}^{(3)}(t)\, \mathcal{Y}_\mu.
\]
By the projection formula, 
\begin{align*}   c_r(\mathcal{V}^\vee)\cdot \CSM(Y(\lambda)^\circ)
& = \frac{1}{k!(n-k)!}c_r(\mathcal{V}^\vee)\cdot \pi_*\mathcal{Y}_{\lambda}
 =\frac{1}{k!(n-k)!}\pi_*(e_r(x_{[k]}) \cdot \mathcal{Y}_{\lambda})\\[5pt]
& = \frac{1}{k!(n-k)!} 
\sum_{\mu\subseteq (n-k)^k} P_{\lambda,\mu}^{(3)}(t)\,\pi_*\mathcal{Y}_{\mu}\\[5pt]
& =  \sum_{\mu\subseteq (n-k)^k} P_{\lambda,\mu}^{(3)}(t)\,\CSM(Y(\mu)^\circ).
\end{align*}

\item[(iv)] Motivic Chern classes: $(p,q,\hbar) = (1,-y,1+y)$. 
In this case, 
\[
\mathcal{Y}_{\lambda}=\sum_{
\begin{subarray}{c}
w\in \S_n\\
\mathrm{Gr}(w)=\lambda
\end{subarray}}\MC(Y(w)^\circ).
\]
By Proposition \ref{SchRep4:Gr}(i), 
we have 
\begin{align*}
\pi_*\mathcal{Y}_{\lambda}&=\left(\sum_{
\begin{subarray}{c}
w\in \S_n\\
\mathrm{Gr}(w)=\lambda
\end{subarray}}(-y)^{\binom{k}{2}+\binom{n-k}{2}-\ell(w)+|\Gr(w)|}\right)\MC_y(Y(\lambda)^\circ)\\[5pt]
&=  \left(\sum_{v\in \S_k\times \S_{n-k}} (-y)^{\binom{k}{2}+\binom{n-k}{2}-\ell(v)}\right)\,  \MC_y(Y(\lambda)^\circ)\\[5pt]
&=\left(\sum_{v\in \S_k\times \S_{n-k}} (-y)^{\ell(v)}\right)\,  \MC_y(Y(\lambda)^\circ).  
\end{align*}
%\[
%\pi_*\mathcal{Y}_{\lambda}
%=\left(\sum_{v\in \S_k\times \S_{n-k}} (-y)^{\ell(v)}\right)\,  \MC_y(Y(\lambda)^\circ).
%\]
Suppose that  
\[
e_r(x_{[k]})\cdot \mathcal{Y}_\lambda = \sum_{\mu\subseteq (n-k)^k}P_{\lambda,  \mu}^{(4)}(t)\, \mathcal{Y}_\mu.
\]
Still, as in (iii),  using the projection formula, we obtain that 
$$c_r(\mathcal{V}^\vee)\cdot \MC_y(Y(\lambda)^\circ)=\sum_{\mu\subseteq (n-k)^k} P_{\lambda,\mu}^{(4)}(t)\,\MC_y(Y(\mu)^\circ).$$
\end{itemize}
\end{Prop}

\subsection{Proof of  Theorem \ref{ThPProdForm}}
In order to give a proof of Theorem \ref{ThPProdForm}, we first demonstrate  several lemmas, which we think might possibly  be known to experts in   affine Hecke algebras.

\begin{Lemma}\label{PLem1}
Let  $f\in \A[x_1,\ldots,x_n]$. If   $f$ is symmetric in $x_i$ and $x_{i+1}$, then  
\begin{equation}\label{PPOO-1}
\bT_if=f\bT_i.
\end{equation}
\end{Lemma}

\begin{proof}
By \eqref{Lebniz1}, this is clear for $f\in \A[x_j\colon j\neq i,i+1]$. Since $f$ is symmetric in $x_i$ and $x_{i+1}$, ``the fundamental theorem of symmetric functions'' tells  that $f$  can be expressed  as a polynomial in $e_1(x_i, x_{i+1})=x_i+x_{i+1}$ and $e_2(x_i, x_{i+1})=x_ix_{i+1}$ with coefficients in $\A[x_j\colon j\neq i,i+1]$. Hence it suffices to verify \eqref{PPOO-1} for 
 $f=x_i+x_{i+1}$ or $f=x_ix_{i+1}$. 
It follows from \eqref{Lebniz2} and \eqref{Lebniz3} that 
\begin{align*}
\bT_ix_i&=x_{i+1}\bT_i+(\hbar-(p-q)x_{i+1}),\\[5pt]
\bT_ix_{i+1}&=x_{i}\bT_i-(\hbar-(p-q)x_{i+1}),
\end{align*}
from which we can directly check that 
$\bT_i(x_i+x_{i+1})=(x_i+x_{i+1})\bT_i$ 
and $\bT_i(x_ix_{i+1})=(x_ix_{i+1})\bT_i$. 
\end{proof}

\begin{Lemma}\label{absorb}
For $v\in \S_k\times \S_{n-k}$, we have 
\begin{equation}\label{Lem-12}
\Sigma_k^n\bT_v=\bT_v\Sigma^n_k=p^{\ell(v)}\Sigma^n_k.
\end{equation}
In particular, if $w\in \S_n$ has the parabolic decomposition  $w=w_{\lambda}v$, then 
\begin{equation}\label{Lem-1223}
\bT_w\Sigma_k^n = \bT_{w_{\lambda}}\bT_v \Sigma_k^n = p^{\ell(v)}\bT_{\lambda}.
\end{equation}
\end{Lemma}

\begin{proof}
To conclude \eqref{Lem-12}, it suffices to check 
that for $v=s_i$ with $1\leq i<k$ or $k<i< n$, 
\[\Sigma_k^n\bT_{i}=\bT_{i}\Sigma^n_k=p\Sigma^n_k.\]
To justify $\Sigma_k^n\bT_{i}=p\Sigma^n_k$, notice that 
\begin{align*}
\Sigma^n_k&=\frac{1}{[k]![n-k]!}
\left(\sum_{v(i)<v(i+1)} q^{d-\ell(v)}\bT_v+\sum_{v(i)<v(i+1)} q^{d-\ell(v)-1}\bT_{v}\bT_i\right)\\[5pt]
&=\frac{1}{[k]![n-k]!}
\left(\sum_{v(i)<v(i+1)} q^{d-1-\ell(v)}\bT_v\right)
\big(\bT_i+q\big).  
\end{align*}
Since $(\bT_i+q)\bT_i=p(\bT_i+q),$
we deduce that $\Sigma_k^n\bT_{i}=p\Sigma^n_k$.
Similarly, we obtain $\bT_{i}\Sigma_k^n=p\Sigma^n_k$ by noticing  
\begin{align*}
 \Sigma^n_k&=
 \frac{1}{[k]![n-k]!}
\left(\sum_{v^{-1}(i)<v^{-1}(i+1)} q^{d-\ell(v)}\bT_v+\sum_{v^{-1}(i)<v^{-1}(i+1)} q^{d-1-\ell(v)}\bT_i\bT_{v}\right)\\[5pt]
&=\frac{1}{[k]![n-k]!}\big(\bT_i+q\big)
\left(\sum_{v^{-1}(i)<v^{-1}(i+1)} q^{d-1-\ell(v)}\bT_v\right)
\end{align*}
and   $\bT_i(\bT_i+q)=p(\bT_i+q)$.
\end{proof}

The third lemma concerns  another \emph{quantum symmetrizer}
$$\Pi_k^n = \sum_{v\in \S_k\times \S_{n-k}} p^{d-\ell(v)}\T_v,$$
where, as before, $d=\binom{k}{2}+\binom{n-k}{2}$ is the length of the longest permutation  in $\S_k\times \S_{n-k}$.

\begin{Lemma}\label{PLem1prime} 
Suppose that   $w\in \S_n$ has the parabolic decomposition  $w=w_{\lambda}v$ with respect to $\S_k\times \S_{n-k}$. Then, for the basis element $\mathcal{Y}_{w}$ in  $\mathbb{V}$,
$$\Pi_k^n\cdot \mathcal{Y}_{w}=q^{d-\ell(v)}\mathcal{Y}_{\lambda}.$$
\end{Lemma}

\begin{proof}
Similar to Lemma \ref{absorb}, we can deduce that for any $v\in \S_k\times \S_{n-k}$,
\begin{equation}\label{PO-12}
\Pi_k^n\T_v=\T_v\Pi^n_k=q^{\ell(v)}\Pi^n_k.
\end{equation}
If $w=w_{\lambda}v_0$ with $v_0$ being the longest element of $\S_k\times \S_{n-k}$, then 
\begin{align*}
\Pi_k^n\cdot \mathcal{Y}_{w} 
& = \sum_{v\in \S_k\times \S_{n-k}} p^{d-\ell(v)}\T_v
\cdot\mathcal{Y}_{w}
= \sum_{v\in \S_k\times \S_{n-k}}p^{d-\ell(v)}\mathcal{Y}_{w_{\lambda}v_0v^{-1}}\\[5pt]
& = \sum_{v\in \S_k\times \S_{n-k}}p^{\ell(v)}\mathcal{Y}_{w_{\lambda}v} 
=\mathcal{Y}_{\lambda}.
\end{align*}
In general, note that 
$$\mathcal{Y}_w=\T_{v^{-1}v_0}\mathcal{Y}_{w_{\lambda}v_0}.$$
This, combined with \eqref{PO-12} and the above case for $w_{\lambda}v_0$, gives 
\begin{align*}
\Pi_k^n\cdot\mathcal{Y}_{w} = \Pi_k^n\cdot \left(\T_{v^{-1}v_0}\cdot\mathcal{Y}_{w_{\lambda}v_0}\right)
=q^{\ell(v^{-1}v_0)}\Pi_k^n\cdot\mathcal{Y}_{w_{\lambda}v_0}=q^{d-\ell(v)}\mathcal{Y}_{\lambda}.
\end{align*}
This completes the proof.
\end{proof}

We still need a  Grassmannian
analogue of Theorem \ref{Reconstruction}.

\begin{Th}\label{dualbasis}
For   $\mu,\nu\subseteq (n-k)^k$, we have
$$\ev\left(\bT_{\mu}\cdot\mathcal{Y}_{\nu}\right)=
\begin{cases}
\ev(\mathcal{Y}_{\operatorname{id}}), & \mu=\nu,\\[5pt]
0, & \text{otherwise}.
\end{cases}
$$
Moreover, for any $\mathbb{F}$-linear combination $f$ of $\mathcal{Y}_{\lambda}$'s, we have
\begin{equation}\label{OO-11}
f=\frac{1}{\ev(\mathcal{Y}_{\operatorname{id}})}\sum_{\lambda\subseteq (n-k)^k} \ev\left(\bT_\lambda \cdot f\right)\mathcal{Y}_\lambda.
\end{equation}
\end{Th}

\begin{proof}
By the definitions of $\bT_{\mu}$ and 
$\mathcal{Y}_{\nu}$,
we see that
\[
\ev\left(\bT_{\mu}\cdot\mathcal{Y}_{\nu}\right)
= \frac{1}{[k]![n-k]!}\sum_{v_1,v_2\in \S_k\times \S_{n-k}} q^{d-\ell(v_1)}p^{\ell(v_2)}\ev\left(\bT_{w_{\mu}v_1}\cdot\mathcal{Y}_{w_{\nu}v_2}\right),
\]
which, in conjunction  with Theorem \ref{Reconstruction}, reduces to
\[
\delta_{\mu\nu}\frac{\ev(\mathcal{Y}_{\operatorname{id}})}{[k]![n-k]!}\sum_{v\in \S_k\times \S_{n-k}} q^{d-\ell(v)}p^{\ell(v)},
\]
and can be further  simplified to 
$\delta_{\mu\nu}\ev(\mathcal{Y}_{\operatorname{id}})$ by means of \eqref{eq:qiden}.
The proof of the expansion in \eqref{OO-11}  is completely  the same as that in Theorem \ref{Reconstruction}, and thus is omitted. 
\end{proof}

We are now in a position to finish the proof of Theorem  \ref{ThPProdForm}.

\begin{proof}[Proof of Theorem \ref{ThPProdForm}]
Since  $f\in \A[x_1,\ldots,x_n]$ is symmetric under $\S_k\times \S_{n-k}$,
by Lemma \ref{PLem1}, we see that $f$ commutes with $\bT_v$ for each $v\in \S_k\times \S_{n-k}$, and thus $\Sigma_{k}^n f=f \Sigma_k^n.$
In particular,
$$\bT_{\mu}f=\bT_{w_{\mu}}\Sigma_k^nf=
\bT_{w_{\mu}}f\Sigma_k^n\in \Haff_n\cdot \Sigma_k^n=\sum_{\lambda\subseteq (n-k)^k}\A[x_1,\ldots,x_n]\cdot\bT_{\lambda},$$
where the last equality is implied by 
\eqref{Lem-1223}. 
This shows that $\bT_\mu f$ can be expanded  as a form  given in \eqref{BH-1}.

We next show that $f\cdot \mathcal{Y}_{\lambda}$ has an expansion in \eqref{BH-2}. 
Let $v_0\in \S_k\times \S_{n-k}$ be the longest element. By Lemma \ref{PLem1prime}, we have
$ \mathcal{Y}_{\lambda}=\Pi_k^n\cdot \mathcal{Y}_{w_\lambda v_0}$. 
Similar to Lemma \ref{PLem1}, we
can deduce $\Pi_k^nf=f\Pi_k^n$. 
Thus, 
\[
f\cdot \mathcal{Y}_{\lambda} 
= f\cdot \left(\Pi_k^n\cdot\mathcal{Y}_{w_\lambda v_0}\right)
= \Pi_k^n\cdot (f\cdot \mathcal{Y}_{w_\lambda v_0}).
\]
Expanding   $f\cdot \mathcal{Y}_{w_\lambda v_0}$   in the basis $\{\mathcal{Y}_w\colon w\in \S_n\}$ and using   Lemma \ref{PLem1prime}, we  see that $f\cdot \mathcal{Y}_{\lambda}$ can be  expanded in the form of \eqref{BH-2}. 
Finally, we  invoke Theorem \ref{dualbasis}, and then use  the same manner as in the proof of  Theorem \ref{Prodform} to conclude that the coefficients in  \eqref{BH-1} and  \eqref{BH-2} coincide. 
\end{proof}

\section{Pieri type formulas}\label{sec5}

To establish our  Pieri formulas, it follows from Theorem \ref{ThPProdForm} and Proposition  \ref{SchRep:Grr}
that we need to evaluate the action $e_r(x_{[k]})\cdot \mathcal{Y}_{\lambda}$, or equivalently, 
the product $\bT_{\mu}e_r(x_{[k]})$, which is the goal of this section.

 \subsection{Main result}
 
Recall that  for
$\lambda\subseteq (n-k)^k$ and 
$1\leq i\leq k$, 
\begin{equation}\label{POIU-1}
\fhead{i}\to \lambda = 
t_{c}\cdot \lambda+
\sum_{\mu/\lambda=\eta}
\left(\hbar-(p-q)t_{\mathtt{h}(\eta)}\right)p^{\htt(\eta)-1}q^{\wdd(\eta)-1}\cdot\mu,
\end{equation}
where $c=\lambda_i+k+1-i$, 
$\mathtt{h}(\eta)=\mu_i+k+1-i$,
and the sum runs   over  $\mu\subseteq (n-k)^k$ such that $\mu/\lambda$ is a ribbon with {head} in   row $i$.  Replacing $\lambda$ and $\mu$ respectively by $\mathcal{Y}_{\lambda}$ and  $\mathcal{Y}_{\mu}$   in \eqref{POIU-1}, we are led to the notation  
\[
\fhead{i}\to \mathcal{Y}_{\lambda}.
\]

The main theorem is as follows.

\begin{Th}\label{MainThPieri}
Let $\mathbb{V}$ be a Schubert representation of $\Haff_n$ with basis $\{\mathcal{Y}_{w}\colon w\in\S_n\}$. 
For  $\lambda\subseteq (n-k)^k$ and  $0\leq r\leq k$,  we have 
\begin{equation}\label{UTYRE-1}
e_r(x_{[k]})\cdot \mathcal{Y}_{\lambda}
=\sum_{1\leq i_1<\cdots<i_r\leq k}
\fhead{i_r}\to\cdots \to
\fhead{i_1}\to \mathcal{Y}_{\lambda}.
\end{equation}

\end{Th}

% Theorem \ref{MainThPieri}, together with  Theorem \ref{EHTT-123} in the appendix, leads to the following  expansion of 
% $h_r(x_{[k]})\cdot \mathcal{Y}_{\lambda}$.

% \begin{Th} 
% For $0\leq r\leq k$, 
% \begin{equation*} 
% h_r(x_{[k]})\cdot \mathcal{Y}_{\lambda}
% =\sum_{1\leq i_1\leq \cdots \leq  i_r\leq k}
% \fhead{i_1}\to\cdots \to
% \fhead{i_r}\to \mathcal{Y}_{\lambda}.
% \end{equation*}
% \end{Th}

By Theorem \ref{ThPProdForm},  the proof of Theorem \ref{MainThPieri} 
is equivalent to  the determination of  $\bT_{\mu}e_r(x_{[k]})$. To expand $\bT_{\mu}e_r(x_{[k]})$, we need the dual version  of the ribbon Schubert operator in \eqref{POIU-1}. That is, for $\mu\subseteq (n-k)^k$ and $1\leq i\leq k$, let 
\begin{equation}\label{YHNG}
 \mu\to\fhead{i}
=t_{c}\cdot \mu+
\sum_{\mu/\lambda=\eta}
\left(\hbar-(p-q)t_{\mathtt{h}(\eta)}\right)p^{\htt(\eta)-1}q^{\wdd(\eta)-1}\cdot\lambda,   
\end{equation}
where $c=\mathtt{h}(\eta)=\mu_i+k+1-i$, 
and the sum is   over  $\lambda\subseteq (n-k)^k$ such that $\mu/\lambda$ is a ribbon with {head} in   row $i$. Similarly,
let 
\[
\bT_\mu \to\fhead{i} 
\]
be defined by replacing $\mu$ and $\lambda$  
respectively by $\bT_\mu$ and $\bT_\lambda$  in \eqref{YHNG}.

\begin{Th}\label{MainThPieri00}
For  $\mu\subseteq (n-k)^k$ and  $0\leq r\leq k$,  we have 
\begin{equation}\label{UTYRE-100}
\bT_{\mu}e_r(x_{[k]})
=\sum_{1\leq i_1<\cdots<i_r\leq k}
\bT_{\mu}\to \fhead{i_r}\to\cdots\to\fhead{i_1}.
\end{equation}

\end{Th}

Using   Theorem \ref{ThPProdForm} and comparing the definitions in \eqref{POIU-1} and \eqref{YHNG}, it can be seen  that once 
Theorem \ref{MainThPieri00} is true, we arrive at a proof of  Theorem \ref{MainThPieri}.
By Lemma \ref{PLem1}, we notice that
$$\bT_{\mu}e_r(x_{[k]})=\bT_{w_{\mu}} \Sigma_k^n\,e_r(x_{[k]})=
\bT_{w_{\mu}}e_r(x_{[k]})\Sigma_k^n. $$
In fact, we have the following refinement,  which may also be regarded as an algebraic interpretation  of ribbon Schubert  operators. 

\begin{Th}\label{Amonomial}
For $\mu\subseteq (n-k)^k$ and  a subset $I=\{ i_1<\cdots<i_r\}$ of $\{1,\ldots, k\}$, we have
$$\bT_{w_{\mu}}x_{i_1}\cdots x_{i_r} \Sigma_k^n = 
\bT_{\mu}\to \fhead{k+1-i_1}\to\cdots\to\fhead{k+1-i_r}.$$
\end{Th}

Summing over all subsets $\{ i_1<\cdots<i_r\}$ of $\{1,\ldots, k\}$, Theorem \ref{Amonomial}  implies Theorem \ref{MainThPieri00}, because of the   correspondence between $i_1<\cdots<i_r$ and $k+1-i_r<\cdots< k+1-i_1$.

The rest of this section will be devoted to a  proof of Theorem  \ref{Amonomial}. Our approach is to use induction on $k$.
For simplicity, we use the following notation: 
for $i>0$ and $j\geq 0$, 
$$
\bT_i^{[j]}=\bT_{i+j-1}\cdots \bT_{i+1}\bT_i.
$$
Here, for $j=0$, we adopt the convention that $\bT_i^{[0]}=\T_{\operatorname{id}}$.
If letting  
\[s_i^{[j]}=s_{i+j-1}\cdots s_{i+1}s_i,\] then  for $\lambda\subseteq (n-k)^k$, $w_\lambda$  has the following reduced decomposition 
\[
w_\lambda=s^{[\lambda_k]}_1\, s^{[\lambda_{k-1}]}_{2}\cdots s^{[\lambda_1]}_k.
\]
Thus  
$$\bT_{w_\lambda}=\bT^{[\lambda_k]}_1\,\bT^{[\lambda_{k-1}]}_{2}\cdots \bT^{[\lambda_1]}_k.$$
For example, when $k=3$, $n=8$ and $\lambda=(4,2,1)$, we have $w_\lambda=24713568$, and so 
\begin{align*}
w_{(4,2,1)}=s_1\cdot s_3s_2\cdot s_6s_5s_4s_3,\qquad 
\bT_{w_{(4,2,1)}} = 
\bT_1
\cdot\bT_3\bT_2
\cdot\bT_6\bT_5\bT_4\bT_3.
\end{align*}

We now present a proof of Theorem \ref{Amonomial}. In the proof, we shall need results in several lemmas whose proofs  will be given  in Subsection \ref{lemmassome}

\begin{proof}[Proof of Theorem \ref{Amonomial}]
As promised, we use induction on $k$. 
We first check the case of $k=1$, namely, $\mu=(\mu_1)$. We need to justify
\[
\bT_{w_\mu}x_1\Sigma_1^n=\bT_\mu\rightarrow   \fhead{1},
\]
which is exactly the  statement of Lemma \ref{PLem4} in the case when $k=1$  and $j=\mu_1$.

Now assume that $k>1$. Denote by $\mu^-=(\mu_2, \ldots,\mu_k)$ the partition obtained from $\mu$ by removing $\mu_1$. Then $\mu^-$ belongs to the $(k-1)\times \mu_2$ rectangle $(\mu_2)^{k-1}$. Let $m=k-1+\mu_2$. So $w_{\mu^-}$ is a Grassmannian permutation of $\S_m$. Notice that $\bT_{w_\mu}=\bT_{w_{\mu^-}}\, \bT_{k}^{[\mu_1]}$.

Case 1. 
  $k\notin I$, that is, $i_r<k$. By Lemma \ref{PLem2-1} and induction,
\begin{align}
\bT_{w_{\mu}}x_{i_1}\cdots x_{i_r} \Sigma_k^n
& = \bT_{w_{\mu^-}}x_{i_1}\cdots x_{i_r}\Sigma_{k-1}^{m}\,\bT_{k}^{[\mu_1]} \Sigma_k^n\nonumber\\[5pt]
& = \left(\bT_{\mu^-}\rightarrow \fhead{k-i_1}\to\cdots\rightarrow\fhead{k-i_r}\right)
\bT_{k}^{[\mu_1]}\Sigma_k^n.\label{HGF-1}
\end{align}
For each $\bT_\nu$ appearing in 
$\bT_{\mu^-}\rightarrow \fhead{k-i_1}\to\cdots\rightarrow\fhead{k-i_r}$, by \eqref{tmksigma}, we find that 
$\bT_{\nu}\, \bT_{k}^{[\mu_1]}\, \Sigma_{k}^n=\bT_{\nu^+}$, 
where $\nu^+=(\mu_1, \nu_1, \ldots, \nu_{k-1})\subseteq (n-k)^k$ is obtained from $\nu$ by adding $\mu_1$ as the largest part.
So, \eqref{HGF-1} can be reformulated as
\[
\bT_{\mu}\rightarrow \fhead{k+1-i_1}\to\cdots\rightarrow\fhead{k+1-i_r}.
\]

Case 2.   $k\in I$, that is, $i_r=k$.  By Lemma \ref{PLem2-1} and induction,
\begin{align}
\bT_{w_{\mu}} x_{i_1}\cdots x_{i_r} \Sigma_k^n
& = \bT_{w_{\mu^-}}x_{i_1}\cdots x_{i_{r-1}}\Sigma_{k-1}^{m}\,\bT_{k}^{[\mu_1]}x_k\Sigma_k^n\nonumber\\[5pt]
& = \left(\bT_{\mu^-}\rightarrow \fhead{k-i_1}\to\cdots\rightarrow\fhead{k-i_{r-1}}\right)\,
\bT_{k}^{[\mu_1]} x_k\Sigma_k^n.\label{TIEA}
\end{align}
Assume that 
\begin{align}\label{TIEB}
\bT_{\mu^-}\rightarrow \fhead{k-i_1}\to\cdots\rightarrow\fhead{k-i_{r-1}}
=\sum_{\nu\subseteq (\mu_2)^{k-1}}c_{\nu}\,\bT_{\nu}.    \end{align}
We have to consider the product 
\[
\bT_{\nu}\bT_{k}^{[\mu_1]} x_k\Sigma_k^n=\bT_{w_\nu}\Sigma_{k-1}^m\bT_{k}^{[\mu_1]} x_k\Sigma_k^n.
\]

By Lemma \ref{PLem2}, 
\[
\Sigma_{k-1}^m\bT_{k}^{[\mu_1]} x_k\Sigma_k^n=\bT_{k}^{[\mu_1]} x_k\Sigma_k^n,
\]
and so
\[
\bT_{\nu}\bT_{k}^{[\mu_1]} x_k\Sigma_k^n = \bT_{w_\nu}\bT_{k}^{[\mu_1]} x_k\Sigma_k^n.
\]
Applying Lemma \ref{PLem4} to $\bT_{k}^{[\mu_1]} x_k\Sigma_k^n$ on the right-hand side, we obtain that
\begin{align}\label{12YFR}
\bT_{\nu}\bT_{k}^{[\mu_1]} x_k\Sigma_k^n= \bT_{w_\nu}\left(x_{k+\mu_1}\bT_k^{[\mu_1]}\Sigma_k^n+\sum_{a+b=\mu_1-1}(\hbar-(p-q)x_{k+\mu_1}) q^{b}\bT_{k}^{[a]}\Sigma_k^n\right).
\end{align}
Since $\nu$ is in the rectangle $(\mu_2)^{k-1}$,
it follows from \eqref{Lebniz1} that $\bT_{w_\nu}$  and $x_{k+\mu_1}$ commute, and so \eqref{12YFR} can be rewritten as 
\begin{align}\label{TIEC}
\bT_{\nu}\bT_{k}^{[\mu_1]} x_k\Sigma_k^n =
x_{k+\mu_1}\bT_{w_\nu}\bT_k^{[\mu_1]}\Sigma_k^n+\sum_{a+b=\mu_1-1}(\hbar-(p-q)x_{k+\mu_1}) q^{b}\bT_{w_\nu}\bT_{k}^{[a]}\Sigma_k^n.
\end{align}

Let  $\nu^+=(\mu_1, \nu_1, \ldots, \nu_{k-1})\subseteq (n-k)^k$. 
It is clear that  
\begin{align}\label{U7643}
\bT_{w_\nu}\bT_k^{[\mu_1]}\Sigma_k^n=\bT_{w_{\nu^+}}\Sigma_k^n=\bT_{\nu^+}. 
\end{align}
On the other hand, 
 by Lemma \ref{PLem3-1}, we obtain that 
 \begin{align*}
q^{b}\bT_{w_\nu}\bT_{k}^{[a]}\Sigma_k^n= q^{\wdd(\nu^+/\nu')-1} p^{\htt(\nu^+/\nu')-1}\bT_{\nu'},    
 \end{align*}
 where $\nu^+/\nu'$ is the ribbon with head in the
 first row and with width equal to $\mu_1-a=b+1$.
Therefore, the total contribution of the  summation term in \eqref{TIEC} is exactly
\begin{align}\label{PIE09}
\sum_{\nu'}(\hbar-(p-q)x_{k+\mu_1})\, p^{\htt(\nu^+/\nu')-1}q^{\wdd(\nu^+/\nu')-1}\bT_{\nu'},    
\end{align}
ranging over  $\nu'$ obtained from $\nu^+$ by deleting a ribbon with head in the first row. 
In view of \eqref{U7643} and \eqref{PIE09}, \eqref{TIEC} becomes 
\begin{align}\label{BUBUBU}
\bT_{\nu}\bT_{k}^{[\mu_1]} x_k\Sigma_k^n &=x_{k+\mu_1}\bT_{\nu^+}+\sum_{\nu'}(\hbar-(p-q)x_{k+\mu_1})\, p^{\htt(\nu^+/\nu')-1}q^{\wdd(\nu^+/\nu')-1}\bT_{\nu'}\nonumber\\
&=\bT_{\nu^+}\to\fhead{1}.
\end{align}
 
Collecting \eqref{TIEA}, \eqref{TIEA} and \eqref{BUBUBU}, we finally conclude that 
$$\bT_{w_{\mu}}x_{i_1}\cdots x_{i_{r-1}}x_k \Sigma_k^n =\bT_{\mu}\rightarrow \fhead{k+1-i_1}\to\cdots\rightarrow\fhead{k+1-i_r}\rightarrow\fhead{1}.$$
The completes the proof of Case 2. 
\end{proof}

% The rest of this section will be devoted to a  proof of Theorem  \ref{Amonomial}. Our approach is to use induction on $k$.
% For simplicity, we use the following notation: 
% for $i>0$ and $j\geq 0$, 
% $$
% \bT_i^{[j]}=\bT_{i+j-1}\cdots \bT_{i+1}\bT_i.
% $$
% Here, for $j=0$, we adopt the convention that $\bT_i^{[0]}=\T_{\operatorname{id}}$.
% For $d_1,d_2,\ldots,d_i\geq 0$, we denote 
% $$\bT^{[d_1,\ldots,d_i]}=\bT_{1}^{[d_1]}\bT_{2}^{[d_2]}\cdots \bT_{i}^{[d_i]}.$$
% If letting  
% \[s_i^{[j]}=s_{i+j-1}\cdots s_{i+1}s_i,\qquad 
% s^{[d_1,\ldots,d_i]}=s_1^{[d_1]}\cdots s_i^{[d_i]}\]
% then $w_\lambda$ for $\lambda\subseteq (n-k)^k$  has the following reduced decomposition 
% \[w_\lambda
% =s^{[\lambda_k,\ldots,\lambda_2,\lambda_1]}
% =s^{[\lambda_k]}_1\cdots s^{[\lambda_2]}_{k-1}\,s^{[\lambda_1]}_k.\]
% Thus one has
% \[
% \bT_{w_\lambda}=\bT^{[\lambda_k,\ldots,\lambda_2,\lambda_1]}=\bT^{[\lambda_k]}_1\cdots \bT^{[\lambda_2]}_{k-1}\,\bT^{[\lambda_1]}_k.\]
% For example, when $k=3$, $n=8$ and $\lambda=(4,2,1)$, we see that
% \begin{align*}
% \bT_{w_{(4,2,1)}} = 
% \bT_1
% \cdot\bT_3\bT_2
% \cdot\bT_6\bT_5\bT_4\bT_3.
% \end{align*}

\subsection{Lemmas for the proof of Theorem \ref{Amonomial}}\label{lemmassome}

\begin{Lemma}\label{PLem2}
For $1\leq k\leq m<k+j\leq n$ and any polynomial $f(x_k)$ in $x_k$, we have 
$$\Sigma_{k-1}^m\bT_k^{[j]}f(x_{k})\Sigma_k^n=\bT_k^{[j]}f(x_k)\Sigma_k^n.$$
\end{Lemma}

\begin{proof}
We first check that for $v\in \S_{k-1}\times \S_{m-k+1}$, 
\begin{equation}\label{II-1}
\bT_v\bT_k^{[j]}f(x_{k})\Sigma_k^n=p^{\ell(v)}\bT_k^{[j]}f(x_{k})\Sigma_k^n.
\end{equation}
It suffices to justify \eqref{II-1} for $v=s_i$ with $1\leq i<k-1$ or $k-1<i<m$. 
To do this, we notice that
\begin{align}\label{ETBUY}
\bT_i\bT_{k}^{[j]}f(x_k)
=\begin{cases}
\bT_{k}^{[j]}\bT_if(x_k),& 1\leq i< k-1,\\[5pt]
\bT_{k}^{[j]}\bT_{i+1}f(x_k),& k-1<i<m,
\end{cases}
\end{align}
where the first equality  follows directly from \eqref{barBraidRelation1}, and the second equality  can be obtained by applying  the  braid relations in  \eqref{barBraidRelation1} and \eqref{barBraidRelation2}: 
\begin{align*}
\bT_i\bT_{k}^{[j]}
& = \bT_i(\cdots \bT_{i+1}\bT_{i}\cdots )
  = \cdots \bT_i\bT_{i+1}\bT_{i}\cdots \\
& = \cdots \bT_{i+1}\bT_{i}\bT_{i+1}\cdots 
  = (\cdots \bT_{i+1}\bT_{i}\cdots )\bT_{i+1}.
\end{align*}
In the above derivation, the condition  $m<k+j$ ensures that there  must appear the factor $\bT_{i+1}$
in $\bT_{k}^{[j]}$. 
By the commutative relation in \eqref{Lebniz1}, \eqref{ETBUY} can be written as 
$$\bT_i\bT_{k}^{[j]}f(x_k)
=\begin{cases}
\bT_{k}^{[j]}f(x_k)\bT_i,& 1\leq i< k-1,\\[5pt]
\bT_{k}^{[j]}f(x_k)\bT_{i+1},& k-1<i<m.
\end{cases}$$
%If one prefers, then it can be illustrated by diagrammatics
Therefore, by Lemma \ref{absorb}, we obtain 
$$
\bT_i\bT_k^{[j]}f(x_k)\Sigma_k^n
= p\bT_k^{[j]}f(x_k)\Sigma_k^n.$$
This verifies \eqref{II-1} for $v=s_i$, and thereby for general $v\in \S_{k-1}\times \S_{m-k+1}$. 
We can now conclude the lemma by combining \eqref{II-1} and \eqref{eq:qiden}. 
\end{proof}

\begin{Lemma}\label{PLem2-1}
For $\mu\subseteq (n-k)^k$ and a monomial $x_1^{d_1}\cdots x_k^{d_k}$,
\begin{equation}\label{POQ-1}
\bT_{w_{\mu}}x_1^{d_1}\cdots x_k^{d_k}\Sigma_{k}^n
= \bT_{w_{\mu^-}}x_1^{d_1}\cdots x_{k-1}^{d_{k-1}}\Sigma_{k-1}^{m}
\bT_{k}^{[\mu_1]}x_k^{d_k}\Sigma_{k}^n,
\end{equation}
where $m= k-1+\mu_2$, and $\mu^-=(\mu_2,\ldots, \mu_{k})  $ is the partition in the rectangle $(\mu_2)^{k-1}$   by removing  $\mu_1$ from $\mu$. In particular, when $d_1=\cdots=d_k=0$,   we have
\begin{align}\label{tmksigma}
\bT_\mu=\bT_{\mu^-}\, \bT_{k}^{[\mu_1]}\, \Sigma_{k}^n.
\end{align}
% % 
% \begin{equation}\label{POQ-1}
% \bT^{[\mu_k,\ldots,\mu_2,\mu_1]}x_1^{d_1}\cdots x_k^{d_k}\Sigma_{k}^n
% = \bT^{[\mu_k,\ldots,\mu_2]}x_1^{d_1}\cdots x_{k-1}^{d_{k-1}}\Sigma_{k-1}^{m}
% \bT_{k}^{[\mu_1]}x_k^{d_k}\Sigma_{k}^n,
% \end{equation}
% where $m= k-1+\mu_2<k+\mu_1$. 
\end{Lemma}

\begin{proof}
This can be done via  straightforward  computation:
\begin{align*}
\bT_{w_{\mu}}x_1^{d_1}\cdots x_k^{d_k}\Sigma_{k}^n
& = 
\bT_{w_{\mu^-}}
\bT_{k}^{[\mu_1]}x_1^{d_1}\cdots x_k^{d_k}\Sigma_{k}^n\\[5pt]
& =
\bT_{w_{\mu^-}}x_1^{d_1}\cdots x_{k-1}^{d_{k-1}}
\bT_{k}^{[\mu_1]}x_k^{d_k}\Sigma_{k}^n\\[5pt]
& =
\bT_{w_{\mu^-}}x_1^{d_1}\cdots x_{k-1}^{d_{k-1}}
\Sigma_{k-1}^m\bT_{k}^{[\mu_1]}x_k^{d_k}\Sigma_{k}^n,
\end{align*}
where the second equality is clear from \eqref{Lebniz1}, and the last equality follows from Lemma \ref{PLem2}.
\end{proof}

The next lemma  reformulates   the factor $\bT_{k}^{[\mu_1]}x_k^{d_k}\Sigma_{k}^n$ in \eqref{POQ-1} in the case  $d_k=1$.
Its proof requies  the following identities 
\begin{gather}\label{EQQQ-1}
\bT_i x_i % =x_{i+1}\T_i+\hbar
=x_{i+1}\bT_i+(\hbar-(p-q)x_{i+1}),\\[5pt]
% x_i\bT_i % =\T_ix_{i+1}+\hbar
% =\bT_i x_{i+1}+(\hbar-(p-q)x_{i+1}),
% ,\\[5pt]
\bT_i(\hbar-(p-q)x_i) =(\hbar-(p-q)x_{i+1})\T_i,\label{EQQQ-2}
% \\[5pt]
% (\hbar-(p-q)x_i)\bT_i =\T_i(\hbar-(p-q)x_{i+1}).%\label{EQQQ-3}
\end{gather}
which can be easily deduced  from \eqref{Lebniz2} and \eqref{Lebniz3}.

\begin{Lemma}\label{PLem4}
For $k+j\leq n$,
\begin{equation}\label{HU-11}
\bT_k^{[j]}x_k\Sigma_k^n
=x_{k+j}\bT_k^{[j]}\Sigma_k^n+\sum_{a+b=j-1}(\hbar-(p-q)x_{k+j}) q^{b}\,\,\bT_{k}^{[a]}\Sigma_k^n.
\end{equation}
\end{Lemma}

\begin{proof}
This is obvious for $j=0$ since both sides are just $x_k\Sigma_k^n$. 
% When $j=1$, \eqref{HU-11} follows from \eqref{Lebniz2}.
When $j=1$, \eqref{HU-11} follows from \eqref{EQQQ-1}. 
Assume  that     the assertion  is true for $j$. Consider the situation for $j+1$ with $j\geq 1$ and $k+j+1\leq n$. By induction,
\begin{align}
\bT_k^{[j+1]}x_k\Sigma_k^n
&=\bT_{k+j}\bT_k^{[j]}x_k\Sigma_k^n\nonumber\\[5pt]
& =
\bT_{k+j}
\left(x_{k+j}\bT_k^{[j]}\Sigma_k^n+
\sum_{a+b=j-1}(\hbar-(p-q)x_{k+j}) q^{b}\bT_{k}^{[a]}\Sigma_k^n\right).\label{FINAL}
\end{align}
We need the following evaluations. 
% {\color{red}By \eqref{Lebniz2}, we deduce that}
By \eqref{EQQQ-1}, we deduce that
\begin{align}
\bT_{k+j}x_{k+j}\bT_k^{[j]}
% & = \left(x_{k+j+1}T_{k+j}+\hbar\right)\bT_k^{[j]}\nonumber\\[5pt]
& = \left(x_{k+j+1}\bT_{k+j}+(\hbar-(p-q)x_{k+j+1})\right)\bT_k^{[j]}\nonumber\\[5pt]
& = x_{k+j+1}\bT_k^{[j+1]}+(\hbar-(p-q)x_{k+j+1})\bT_k^{[j]}.\label{XIONG-1}
\end{align}
By \eqref{EQQQ-2}, we have 
\begin{align}
&\bT_{k+j}(\hbar-(p-q)x_{k+j}) q^{b}\bT_{k}^{[a]}\Sigma_k^n\nonumber\\[5pt]
&\,= (\hbar-(p-q)x_{k+j+1}) \T_{k+j}q^b\bT_{k}^{[a]}\Sigma_k^n\nonumber\\[5pt]
&\,= (\hbar-(p-q)x_{k+j+1})q^b\big(\bT_{k+j}-(p-q)\big)\bT_{k}^{[a]}\Sigma_k^n.\label{MNB-11}
\end{align}
Since $a\leq j-1$, $\bT_{k+j}$ and $\bT_{k}^{[a]}$ commute. So
\eqref{MNB-11} can be rewritten
as
\begin{equation}\label{MNB-112}
(\hbar-(p-q)x_{k+j+1})q^b\bT_{k}^{[a]}
\big(\bT_{k+j}-(p-q)\big)\Sigma_k^n.
\end{equation}
By the assumption that $j\geq 1$ and $k+j+1\leq n$, we have $s_{k+j}\in \S_k\times \S_{n-k}$. So, by \eqref{Lem-12}, we see that $\bT_{k+j}\Sigma_k^n=p\Sigma_k^n$, and hence \eqref{MNB-112} can be simplified as
\begin{equation}\label{MNB-1123}
(\hbar-(p-q)x_{k+j+1})q^{b+1}\bT_{k}^{[a]}
\Sigma_k^n.
\end{equation}
Putting  \eqref{XIONG-1} and \eqref{MNB-1123} into \eqref{FINAL}, we are given 
\begin{align*}
\bT_k^{[j+1]}x_k\Sigma_k^n
% & =
% \bT_{k+r}
% \bigg(x_{k+r}\bT_k^{[r]}\Sigma_k^m+
% \sum_{a+b=r-1}(\hbar-(p-q)x_{k+r}) q^{b}\bT_{k}^{[a]}\Sigma_k^m\bigg)\\
& = x_{k+j+1}\bT_{k}^{[j+1]}\Sigma_k^n+
\sum_{a+b=j}(\hbar-(p-q)x_{k+j}) q^{b}\bT_{k}^{[a]}\Sigma_k^n.
\end{align*}
This confirms  the assertion  for $j+1$, completing the proof. 
\end{proof}

% \begin{Lemma}\label{PLem3}
% For $a>b$, 
% $$\bT^{[\cdots,b,a,\cdots]}\Sigma_k^n
% =p\,\bT^{[\cdots,a,b-1,\cdots]}\Sigma_k^n.$$
% \end{Lemma}

The following lemma will be used in  the proof of Lemma \ref{PLem3-1}.

\begin{Lemma}\label{PLem3}
For $i\geq 2$ and $a>b$,  
$$\bT_{i-1}^{[a]}\,\bT_i^{[b]}
=\bT_{i-1}^{[b]}\,\bT_{i-1}^{[a]}. $$
\end{Lemma}

\begin{proof}
We complete the proof by induction on $b$. It is obvious when $b=0$. We next consider the case of $b\geq 1$. Notice that
\begin{align}
\bT_{i-1}^{[a]}\,\bT_i^{[b]}&=\left(\bT_{i+b}^{[a-b-1]}\, \bT_{i+b-1}\, \bT_{i+b-2}\, \bT_{i-1}^{[b-1]}\right) \left(\bT_{i+b-1}\, \bT_{i}^{[b-1]} \right)\nonumber\\[5pt]
&=\bT_{i+b}^{[a-b-1]}\, \left(\bT_{i+b-1}\, \bT_{i+b-2}\,\bT_{i+b-1} \right) \, \bT_{i-1}^{[b-1]}\,\bT_{i}^{[b-1]}\nonumber \\[5pt]
&=\bT_{i+b}^{[a-b-1]}\, \left(\bT_{i+b-2}\, \bT_{i+b-1}\,\bT_{i+b-2} \right) \, \bT_{i-1}^{[b-1]}\,\bT_{i}^{[b-1]}\nonumber \\[5pt]
&=\bT_{i+b}^{[a-b-1]}\,\bT_{i+b-2}\,  \bT_{i-1}^{[b+1]}\,\bT_{i}^{[b-1]}.\label{OOO-22}
\end{align}
Applying induction to $\bT_{i-1}^{[b+1]}\,\bT_{i}^{[b-1]}$ gives
\begin{equation}\label{OOO-23}
\bT_{i-1}^{[b+1]}\,\bT_{i}^{[b-1]}=\bT_{i-1}^{[b-1]}\,\bT_{i-1}^{[b+1]}.
\end{equation}
Substituting  \eqref{OOO-23} into \eqref{OOO-22}, we have 
\begin{align}
\bT_{i-1}^{[a]}\,\bT_i^{[b]}
&=\bT_{i+b}^{[a-b-1]}\,\bT_{i+b-2}\,\bT_{i-1}^{[b-1]}\,\bT_{i-1}^{[b+1]}\nonumber\\[5pt]
&=\bT_{i+b}^{[a-b-1]}\,\bT_{i-1}^{[b]}\,\bT_{i-1}^{[b+1]}.\label{OOO-234}
\end{align}
Noticing that $\bT_{i+b}^{[a-b-1]}$ and $\bT_{i-1}^{[b]}$ are commutable, \eqref{OOO-234}
can be rewritten  as
\[
\bT_{i-1}^{[a]}\,\bT_i^{[b]}=\bT_{i-1}^{[b]}\,\bT_{i+b}^{[a-b-1]}\,\bT_{i-1}^{[b+1]}=\bT_{i-1}^{[b]}\,\bT_{i-1}^{[a]},
\]
as required. 
\end{proof}

\begin{Lemma}\label{PLem3-1}
Let  $\lambda\subseteq (n-k)^k$ and $1\leq i\leq k$. For 
  $0\leq a_i<\lambda_i$, 
\begin{equation}\label{TYTY-1}
\bT_{1}^{[\lambda_k]}
\cdots
\bT_{k-i}^{[\lambda_{i+1}]}\,
\bT_{k+1-i}^{[a_i]}\,
\bT_{k+2-i}^{[\lambda_{i-1}]} 
\cdots
\bT_{k}^{[\lambda_{1}]}\Sigma_k^n 
=p^{\htt(\lambda/\nu)-1}
\bT_{\nu},
\end{equation}
where $\nu$ is the partition obtained from $\lambda$ by deleting a ribbon with   head in the $i$-th row and of width $\lambda_i-a_i$. 
\end{Lemma}

\begin{proof}
The proof is divided into two cases.

Case 1.  $a_i\geq \lambda_{i+1}$. In this case,
the sequence obtained from $\lambda$ by replacing $\lambda_i$ with $a_i$ is still a partition in $(n-k)^k$, which is exactly the partition  $\nu$. So the left-hand side of  
\eqref{TYTY-1} is $\bT_\nu$, which equals the right-hand side of \eqref{TYTY-1} since the ribbon  $\lambda/\nu$ has exactly one row.

Case 2.  $a_i<\lambda_{i+1}$.
We see that 
\begin{align}
\bT_{1}^{[\lambda_k]}&
\cdots
\bT_{k-i}^{[\lambda_{i+1}]}\,
\bT_{k+1-i}^{[a_i]}\,
\bT_{k+2-i}^{[\lambda_{i-1}]} 
\cdots
\bT_{k}^{[\lambda_{1}]}\Sigma_k^n\nonumber\\[5pt]
&=\bT_{1}^{[\lambda_k]}
\cdots
\bT_{k-i}^{[a_i]}\,
\bT_{k-i}^{[\lambda_{i+1}]}\,
\bT_{k+2-i}^{[\lambda_{i-1}]} 
\cdots
\bT_{k}^{[\lambda_{1}]}\Sigma_k^n\nonumber\\[5pt]
&=\bT_{1}^{[\lambda_k]}
\cdots
\bT_{k-i}^{[a_i]}\,
\bT_{k+1-i}^{[\lambda_{i+1}-1]}\bT_{k-i}\,
\bT_{k+2-i}^{[\lambda_{i-1}]} 
\cdots
\bT_{k}^{[\lambda_{1}]}\Sigma_k^n\nonumber\\[5pt]
&=\bT_{1}^{[\lambda_k]}
\cdots
\bT_{k-i}^{[a_i]}\,
\bT_{k+1-i}^{[\lambda_{i+1}-1]}\,
\bT_{k+2-i}^{[\lambda_{i-1}]} 
\cdots
\bT_{k}^{[\lambda_{1}]}\,\bT_{k-i}\,\Sigma_k^n\nonumber\\[5pt]
&=p\left(\bT_{1}^{[\lambda_k]}
\cdots
\bT_{k-i}^{[a_i]}\,
\bT_{k+1-i}^{[\lambda_{i+1}-1]}\,
\bT_{k+2-i}^{[\lambda_{i-1}]} 
\cdots
\bT_{k}^{[\lambda_{1}]}\,\Sigma_k^n\right),\label{XGEN-1}
\end{align}
where the first equality is by Lemma \ref{PLem3}, the third equality holds since $\bT_{k-i}$ commutes with $\bT_{k+2-i}^{[\lambda_{i-1}]}, \ldots, \bT_{k}^{[\lambda_{1}]}$, and the last equality follows from  the 
fact $\bT_{k-i}\,\Sigma_k^n=p\Sigma_k^n$ 
according to \eqref{Lem-12}.

We proceed by comparing $a_i$ and $\lambda_{i+2}$. If $a_i\geq \lambda_{i+2}$,
then 
\[(\lambda_1,\ldots, \lambda_{i-1}, \lambda_{i+1}-1, a_i, \lambda_{i+2},\ldots, \lambda_k) 
\]
is exactly the partition $\nu$, and so
\eqref{XGEN-1} equals $p\bT_\nu$, which
is the right-hand side of \eqref{TYTY-1} by noticing that the ribbon $\lambda/\nu$ has two rows. Otherwise, we may continue the same procedure as in \eqref{XGEN-1}. The process will stop until meeting a part, say $\lambda_j$, such that $a_i\geq \lambda_j$ (here we   set $\lambda_{j}=0$ for $j=k+1$). In other words, $j$ is the smallest index satisfying $a_i\geq \lambda_j$. 
The left-hand side of \eqref{TYTY-1} eventually becomes 
\begin{equation}\label{P097}
p^{j-i-1}\left(\bT_{1}^{[\lambda_k]}
\cdots \bT_{k+1-j}^{[\lambda_{j}]}
\bT_{k+2-j}^{[a_i]}
\bT_{k+3-j}^{[{\lambda_{j-1}-1}]}\cdots
\bT_{k+1-i}^{[\lambda_{i+1}-1]}\,
\bT_{k+2-i}^{[\lambda_{i-1}]} 
\cdots
\bT_{k}^{[\lambda_{1}]}\,\Sigma_k^n\right).
\end{equation}
It is easily checked that $\nu$ is the partition
\[
(\lambda_1,\ldots,\lambda_{i-1}, \lambda_{i+1}-1,\ldots,\lambda_{j-1}-1,a_i,\lambda_j\ldots, \lambda_k).
\]
So \eqref{P097} is equal to  $p^{j-i-1} \bT_\nu$. We reach \eqref{TYTY-1} since there are $j-i$ rows in  the ribbon $\lambda/\nu$.
\end{proof}

\section{Proofs of Theorems \ref{ThmA}, \ref{ThmB}, \ref{ThmD} and \ref{ThmE}}\label{sec-66}

In this section, we finish the proofs of the theorems given in Introduction. The descriptions of Theorems \ref{ThmA} and \ref{ThmB}
are in terms of the operators 
$\fhead{i}$ and $\xfhead{i}$. 
We shall first define more types of ribbon Schubert operators, which will be used  for the proof  of  
 Theorem \ref{ThmB}, as well as for deriving Pieri formulas for other classes  in Section \ref{sec67}.

Let us recall that  
\begin{align*} 
\fhead{i}\to \lambda &= 
t_{c}\cdot \lambda+
\sum_{\mu/\lambda=\eta}
\left(\hbar-(p-q)t_{\mathtt{h}(\eta)}\right)p^{\htt(\eta)-1}q^{\wdd(\eta)-1}\cdot\mu,\\[5pt]
\xfhead{i}\to \lambda &= 
t_{c}\cdot \lambda+
\sum_{\mu/\lambda=\eta}
\left(\hbar-(p-q)t_{\mathtt{t}(\eta)}\right)p^{\htt(\eta)-1}q^{\wdd(\eta)-1}\cdot\mu,
\end{align*} 
where $c=\lambda_i+k+1-i$,  $\mathtt{h}(\eta)$ and $\mathtt{t}(\eta)$ are as defined in 
\eqref{POIU-1s}. If the ribbons added to $\lambda$ are required to have tails in row $i$, then we reach the following notation:
\begin{align} 
\xftail{i}\to \lambda &= 
t_{c}\cdot \lambda+
\sum_{\mu/\lambda=\eta}
\left(\hbar-(p-q)t_{\mathtt{h}(\eta)}\right)\,p^{\htt(\eta)-1}q^{\wdd(\eta)-1}\cdot\mu,\\[5pt]
\ftail{i}\to \lambda &= 
t_{c}\cdot \lambda+
\sum_{\mu/\lambda=\eta}
\left(\hbar-(p-q)t_{\mathtt{t}(\eta)}\right)\,p^{\htt(\eta)-1}q^{\wdd(\eta)-1}\cdot\mu,\label{TIETOU-1}
\end{align}
where the statistics $c$, $\mathtt{h}(\eta)$ and $\mathtt{t}(\eta)$ are the same as above, with the exception that now the sum is  over  $\mu\subseteq (n-k)^k$
such that $\mu/\lambda$ is a ribbon with tail in row $i$. We remark that in \eqref{TIETOU-1}, there holds that $c={\mathtt{t}(\eta)}$.

The dual versions of the above four operators are a procedure of deleting ribbons, as already encountered in \eqref{YHNG}. For $\mu\subseteq (n-k)^k$, define 
\begin{align} 
 \mu\to\fhead{i}
&=t_{c}\cdot \mu+
\sum_{\mu/\lambda=\eta}
\left(\hbar-(p-q)t_{\mathtt{h}(\eta)}\right)p^{\htt(\eta)-1}q^{\wdd(\eta)-1}\cdot\lambda, \label{QQ-1}  \\[5pt]
 \mu\to\xfhead{i}
&=t_{c}\cdot \mu+
\sum_{\mu/\lambda=\eta}
\left(\hbar-(p-q)t_{\mathtt{t}(\eta)}\right)p^{\htt(\eta)-1}q^{\wdd(\eta)-1}\cdot\lambda,\label{QQ-2}  \\[5pt]
 \mu\to\xftail{i}
&=t_{c}\cdot \mu+
\sum_{\mu/\lambda=\eta}
\left(\hbar-(p-q)t_{\mathtt{h}(\eta)}\right)p^{\htt(\eta)-1}q^{\wdd(\eta)-1}\cdot\lambda, \label{QQ-3}  \\[5pt]
 \mu\to\ftail{i}
&=t_{c}\cdot \mu+
\sum_{\mu/\lambda=\eta}
\left(\hbar-(p-q)t_{\mathtt{t}(\eta)}\right)p^{\htt(\eta)-1}q^{\wdd(\eta)-1}\cdot\lambda, \label{QQ-4} 
\end{align}
where $c=\mu_i+k+1-i$, and the sum goes   over  $\lambda\subseteq (n-k)^k$
such that in \eqref{QQ-1}  and \eqref{QQ-2} (resp., in \eqref{QQ-3}  and \eqref{QQ-4}),    $\mu/\lambda$ is a ribbon with head (resp., tail) in row $i$. We also remark that  in \eqref{QQ-1}, there holds that $c={\mathtt{h}(\eta)}$.

\begin{Eg}
Let  $k=3$, $n=7$,   $\lambda=(3,1,0)$, and $i=2$. 
\begin{align*}
    \fhead{2}\to \young{&&&u|\\\\|l|}
    &=t_3\cdot\young{&&&u|\\{\Ythicker\wall{r}}||\\l|}
    +
    (\hbar-(p-q)t_4)\cdot\young{&&&u|\\&||*(yellow)\,\,\,2\bullet\\|l|}
    +
    (\hbar-(p-q)t_5)\cdot\young{&&&u|\\&<|*(yellow)&>|*(yellow)\,\,\,2\bullet\\|l|}\\
    &\quad+
    (\hbar-(p-q)t_4)pq\cdot\young{&&&u|\\&|A|*(yellow)\,\,\,2\bullet\\<|*(yellow)&J|*(yellow)}+
    (\hbar-(p-q)t_5)pq^2\cdot\young{&&&u|\\&F|*(yellow)&>|*(yellow)\,\,\,2\bullet\\<|*(yellow)&J|*(yellow)},
\\
    \xfhead{2}\to \young{&&&u|\\\\|l|}
    &=t_3\cdot\young{&&&u|\\{\Ythicker\wall{r}}||\\l|}
    +
    (\hbar-(p-q)t_3)\cdot\young{&&&u|\\&||*(yellow)\bullet 2\,\,\,\\|l|}
    +
    (\hbar-(p-q)t_3)\cdot\young{&&&u|\\&<|*(yellow)\bullet\,\,\,\,\,&>|*(yellow)2\\|l|}\\
    &\quad+
    (\hbar-(p-q)t_1)pq\cdot\young{&&&u|\\&|A|2*(yellow) \\<|*(yellow)\bullet\,\,\,\,\,\,&J|*(yellow)}+
    (\hbar-(p-q)t_1)pq^2\cdot\young{&&&u|\\&F|*(yellow)&>|2*(yellow) \\<|*(yellow)\bullet\,\,\,\,\,\,&J|*(yellow)},
\\
    \xftail{2}\to \young{&&&u|\\&[]\\|l|}
    &=
    t_3\cdot\young{&&&u|\\{\Ythicker\wall{r}}||\\|l|}
    +
    (\hbar-(p-q)t_4)\cdot\young{&&&u|\\&*(yellow)\,\,\,2\bullet\\|l|}\\
    &\quad+
    (\hbar-(p-q)t_5)q\cdot\young{&&&u|\\&<|2*(yellow)&>|*(yellow)\,\,\,\,\,\bullet\\|l|} 
    +
    (\hbar-(p-q)t_7)pq^2\cdot\young{&&&A|*(yellow)\,\,\,\,\,\bullet\\&<|2*(yellow)&=|*(yellow)&J|*(yellow)\\|l|},
\\
    \ftail{2}\to \young{&&&u|\\&[]\\|l|}
    &=
    t_3\cdot\young{&&&u|\\{\Ythicker\wall{r}}||\\|l|}
    +
    (\hbar-(p-q)t_3)\cdot\young{&&&u|\\&*(yellow)\bullet 2\,\,\,\\|l|}\\
    &\quad+
    (\hbar-(p-q)t_3)q\cdot\young{&&&u|\\&<|*(yellow)\bullet 2\,\,\,&>|*(yellow)\\|l|} 
    +
    (\hbar-(p-q)t_3)pq^2\cdot\young{&&&A|*(yellow)\\&<|*(yellow)\bullet 2\,\,\,&=|*(yellow)&J|*(yellow)\\|l|}.
\\[5pt]
    \young{&&&u|\\\\l|}\to \fhead{2}& 
    = t_3\cdot \young{&&&u|\\{\Ythicker\wall{r}}||\\l|}
    +(\hbar-(p-q)t_3)\cdot 
    \young{&&&u|\\F|*(lightgray)\,\,\, 2\bullet\\l|}
\\[5pt]%
    \young{&&&u|\\\\l|}\to \xfhead{2}& 
    = t_3\cdot \young{&&&u|\\{\Ythicker\wall{r}}||\\l|}
    +(\hbar-(p-q)t_2)\cdot 
    \young{&&&u|\\F|*(lightgray)\bullet 2\,\,\,\\l|}
\\[5pt]%
    \young{&&&u|\\\\l|}\to \xftail{2}& 
    = t_3\cdot \young{&&&u|\\{\Ythicker\wall{r}}||\\l|}
    +(\hbar-(p-q)t_3)\cdot 
    \young{&&&u|\\F|*(lightgray)\,\,\,2\bullet\\l|}
    +pq^2(\hbar-(p-q)t_6)\cdot 
    \young{F|*(lightgray)&u|*(lightgray)&u|*(lightgray)\,\,\,\,\,\bullet&u|\\l|2*(lightgray) \\l|}
\\[5pt]%
    \young{&&&u|\\\\l|}\to \ftail{2}& 
    = t_3\cdot \young{&&&u|\\{\Ythicker\wall{r}}||\\l|}
    +(\hbar-(p-q)t_2)\cdot 
    \young{&&&u|\\F|*(lightgray)\bullet 2\,\,\,\\l|}
    +pq^2(\hbar-(p-q)t_2)\cdot 
    \young{F|*(lightgray)&u|*(lightgray)&u|*(lightgray)&u|\\l|*(lightgray)\bullet 2\,\,\,\\l|}
\end{align*}

\end{Eg}

\subsection{Proofs of Theorems \ref{ThmA} and \ref{ThmB}}

Clearly, Theorem \ref{ThmA} is a direct consequence of
Proposition \ref{SchRep:Grr} and Theorem \ref{MainThPieri}. To prove Theorem \ref{ThmB}, we first notice the following Pieri formula for equivariant  motivic Chern classes of Schubert cells.

\begin{Th}\label{ThmBGH}
Set 
 $(p, q, \hbar)=(1, -y, 1+y)$. 
For $\mu\subseteq (n-k)^k$ and $0\leq r\leq k$, the equivariant  motivic Chern classes
for $\MC_y(X(\mu)^\circ)\in K_T(\Gr(k,n))[y]$ satisfies the following Pieri formula:
\begin{align}\label{ThmA-PPP-UU}
c_r(\mathcal{V}^\vee)\cdot \MC_y(X(\mu)^\circ)
=\sum_{1\leq i_1<\cdots<i_r\leq k}
\MC_y(X(\mu)^\circ)\to \ftail{i_r}\to\cdots\to\ftail{i_1}.
\end{align}
\end{Th}

\begin{proof}
By Proposition   \ref{SchRep4:Gr}(iii), we have $\MC_y(X(\mu)^\circ)=w_0^L \MC_y(Y(\overline{\mu})^\circ).$
In view of the Weyl action in  \eqref{wyleaCT}, it is easy to check that for $1\leq i\leq k$,
\[
\MC_y(X(\mu)^\circ)\to \ftail{i}=w_0^L \left(\fhead{k+1-i}\to\MC_y(Y(\overline{\mu})^\circ)\right).
\]
Therefore,
\begin{align*}
&\sum_{1\leq i_1<\cdots<i_r\leq k}
\MC_y(X(\mu)^\circ)\to \ftail{i_r}\to\cdots\to\ftail{i_1}\\[5pt]
&\qquad =  w_0^L  \sum_{1\leq i_1<\cdots<i_r\leq k}\fhead{k+1-i_1}\to\cdots\to\fhead{k+1-i_r}\to\MC_y(Y(\overline{\mu})^\circ)\\[5pt]
&\qquad =  w_0^L  \sum_{1\leq j_1<\cdots<j_r\leq k}\fhead{j_r}\to\cdots\to\fhead{j_1}\to\MC_y(Y(\overline{\mu})^\circ),
\end{align*}
which, along with Theorem \ref{ThmA}, becomes 
\[
w_0^L \left(c_r(\mathcal{V}^\vee)\cdot \MC_y(Y(\overline{\mu})^\circ)\right)=c_r(\mathcal{V}^\vee)\cdot \MC_y(X(\mu)^\circ).
\]
This completes the proof of \eqref{ThmA-PPP-UU}.
\end{proof}

Based on Theorem \ref{ThmBGH}, we obtain a Pieri formula for equivariant Segre motivic classes of opposite Schubert cells, which is formulated   in terms of the operator $\ftail{i}$. 

\begin{Th}\label{ThmBBB}
Set 
 $(p, q, \hbar)=(1, -y, 1+y)$. For $\lambda\subseteq (n-k)^k$ and $0\leq r\leq k$,
the equivariant  Segre motivic  class $\SMC_y(Y(\lambda)^\circ)\in \mathbb{K}_T(\Gr(k,n))$ satisfies the following Pieri formula:
\begin{align}\label{ThmB-QQQWQQ}
c_r(\mathcal{V}^\vee)\cdot \SMC_y(Y(\lambda)^\circ)
=\sum_{1\leq i_1<\cdots<i_r\leq k}
\ftail{i_r}\to \cdots \to \ftail{i_1}\to 
\SMC_y(Y(\lambda)^\circ).
\end{align}
\end{Th}

\begin{proof}
The proof of \eqref{ThmB-QQQWQQ} is equivalent to showing that 
the coefficient of  $\SMC_y(Y(\mu)^\circ)$  in the expansion of $c_r(\mathcal{V}^\vee)\cdot \SMC_y(Y(\lambda)^\circ)$ is equal to the coefficient of $\MC_y(X(\lambda)^\circ)$ in the expansion of $c_r(\mathcal{V}^\vee)\cdot \MC_y(X(\mu)^\circ)$.
This can be seen as follows.
By  Proposition   \ref{SchRep4:Gr}(ii), 
$\MC_y(X(\lambda)^\circ)$ is dual to $\SMC_y(Y(\lambda)^\circ)$ under the Poincar\'e pairing.
So  the coefficient of  $\SMC_y(Y(\mu)^\circ)$  in  $c_r(\mathcal{V}^\vee)\cdot \SMC_y(Y(\lambda)^\circ)$ is 
\[
\left<\MC_y(X(\mu)^\circ),c_r(\mathcal{V}^\vee)\cdot \SMC_y(Y(\lambda)^\circ) \right>=\left<c_r(\mathcal{V}^\vee)\cdot \MC_y(X(\mu)^\circ),\SMC_y(Y(\lambda)^\circ) \right>,
\]
which is just the coefficient of $\MC_y(X(\lambda)^\circ)$ in  $c_r(\mathcal{V}^\vee)\cdot \MC_y(X(\mu)^\circ)$.
\end{proof}

We can now present a proof of Theorem \ref{ThmB}.

\begin{proof}[Proof of Theorem \ref{ThmB}]
Using \eqref{OMJUE}  in the appendix, we see that \eqref{ThmB-QQQWQQ} can be described equivalently  in terms of the operator $\xfhead{i}$, which is exactly what we need in 
Theorem \ref{ThmB}.
\end{proof}

\begin{Eg}
Let $n=5$, $k=3$, $r=2$, and $\lambda=(1,1,0)$.
We compute 
$c_2(\mathcal{V}^\vee)\cdot \SMC_y(Y(1,1,0)^\circ).$
The answer will be given by 
$$
\big(\xfhead{2}\to\xfhead{1}\to (1,1,0)\big)
+\big(\xfhead{3}\to\xfhead{1}\to (1,1,0)\big)
+\big(\xfhead{3}\to\xfhead{2}\to (1,1,0)\big). $$
We use the abbreviation $S_\lambda=\SMC_y(Y(\lambda)^\circ)$. Then we have
\begin{align*}
c_2(\mathcal{V}^\vee)\cdot   S_{(1,1,0)} & 
=(t_1t_3+t_1t_4+t_3t_4)\cdot S_{(1,1,0)}+(1+y)(t_1+t_3)(1-t_4)\cdot S_{(2,1,0)}\\
&\quad+ (1+y)(t_3+t_4)(1-t_1)\cdot S_{(1,1,1)}+(1+y)^2(1-t_1)(1-t_4)\cdot S_{(2,1,1)}\\
&\quad
+\big[(1+y)^2(1-t_3)(1-t_4)+(1+y)(1-t_3)(t_1+t_4)\big]\cdot S_{(2,2,0)}\\
&\quad +(1+y)^2(1-t_1)(1-t_3)\cdot S_{(2,2,1)}+\big[
-y(1+y)(1-t_1)(t_3+t_4)\big.
\\
&\quad 
\big.-y(1+y)^2(1-t_1)(1-t_4)-y(1+y)^2(1-t_1)(1-t_3)
\big]\cdot S_{(2,2,2)}.
\end{align*}
The diagram illustration is as follows.  
%-----------------------------
% diagram via *|i].
%-----------------------------
$$
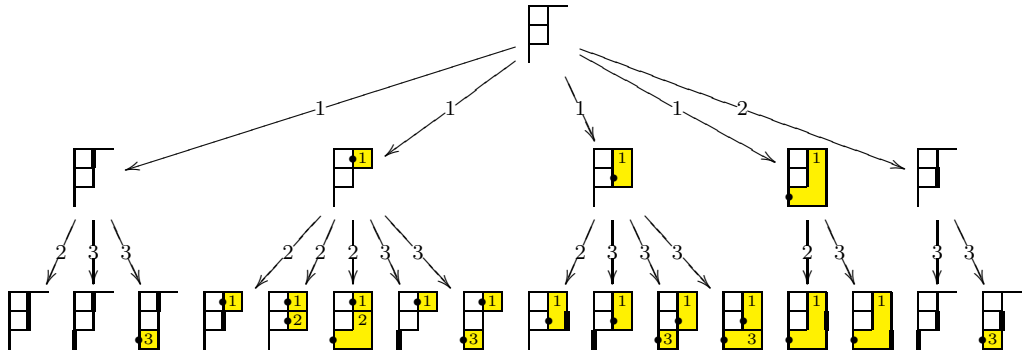
\xymatrix@C=0pc{
&&&&&&&&
\young{&u|\\\\l|}
    \ar[dlllllll]|{{1}}
    \ar[dlll]|{{1}}
    \ar[dr]|{{1}}
    \ar[drrrr]|{{1}}
    \ar[drrrrrr]|{{2}}
&&&&&\\
&
\young{{\Ythicker\wall{r}}||&u|\\\\l|}
    \ar[dl]|{{2}} 
    \ar[d]|{{3}} 
    \ar[dr]|{{3}} 
&&&&
\young{&*(yellow)\bullet1\,\,\,\\\\l|}
    \ar[dll]|{{2}} 
    \ar[dl]|{{2}} 
    \ar[d]|{{2}} 
    \ar[dr]|{{3}} 
    \ar[drr]|{{3}} 
&&&&
\young{&A|*(yellow)1\\&V|*(yellow)\bullet\,\,\,\,\,\\l|}
    \ar[dl]|{{2}} 
    \ar[d]|{{3}} 
    \ar[dr]|{{3}} 
    \ar[drr]|{{3}} 
&&&
\young{&A|*(yellow)1\\&H|*(yellow)\\<|*(yellow)\bullet\,\,\,\,\,&J|*(yellow)}
    \ar[d]|{{2}} 
    \ar[dr]|{{3}} 
&&
\young{&u|\\{\Ythicker\wall{r}}||\\l|}
    \ar[d]|{{3}} 
    \ar[dr]|{{3}} 
\\
\young{
    {\Ythicker\wall{r}}||&u|\\
    {\Ythicker\wall{r}}||\\l|} 
&\young{
    {\Ythicker\wall{r}}||&u|\\\\
    {\Ythicker\wall{l}}|l|}
&\young{
    {\Ythicker\wall{r}}||&u|\\\\
    *(yellow)\bullet3\,\,\,}
&\young{
    &*(yellow)\bullet1\,\,\,\\{\Ythicker\wall{r}}||\\l|}
&\young{
    &*(yellow)\bullet1\,\,\,\\
    &*(yellow)\bullet2\,\,\,\\l|}
&\young{
    &*(yellow)\bullet1\,\,\,\\&A|*(yellow)2\\
    <|*(yellow)\bullet\,\,\,\,\,\,&J|*(yellow)}
&\young{
    &*(yellow)\bullet1\,\,\,\\\\
    {\Ythicker\wall{l}}|l|}
&\young{
    &*(yellow)\bullet1\,\,\,\\\\
    *(yellow)\bullet3\,\,\,}
&\young{
    &A|*(yellow)1\\
    &{\Ythicker\wall{r}}|V|*(yellow)\bullet\,\,\,\,\,\\l|}
&\young{
    &A|*(yellow)1\\
    &V|*(yellow)\bullet\,\,\,\,\,\\
    {\Ythicker\wall{l}}|l|}
&\young{
    &A|*(yellow)1\\
    &V|*(yellow)\bullet\,\,\,\,\,\\
    *(yellow)\bullet3\,\,\,}
&\young{
    &A|*(yellow)1\\
    &V|*(yellow)\bullet\,\,\,\,\,\\
    <|*(yellow)\bullet\,\,\,\,\,&>|*(yellow)3}
&\young{
    &A|*(yellow)1\\&{\Ythicker\wall{r}}|H|*(yellow)\\
    <|*(yellow)\bullet\,\,\,\,\,&J|*(yellow)}
&\young{
    &A|*(yellow)1\\&H|*(yellow)\\
    <|*(yellow)\bullet\,\,\,\,\,&{\Ythicker\wall{r}}|J|*(yellow)}
&\young{
    &u|\\{\Ythicker\wall{r}}||\\{\Ythicker\wall{l}}|l|}
&\young{
    &u|\\{\Ythicker\wall{r}}||\\
    *(yellow)\bullet3\,\,\,}
}$$
%-----------------------------
% diagram via *|i].
%-----------------------------
\end{Eg}

\subsection{Proof of Theorem   \ref{ThmD}}\label{sec:relMCSMC}
% {Relation between MC and SMC}

For  $\lambda\subseteq (n-k)^k$,
let   
$$\cdot|_{\lambda}\colon K_T(\Gr(k,n))\longrightarrow K_T(\pt)$$
be the restriction  map at the fixed point $w_{\lambda} P/P\in G/P=\Gr(k,n)$. 
Denote 
\begin{align*}
\mathcal{D}=1-c_1(\mathcal{V}^\vee)+\cdots+(-1)^k c_k(\mathcal{V}^\vee)=[\det \mathcal{V}]\in K_T(\Gr(k,n)).
\end{align*}
As $\pi^* \mathcal{D}=(1-x_1)\cdots (1-x_k)$, one has
\begin{align*}
\mathcal{D}_{\lambda} 
:=\mathcal{D}|_{\lambda}
=(1-t_{w_{\lambda}(1)})\cdots
(1-t_{w_{\lambda}(k)})\in K_T(\pt).
\end{align*} 
In particular,   $\mathcal{D}_{\varnothing}=(1-t_1)\cdots(1-t_k).$
Note that both $\mathcal{D}$ and $\mathcal{D}_{\lambda}$ are units in $K_T(\Gr(k,n))$. 
By \cite{LS-2}, the following relations hold:
\begin{equation}\label{eq:Gsquare=}
    \mathcal{D}=\mathcal{D}_{\varnothing}\cdot (1-[\mathcal{O}_{Y(\square)}]),
    \qquad 
\mathcal{D}_{\lambda}=\mathcal{D}_{\varnothing}\cdot (1-[\mathcal{O}_{Y(\square)}]|_{\lambda}).
\end{equation}
With  \eqref{eq:Gsquare=}, Theorem \ref{ThmD} is equivalent to the following statement. 

\begin{Th}\label{th:relMCSMC}
In $\mathbb{K}_T(\Gr(k,n))$, we have 
\begin{equation*}%\label{eq:relMCSMC}
\lambda_y(\mathscr{T}^\vee_{\Gr(k,n)})\cdot \mathcal{D}\cdot \SMC_y(Y(\lambda)^\circ)=
\mathcal{D}_{\lambda}\cdot \MC_y(Y(\lambda)^\circ).
\end{equation*}
\end{Th}

\begin{proof}
Define a $\mathbb{K}_T(\mathsf{pt})$-linear map $\varphi\colon \mathbb{K}_T(\Gr(k,n))\to 
\mathbb{K}_T(\Gr(k,n))$ by 
$$
\varphi\big(\SMC_y(Y(\lambda)^\circ)\big)= 
\mathcal{D}_{\lambda}\cdot \MC_y(Y(\lambda)^\circ). $$
We claim that for any $\beta,\gamma\in \mathbb{K}_T(\Gr(k,n))$, 
\begin{equation}\label{eq:modmorphi}
    \beta \cdot \varphi(\gamma)=\varphi(\beta\cdot\gamma),
\end{equation}
that is, $\varphi$ is a $\mathbb{K}_T(\Gr(k,n))$-module homomorphism. 
It suffices to show that
\begin{equation}\label{eq:modmorphicheck}
\varphi\big(c_r(\mathcal{V}^\vee) \cdot\SMC_y(Y(\lambda)^\circ)\big)
=
c_r(\mathcal{V}^\vee) \cdot \varphi\big(\SMC_y(Y(\lambda)^\circ)\big).
\end{equation}
Notice that for a ribbon $\eta=\mu/\lambda \subseteq (n-k)^k$, it is direct to check that
$$(1-t_{\mathtt{t}(\eta)})\cdot \mathcal{D}_{\mu} 
=(1-t_{\mathtt{h}(\eta)})\cdot \mathcal{D}_{\lambda}.  $$
Thus the factor $\mathcal{D}_{\lambda}$ intertwines the operators $\ftail{i}$ and $\fhead{i}$ (with the setting $(p,q,\hbar)=(1,-y,y+1)$). Specifically, 
$$L\circ (\ftail{i}\to) = (\fhead{i}\to) \circ L,$$
where  $L$  is the linear operator sending $\nu$ to $\mathcal{D}_{\nu}\cdot \nu$ for any $\nu \subseteq (n-k)^k$.
% $$L(\ftail{i}\to\lambda) = \fhead{i}\to  (L\lambda).$$
As a result, 
\begin{equation}\label{eq:intertwine}
    L\circ \left(\sum_{1\leq i_1<\ldots<i_r\leq k}\ftail{i_r}\to\cdots\to \ftail{i_1}\to\right) = \left(\sum_{1\leq i_1<\ldots<i_r\leq k}\fhead{i_r}\to\cdots \to\fhead{i_1}\to\right)  \circ L.
\end{equation}
Now we have the following commutative diagram. 
$$
\xymatrix{
\displaystyle\bigoplus_{\lambda\subseteq (n-k)^k}\mathbb{K}_T(\pt)\cdot \lambda
    \ar[ddd]_L
    \ar[rrr]^{{\footnotesize\displaystyle\sum_{1\leq i_1<\cdots<i_r\leq k}\ftail{i_r}\to\cdots\to \ftail{i_1}\to}}
    \ar@{..>}[dr]|-{\lambda\mapsto \SMC_y(Y(\lambda)^\circ)}
&&&
\displaystyle\bigoplus_{\lambda\subseteq (n-k)^k}\mathbb{K}_T(\pt)\cdot \lambda
    \ar[ddd]^{L}
    \ar@{..>}[dl]|-{\lambda\mapsto \SMC_y(Y(\lambda)^\circ)}\\
&\mathbb{K}_T(\Gr(k,n))
    \ar[d]_{\varphi}
    \ar[r]^{c_r(\mathcal{V}^\vee)} &
\mathbb{K}_T(\Gr(k,n))
    \ar[d]^{\varphi} \\
&\mathbb{K}_T(\Gr(k,n))
    \ar[r]_{c_r(\mathcal{V}^\vee)}&
    \mathbb{K}_T(\Gr(k,n))\\
\displaystyle\bigoplus_{\lambda\subseteq (n-k)^k}\mathbb{K}_T(\pt)\cdot \lambda
    \ar[rrr]_{{\footnotesize\displaystyle\sum_{1\leq i_1<\cdots<i_r\leq k}\fhead{i_r}\to\cdots \to\fhead{i_1}\to}}
    \ar@{..>}[ur]|-{\lambda\mapsto \MC_y(Y(\lambda)^\circ)}
&&&
\displaystyle\bigoplus_{\lambda\subseteq (n-k)^k}\mathbb{K}_T(\pt)\cdot \lambda
    \ar@{..>}[ul]|-{\lambda\mapsto \MC_y(Y(\lambda)^\circ)}
}$$
Consider the diagram as a cube with six faces. 
The commutativity of the front face (that is, the face with long solid arrows) follows from \eqref{eq:intertwine}, while the left and right faces commute by the definition of $\varphi$. The upper and lower faces commute by Theorem \ref{ThmA} and Theorem \ref{ThmB}.
Since all dashed arrows are isomorphisms, the back face also commutes, which implies \eqref{eq:modmorphicheck}. 
This proves the claim in \eqref{eq:modmorphi}.

% So $\varphi\big(\SMC_y(Y(\lambda)^\circ)\big)=\mathcal{D}_{\lambda}\cdot \MC_y(Y(\lambda)^\circ)$ also satisfies the {\color{red} tail-valued Pieri rule (we did not define this terminology)}. So \eqref{eq:modmorphicheck} holds and the claim follows. ({\color{red} why?})

Now setting  $\beta=\SMC_y(Y(\lambda)^\circ$ and $\gamma=1$  in \eqref{eq:modmorphi}, we obtain
\begin{equation}\label{1qacd}
\varphi(1)\cdot \SMC_y(Y(\lambda)^\circ)=\varphi(\SMC_y(Y(\lambda)^\circ).    
\end{equation}
Let  $\alpha\in \mathbb{K}_T(\Gr(k,n))$ be such that 
$$\varphi(1)=\alpha\cdot \lambda_{y}(\mathscr{T}_{\Gr(k,n)})\cdot \mathcal{D}.$$
Then \eqref{1qacd} becomes
\begin{equation}\label{eq:charalpha}
\alpha \cdot \lambda_{y}(\mathscr{T}_{\Gr(k,n)})\cdot \mathcal{D}\cdot \SMC_y(Y(\lambda)^\circ)= \mathcal{D}_{\lambda}\cdot \MC_y(Y(\lambda)^\circ),
\end{equation}
for any $\lambda$.
To finish the proof, it is enough to check that $\alpha=1$. 
By Lemma \ref{lem:locidtt} below, we see  that $\alpha|_{\lambda}=1$ for any $\lambda\subseteq (n-k)^k$, and hence $\alpha=1$. 
\end{proof}

% \begin{proof}{\color{red} to polish}
% Since $\MC_y(Y(\lambda)^\circ)$ satisfies 
% the head-valued Pieri rule, it is direct to check
% the class $D_{\lambda}\cdot \MC_y(Y(\lambda)^\circ)$ satisfies 
% the tail-valued Pieri rule. 
% Thus, there exists an $\alpha\in \mathbb{K}_T(\Gr(k,n))$ such that 
% $$\mathcal{D}_{\lambda}\cdot \MC_y(Y(\lambda)^\circ)
% =\alpha \cdot \SMC_y(Y(\lambda)^\circ).$$
% To be exact, the $\mathbb{K}_T(\pt)$-linear map sending $\SMC_y(Y(\lambda)^\circ)$ to $D_{\lambda}\cdot \MC_y(Y(\lambda)^\circ)$ is an isomorphism of left $\mathbb{K}_T(\Gr(k,n))$-modules between the regular module $\mathbb{K}_T(\Gr(k,n))$ itself. 
% Thus, the map must be given by multiplication of a unit in the ring $\mathbb{K}_T(\Gr(k,n))$. 
% Let us denote $\alpha_1\in \mathbb{K}_T(\Gr(k,n))$ such that 
% $$\alpha=\alpha_1\cdot \lambda_{y}(\mathscr{T}_{\Gr(k,n)})\cdot \mathcal{D}.$$
% That is
% $$\mathcal{D}_{\lambda}\cdot \MC_y(Y(\lambda)^\circ)
% = \alpha_1 \cdot \lambda_{y}(\mathscr{T}_{\Gr(k,n)})\cdot \mathcal{D}\cdot \SMC_y(Y(\lambda)^\circ).$$
% It suffices to show $\alpha_1=1$. 
% By the lemma below, we see immediate that $\alpha_1|_{\lambda}=1$ for all $\lambda$. 
% This proves $\alpha_1=1$. 
% \end{proof}

\begin{Lemma}\label{lem:locidtt}
For   $\lambda\subseteq (n-k)^k$, we have
$$\MC_y(Y(\lambda)^\circ)|_{\lambda}
=\lambda_y(\mathscr{T}^\vee_{\Gr(k,n)})\big|_{\lambda}\cdot \SMC_y(Y(\lambda)^\circ)|_{\lambda}\neq 0.$$
\end{Lemma}
\begin{proof}
%Let us multiply $\MC_y(X(\lambda)^\circ)|_{\lambda}$ on both sides, and check they are identical and nonzero.
Since $\langle 
\MC_y(X(\lambda)^\circ),\SMC_y(Y(\lambda)^\circ)\rangle=1$, 
%the product $\MC_y(X(\lambda)^\circ)\cdot 
%\SMC_y(Y(\lambda)^\circ)$ must be the class of structure sheaf of the fixed point $w_{\lambda}P/P$. So
the localization formula gives
$$\MC_y(X(\lambda)^\circ)|_{\lambda}\cdot 
\SMC_y(Y(\lambda)^\circ)|_{\lambda}
=\lambda_{-1}(\mathscr{T}^\vee_{\Gr(k,n)})\big|_{\lambda}\neq 0.$$
On the other hand, by \cite[Lemma 9.1(b)]{AMSS19} and the transversality of the intersection of $X(\lambda)$ and $Y(\lambda)$ at the fixed point $w_\lambda P\in \Gr(k,n)$, we get
% by the explicit formulas for $\MC_y(Y(\lambda)^\circ)|_{\lambda}$ and 
% $\MC_y(X(\lambda)^\circ)|_{\lambda}$ given  in \cite{AMSS19,SZZ}, we obtain  that 
$$\MC_y(Y(\lambda)^\circ)|_{\lambda}\cdot
\MC_y(X(\lambda)^\circ)|_{\lambda}
=
\lambda_{-1}(\mathscr{T}^\vee_{\Gr(k,n)})\big|_{\lambda}\cdot \lambda_{y}(\mathscr{T}^\vee_{\Gr(k,n)})\big|_{\lambda}.$$
This finishes the proof of the Lemma.
%Combining the above proves the lemma. 
\end{proof}

% \begin{proof}
% Let us multiply $\SMC_y(X(\lambda)^\circ)|_{\lambda}$ on both sides, and check they are identical and nonzero. 
% Since $\langle 
% \MC_y(Y(\lambda)^\circ),\SMC_y(X(\lambda)^\circ)\rangle=1$, we have 
% $$\MC_y(Y(\lambda)^\circ)|_{\lambda}\cdot 
% \SMC_y(X(\lambda)^\circ)|_{\lambda}
% =\lambda_{-1}(\mathscr{T}^\vee_{\Gr(k,n)})\big|_{\lambda}\neq 0.$$
% There is an explicit formula for $\SMC_y(Y(\lambda)^\circ)|_{\lambda}$ and 
% $\SMC_y(X(\lambda)^\circ)|_{\lambda}$, see \cite{AMSS19,SZZ}, which gives
% $$\SMC_y(Y(\lambda)^\circ)|_{\lambda}\cdot
% \SMC_y(X(\lambda)^\circ)|_{\lambda}
% =\frac{\lambda_{-1}(\mathscr{T}^\vee_{\Gr(k,n)})\big|_{\lambda}}{\lambda_{y}(\mathscr{T}^\vee_{\Gr(k,n)})\big|_{\lambda}}.$$
% This proves the Lemma. 
% \end{proof}

\begin{Rmk}
Lemma \ref{lem:locidtt} holds for any partial flag varieties with the same proof, however, the existence of the class $\alpha$ not depending on $\lambda$ in \eqref{eq:charalpha} would no longer be ensured in general. 
\end{Rmk}

\begin{Eg}For $\Gr(1,2)=\mathbb{P}^1$, we have 
$$K(\Gr(1,2))=\mathbb{Q}[x]/(x^2),\qquad x=1-[\mathcal{O}(-1)].$$
Note that 
$$\lambda_y(\mathscr{T}^\vee_{\mathbb{P}^1})
=1+y[\mathcal{O}(-2)]
=1+y(1-x)^2=(1+y)-2yx\mod (x^2).$$
Moreover, 
$$
\begin{array}{l@{\quad}l@{\quad}l}
{}[\mathcal{O}_{Y(\varnothing)}] = 1,&
{}\MC_y(Y(\varnothing)^\circ) = (1+y)-(2y+1)x,&
{}\SMC_y(Y(\varnothing)^\circ) = 1+\dfrac{y}{1+y}x, \\
{}[\mathcal{O}_{Y(\square)}] = x,&
{}\MC_y(Y(\square)^\circ) = x, &
{}\SMC_y(Y(\square)^\circ) = \dfrac{1}{1+y}x.
\end{array}$$
% Thus 
% $$\lambda_y(\mathscr{T}_{\mathbb{P}^1}^\vee)\cdot (1-[\mathcal{O}_{Y(\square)}])=((1+y)-2yx)(1-x)
% =(1+y)-(1+3y)x.
% $$
Hence we have  
\begin{align*}
\lambda_y(\mathscr{T}_{\mathbb{P}^1}^\vee)\cdot (1-[\mathcal{O}_{Y(\square)}])\cdot \SMC_y(Y(\varnothing)^\circ) 
& =\big((1+y)-2yx\big)\cdot (1-x)\cdot \left(1+\frac{y}{1+y}x\right)\\
& = (1+y)-(2y+1)x\, \bmod (x^2)\\
& = \MC_y(Y(\varnothing)^\circ),\\[5pt]
\lambda_y(\mathscr{T}_{\mathbb{P}^1}^\vee)\cdot (1-[\mathcal{O}_{Y(\square)}])\cdot \SMC_y(Y(\square)^\circ) 
& =\big((1+y)-2yx\big)\cdot (1-x)\cdot \frac{1}{1+y}x\\
& = x \,\bmod (x^2)\\
& = \MC_y(Y(\square)^\circ).
\end{align*}
\end{Eg}

\subsection{Proof  of Theorem \ref{ThmE}}\label{sec:dualsheaf}
We start with some terminology. 
The ring    $\Lambda=\bigoplus_\lambda \mathbb{Q}\cdot s_\lambda(x)$  of symmetric functions is the linear span of   Schur functions  $s_\lambda(x)$. 
The completion of $\Lambda$, denoted  $\hat{\Lambda}$, consists of all infinite linear combinations of Schur functions, or equivalently, 
$$\hat{\Lambda}=\mathbb{Q}[[e_1(x),e_2(x),\ldots]]=\mathbb{Q}[[h_1(x),h_2(x),\ldots]],$$
where
$$e_r(x) =s_{(1^r)}(x)=\sum_{1\leq i_1<\cdots<i_r} x_{i_1}\cdots x_{i_r},\qquad 
h_r(x)=s_{(r)}(x) =\sum_{1\leq i_1\leq \cdots\leq i_r} x_{i_1}\cdots x_{i_r},$$
are respectively the elementary and complete symmetric functions.  

The stable Grothendieck polynomial $G_{\lambda}(x)\in \hat{\Lambda}$  
is the stable limit of the single Grothendieck polynomial $\mathfrak{G}_{w_{\lambda}}(x,0)$. 
%For $\lambda\subseteq (n-k)^k$, 
%when  putting all variables $x_i$ with $i>k$ equal to zero, 
%the polynomial $G_{\lambda}(x_1,\ldots,x_k)$ equals the single Grothendieck polynomial $\mathfrak{G}_{w_{\lambda}}(x,0)$.
Buch \cite{Buch} showed that $G_\lambda(x)$ can be interpreted as the generating function of set-valued tableaux of shape $\lambda$. For two finite nonempty  subsets  $A$ and $B$ of positive integers, write $A < B$ if $\max A < \min B$, and $A \le B$ if $\max A \le\min B$.
A  \emph{set-valued  tableau} of shape $\lambda$ is
a filling of the boxes of $\lambda$ 
with finite nonempty  {subsets} of positive integers such that the subsets are weakly increasing along each row, and are   strictly increasing along each column. 
For a set-valued tableau $T$,
write $|T|$ for the total number of integers   in $T$, and  $x^T=\prod_i x_i^{a_i}$ where $a_i$ is the number of occurrences of $i$ in $T$.
 Buch's formula for $G_\lambda(x)$ is
\begin{equation}\label{POREQ01}
G_\lambda(x)=\sum_T (-1)^{|T|-|\lambda|} x^T,    
\end{equation}
where the sum is over all set-valued tableaux of shape $\lambda$. 

%Let $\Gamma=\bigoplus_\lambda \mathbb{Q}\cdot G_\lambda(x)$ be the ring spanned by all $G_\lambda(x)$, and $I$ be the ideal   generated by   basis elements  $G_\lambda(x)$ for partitions that do not fit in $(n-k)^k$. It follows from \cite[Theorem 8.1]{Buch} that the map $G_\lambda(x)\rightarrow [\mathcal{O}_{Y(\lambda)}]$ 
%induces an isomorphism of rings $\Gamma/I \simeq K(\Gr(k,n))$.

Let $\tilde{K}_{\lambda}(x)$ be obtained   from  $G_\lambda(x)$ by ignoring the sign, that is, \[\tilde{K}_\lambda(x)=\sum_T x^T,\]
where the sum is taken over all set-valued tableaux of shape $\lambda$.
As usual, denote by $\omega\colon {\Lambda}\to {\Lambda}$  the ring involution which sends $e_r(x)$ to $h_r(x)$  for  $r\geq 0$ \cite[Chapter 7]{Stanley2}. Denote $J_{\lambda}(x)=\omega \tilde{K}_{\lambda}(x)$. 
Lam and  Pylyavskyy  \cite[Theorem 9.22]{LP}  proved that $J_{\lambda}(x)$ is the generating function of  weak set-valued tableaux of shape $\lambda$.
A  \emph{weak set-valued tableau}   of shape $\lambda$ is a filling of the boxes of $\lambda$
with finite nonempty \emph{multisets} of positive integers (this means    repeated integers are allowed in each box), with now however the restriction that the   rows are strictly increasing and the   columns are weakly increasing.  
In the following figure, the left one is  a set-valued tableau, while the   right one is a weak set-valued tableau.
$$
\Young[1.5pc]{1&123&35&6\\
234&46\\
5}\qquad 
\Young[1.5pc]{11&334&55&6\\
12&4\\
223}$$

We   turn to the geometric side. 
Since   $Y(\lambda)$ is Cohen--Macaulay \cite[8.22(e)]{KK},   it admits a dualizing sheaf $\omega_{Y(\lambda)}$. 
By \cite[Theorem 5.1]{AMSS22}, one knows that 
\begin{equation}\label{YTRGPOU}
\MC_y(Y(\lambda)^\circ)
=
y^{k(n-k)-|\lambda|}[\omega_{Y(\lambda)}]+\text{classes of lower $y$-degree}.    
\end{equation}
That is, 
$$[\omega_{Y(\lambda)}]=\lim_{y\to \infty}
y^{|\lambda|-k(n-k)}\MC_y(Y(\lambda)^\circ).$$

By the Pieri formula for motivic Chern classes and \eqref{YTRGPOU}, we obtain the following Pieri formula for $[\omega_{Y(\lambda)}]$, which is the key observation  that directs us towards  Theorem \ref{ThmE}.  

\begin{Coro}\label{I87R43}
Let $\lambda\subseteq (n-k)^k$ and  $0\leq r\leq k$. Over $K(\Gr(k,n))$, we have 
$$c_r(\mathcal{V}^\vee)\cdot 
[\omega_{Y(\lambda)}]
=\sum_{\mu}
(-1)^{|\mu/\lambda|-r}\cdot [\omega_{Y(\mu)}],$$
where the sun is over  $\mu\subseteq (n-k)^k$ such that $\mu/\lambda$ is a skew shape with exactly $r$ nonempty rows. 
\end{Coro}

\begin{proof}
By Corollary \ref{coroC}, the Pieri formula for  non-equivariant motivic Chern classes is 
\begin{align}\label{ThmAIJY}
c_r(\mathcal{V}^\vee)\cdot \MC_y(Y(\lambda)^\circ)
=\sum_{1\leq i_1<\cdots<i_r\leq k}
\hhead{i_r}\to \cdots \to \hhead{i_1}\to 
\MC_y(Y(\lambda)^\circ),
\end{align}
where
 \begin{equation*} 
\hhead{i}\to \lambda = 
(1+y)\sum_{\mu/\lambda=\eta}
(-y)^{\wdd(\eta)-1}\cdot\mu.
\end{equation*}   
Take the coefficient of $y^{k(n-k)-|\lambda|}$ on both sides of \eqref{ThmAIJY}. By \eqref{YTRGPOU}, the left-hand side is $c_r(\mathcal{V}^\vee)\cdot [\omega_{Y(\lambda)}]$. To reach the highest degree of $y$ on the right-hand side, the  ribbons added  in each step must be of height one. As a result, the coefficient of  $y^{k(n-k)-|\lambda|}$  contains  the terms $[\omega_{Y(\mu)}]$ (with sign $(-1)^{|\mu/\lambda|-r}$) for partitions  that are obtained from $\lambda$ by adding consecutively $r$ single nonempty rows  from up to down.
This concludes the proof. 
\end{proof}

\begin{Rmk}
Notice that the Pieri formula in Corollary \ref{I87R43} can be described in term of the ribbon Schubert operator in \eqref{POIU-1s} by setting $t=0$ and $(p,q,\hbar)=(0,-1,1)$.   
\end{Rmk}

% {\color{red}
% Recall that 
% $$[\omega_{\Gr(k,n)}]=[\det(\mathcal{V})^{\otimes n}]
% =\big(1-c_1(\mathcal{V}^\vee)+c_2(\mathcal{V}^\vee)-\cdots\big)^n.$$
% Since $G_{\square} = e_1(x)-e_2(x)+\cdots $, we have
% $$[\omega_{\Gr(k,n)}]=\rho\big((1-G_{\square}(x))^n\big)$$
% }

%{\color{red}
%$(p,q,\hbar)=(0,-1,1)$ 
%}

%It is not hard to see from our Pieri rule that
%$[\omega_{Y(\lambda)}]$ satisfies the head-valued Pieri rule for $(p,q,\hbar)=(0,1,-1)$. 
%Note that $p=0$ implies that we are restricted to horizontal strips.  
%When restricting to the non-equivariant K-theory, this rule gives a geometric meaning of the omega involution of stable Grothendieck polynomials studied in \cite[\textsection 9.7]{LP}. 

Define an algebra homomorphism
\begin{equation}\label{epsilo}
\rho\colon \hat{\Lambda}\longrightarrow K(\Gr(k,n)),\qquad \text{by extending }\quad e_r(x) \longmapsto 
\begin{cases}
c_r(\mathcal{V}^\vee), & r\leq k,\\
0, & r> k.
\end{cases}    
\end{equation}
This is well defined since $\rho(f)$ is nilpotent for $f$ without constant terms by the d\'evissage property \cite[Proposition 5.9.5]{CG}. 
Note that 
$$\pi^*(\rho(f))=f(x_1,\ldots,x_k)
\in K(\Fl(n)). $$

\begin{Rmk}
It should be pointed out that when restricted to the subring $\Gamma=\bigoplus_\lambda \mathbb{Q}\cdot G_\lambda(x)$ of $\hat{\Lambda}$, it can be shown that the map $\rho$ sends $G_\lambda(x)$ to  $[\mathcal{O}_{Y(\lambda)}]$, which has been investigated by Buch \cite[Section 8]{Buch}.
Since we do not use this fact for our purpose, the details will not be  discussed here.    
\end{Rmk}

% {\color{red}
% We include the following computation of $[\omega_{\Gr(k,n)}]$ for reader's convenience.}

\begin{Lemma}\label{BHT121}
We have
$$\rho\big((1-G_{\square}(x))^n\big)=[\omega_{\Gr(k,n)}].$$
\end{Lemma}
\begin{proof}
Recall from \cite[Theorem 3.5]{3264} that
$\mathscr{T}_{\Gr(k,n)}
=\mathscr{H}\!{\it om}(\mathcal{V},\mathcal{Q})=\mathcal{V}^\vee\otimes\mathcal{Q}$ 
where $\mathcal{Q}$ is the universal quotient bundle satisfying the following short exact sequence
$$0\longrightarrow \mathcal{V}\longrightarrow\mathcal{O}^{\oplus n}\longrightarrow \mathcal{Q}\longrightarrow 0.$$
It follows that $\det(\mathcal{V})\otimes \det(\mathcal{Q})\cong \mathcal{O}$.
This implies 
$$\omega_{\Gr(k,n)}
=\det(\mathscr{T}^\vee_{\Gr(k,n)})
=\det(\mathcal{V})^{\otimes (n-k)}\otimes \det(\mathcal{Q}^\vee)^{\otimes k}
=\det(\mathcal{V})^{\otimes n}.$$
It is known (for example, from \eqref{POREQ01}) that
$G_{\square}(x) = e_1(x)-e_2(x)+\cdots $. So we have
\begin{align*}
{}[\omega_{\Gr(k,n)}]
&=[\det(\mathcal{V})]^{n}=\big(1-c_1(\mathcal{V}^\vee)+c_2(\mathcal{V}^\vee)-\cdots\big)^n
=\rho\big((1-G_{\square}(x))^n\big). \qedhere
\end{align*}
\end{proof}

% Let us make the statement precise. 
% The following Lemma is well-known. 

% \begin{Lemma}For any partition $\lambda$, 
% $$\rho_{k,n}(G_\lambda)=\begin{cases}
% {}
% [\mathcal{O}_{Y(\lambda)}], & \lambda\subseteq (n-k)^k,\\
% 0, & \text{otherwise}. 
% \end{cases}$$
% \end{Lemma}
% \begin{proof}
% If $\lambda$ has more than $k$ parts, it is known \cite{FK-1} that 
% $$G_{\lambda}(x_1,\ldots,x_k,0,0,\ldots)=0.$$
% So let us assume $\lambda$ has no more than $k$ parts. 
% By definition \eqref{eq:Glambdadef}, 
% $$G_{\lambda}(x_1,\ldots,x_k,0,0,\ldots)=\mathfrak{G}_{w_{\lambda}}(x,0).$$
% Note that $\lambda\subseteq (n-k)^k$ if and only if $w_{\lambda}\in \S_n$. 
% The K-theoretic version of the last property 
%  of \cite[\textsection 10.3]{AF} implies
% $$\pi^*(\rho_{k,n}(f))=
% \mathfrak{G}_{w_{\lambda}}(x,0)=\begin{cases}
% {}[\mathcal{O}_{Y(w_{\lambda})}], & \lambda\subseteq (n-k)^k, \\
% 0, & \text{otherwise}.
% \end{cases}$$
% By (1) of Theorem \ref{SchRep2:Gr} and (3) of Theorem \ref{th:KBorelinj}, we prove the lemma. 
% \end{proof}

We can now provide a proof of Theorem \ref{ThmE}. 

\begin{Th}[=Theorem \ref{ThmE}]
Over $K(\Gr(k,n))$,  we have
\begin{equation}
\rho\big((1-G_{\square}(x))^nJ_{\lambda'}(x)\big)
=
\begin{cases}
{}    [\omega_{Y(\lambda)}], & \lambda\subseteq (n-k)^k,\\[5pt]
0, & \text{otherwise},
\end{cases}
\end{equation}
where $\lambda'$ is the conjugate of $\lambda$. 
\end{Th}
\begin{proof}
Note that 
$$(1-G_{\square}(x))\cdot J_{\lambda}(x)=s_{\lambda'}(x)+(\text{higher degrees}).$$
So each $f\in \hat{\Lambda}$ can be written as a (possibly infinite) linear combination of $(1-G_{\square}(x))\cdot J_{\lambda}(x)$'s. 
% Since $1-G_{\square}(x)$ is a unit in $\hat{\Lambda}$, 
Hence we  can define a map
$\varphi\colon\hat{\Lambda}\longrightarrow K(\Gr(k,n))$ by letting 
$$(1-G_{\square}(x))^n J_{\lambda'}(x)\longmapsto \begin{cases}
{}    [\omega_{Y(\lambda)}], & \lambda\subseteq (n-k)^k,\\[5pt]
0, & \text{otherwise},
\end{cases}$$
and then extending linearly to $\hat{\Lambda}$. 
Our task is to prove $\varphi=\rho$. 
Note that for $\lambda=\emptyset$,
by Lemma \ref{BHT121},
\begin{equation}\label{PLNGI}
\varphi\big((1-G_{\square}(x))^n\big) = [\omega_{Y(\varnothing)}] = [\omega_{\Gr(k,n)}]
=\rho\big((1-G_{\square}(x))^n\big).   
\end{equation}
It suffices to show that for any $f,g\in \hat{\Lambda}$,
\begin{equation}\label{eq:phi=epsphi}
\varphi(fg)=\rho(f)\cdot \varphi(g),
\end{equation}
i.e., $\varphi$ is a $\hat{\Lambda}$-module morphism. 
This is because once \eqref{eq:phi=epsphi} is given, together with \eqref{PLNGI}, we see that for any $g\in \hat{\Lambda}$, 
\begin{align*}
\varphi(g)
=\rho\big(g(1-G_{\square}(x))^{-n}\big)\cdot \varphi\big((1-G_{\square}(x))^n\big)=\rho\big(g(1-G_{\square}(x))^{-n}\big)\cdot \rho\big((1-G_{\square}(x))^n\big)=\rho(g).
\end{align*}

To prove \eqref{eq:phi=epsphi}, it is enough to verify the case when $f=e_r(x)$ and $g=(1-G_{\square}(x))^nJ_{\lambda'}(x)$, that is, 
\begin{equation}\label{eq:phie=epsphi}
\varphi\big(e_r(x)\cdot (1-G_{\square}(x))^n\cdot J_{\lambda'}(x)\big)
=\rho(e_r(x))\cdot [\omega_{Y(\lambda)}].
\end{equation}
%We shall  prove the following equivalent form
%\begin{equation}
%\label{eq:phie=epsphi}
%\varphi(e_r(x)\cdot (1-G_{\square}(x))^n\cdot (-1)^{|\lambda|}\cdot J_{\lambda'}(x))
%=\rho (e_r(x))\cdot (-1)^{|\lambda|}\cdot [\omega_{Y(\lambda)}].
%\end{equation}
Let us first compute the left-hand side of \eqref{eq:phie=epsphi}. 
By \cite[Equation (3.3)]{Lenart1}, 
\[h_r(x)\cdot \tilde{K}_{\lambda}(x)=\sum_{\mu}(-1)^{|\mu/\lambda|-r} \tilde{K}_{\mu}(x)\]
with the sum over all $\mu\supset \lambda$ such that   $\mu/\lambda$ has exactly $r$ nonempty columns. 
%Now substitute $x_i$ with $-x_i$, we get
%$$(-1)^r\cdot h_r\cdot (-1)^{|\lambda|}\cdot %\tilde{K}_{\lambda}
%=\sum_{\mu} (-1)^{|\mu|}\cdot \tilde{K}_{\mu}.$$
Applying the involution $\omega$ to both sides and multiplied by $(1-G_{\square}(x))^n$, we obtain that
$$e_r(x)\cdot(1-G_{\square}(x))^n\cdot J_{\lambda}(x)
=\sum_{\mu} (-1)^{|\mu/\lambda|-r}\cdot (1-G_{\square}(x))^n\cdot J_{\mu}(x)$$
with the sum over all $\mu\supset \lambda$ such that   $\mu/\lambda$ has exactly $r$ nonempty columns. 
Therefore, 
the left-hand side of \eqref{eq:phie=epsphi} can be expressed as 
\begin{equation}\label{BHTV7}
\varphi\big(e_r(x)\cdot (1-G_{\square}(x))^n\cdot J_{\lambda'}(x)\big)=\sum_{\mu} (-1)^{|\mu/\lambda|-r}\cdot \varphi\big((1-G_{\square}(x))^n\cdot J_{\mu'}(x)\big),    
\end{equation}
where the sum is taken over $\mu\supset \lambda$ such that   $\mu/\lambda$ has exactly $r$ nonempty rows.

When $r>k$, each $\mu$ appearing in  \eqref{BHTV7} has strictly more than $k$ rows, and so  \eqref{BHTV7} equals zero, implying \eqref{eq:phie=epsphi}.
When $r\leq k$, \eqref{BHTV7} becomes 
\begin{equation}\label{BHTV789}
\varphi\big(e_r(x)\cdot (1-G_{\square}(x))^n\cdot J_{\lambda'}(x)\big)=\sum_{\mu} (-1)^{|\mu/\lambda|-r}\cdot [\omega_{Y(\mu)}],    
\end{equation}
where the sum is taken over $\mu\subseteq (n-k)^k$ such that   $\mu/\lambda$ has exactly $r$ nonempty rows. In view of Corollary \ref{I87R43}, we see that \eqref{BHTV789} is equal to $c_r(\mathcal{V}^\vee)\cdot 
[\omega_{Y(\lambda)}]=\rho(e_r(x))\cdot 
[\omega_{Y(\lambda)}]$. This justifies the correctness of   \eqref{eq:phie=epsphi} in the case of $r\leq k$. So the proof is complete.
\end{proof}

\begin{Eg}
Consider the case $\Gr(1,2)=\mathbb{P}^1$. 
Recall that
$$K(\Gr(1,2))=\mathbb{Q}[x]/(x^2),\qquad x=1-[\mathcal{O}(-1)].$$
Under this identification, the map $\rho$ is given by sending 
$f\in \hat{\Lambda}$ to $f(x,0,\ldots) \bmod (x^2)\in K(\Gr(1,2))$.
Notice that 
$$[\omega_{Y(\varnothing)}]=1-2x,\qquad [\omega_{Y(\square)}]=x.$$
% Let us first compute $J_{\square}$. 
Moreover,
we have $G_{\varnothing}(x)=J_{\varnothing}(x)=1$, and 
$$G_{\square}(x)=-\sum_{A} (-1)^{-|A|}x^A
=1-\prod_{i=1}^\infty(1-x_i),\qquad 
J_{\square}(x)=\sum_{B}x^B
=-1+\prod_{i=1}^\infty\frac{1}{1-x_i},$$
where the first (resp., the second) sum goes over nonempty sets $A$  (resp., multisets $B$) of positive integers. 
Thus 
$$(1-G_{\square}(x))^2\cdot J_{\varnothing}(x) = \prod_{i=1}^\infty(1-x_i)^2,\qquad 
(1-G_{\square}(x))^2\cdot J_{\square}(x)
= -\prod_{i=1}^\infty(1-x_i)^2+\prod_{i=1}^\infty(1-x_i).$$
This yields that 
% Applying $\rho_{1,2}$ on them is equivalent to setting $x_2=x_3=\cdots=0$ and modulo $(x^2)$. So we get 
\begin{align*}    
\rho\big((1-G_{\square}(x))^2\cdot J_{\varnothing}(x)\big) & = 1-2x=[\omega_{Y(\varnothing)}],\\
\rho\big((1-G_{\square}(x))^2\cdot J_{\square}(x)\big) &= -(1-2x)+1-x=x
=[\omega_{Y(\square)}]. 
\end{align*}
\end{Eg}

\section{Pieri formulas for other classes}\label{sec67}

In this section, we exhibit the Pieri formulas for other classes as introduced in Section \ref{sec:geom}. As direct consequences of Proposition  \ref{SchRep:Grr} and Theorem \ref{MainThPieri}, we are given  the Pieri formulas for  the classes 
$[Y(\lambda)]$, $[\mathcal{I}_{Y(\lambda)}]$, and $ \CSM(Y(\lambda)^\circ)$. Using the Poincar\'e pairing and the Weyl action, we may further obtain the Pieri formulas for the classes $[\mathcal{O}_{Y(\lambda)}]$ and 
$\SM(Y(\lambda)^\circ)$.

\subsection{Schubert Classes}
This corresponds to 
$(p,q,\hbar)=(0,0,1).$
In this case, for $\lambda\subseteq (n-k)^k$ and $1\leq i\leq k$, it is clear that 
\begin{equation*}\label{EDGT}
\fhead{i}\to [Y({\lambda})] = 
t_{c}\cdot [Y({\lambda})]+
[Y({\mu})],
\end{equation*}
where $\mu\subseteq (n-k)^k$ (if any) is obtained from $\lambda$ by adding a single box in row $i$.

\begin{Coro} 
Let $\lambda\subseteq (n-k)^k$ and  $0\leq r\leq k$. Over $H_T^\bullet(\Gr(k,n))$, we have 
\begin{align}\label{EDGTTY} 
c_r(\mathcal{V}^\vee)\cdot [Y({\lambda})]
=\sum_{1\leq i_1<\cdots<i_r\leq k}
\fhead{i_r}\to\cdots \to
\fhead{i_1}\to   
[Y({\lambda})].
\end{align}
\end{Coro}

\begin{Rmk}
A skew shape $\mu/\lambda$ is called  a {\it vertical strip} if  it has at most one box in the same row. Letting $t=0$ in \eqref{EDGTTY}, it is easy to see that $[Y(\mu)]$ appears in the expansion if and only if 
 $\mu/\lambda$ is a vertical strip with $r$ boxes.
That is,  the  non-equivariant version of \eqref{EDGTTY} is
 \begin{align}\label{YHGTRE} 
c_r(\mathcal{V}^\vee)\cdot [Y({\lambda})]
=\sum_{\mu}
[Y({\mu})],
\end{align}  
where the sum is over $\mu\subseteq (n-k)^k$ such that $\mu/\lambda$ is a vertical strip with $r$ boxes.
It is well known that the Schur polynomial $s_\lambda(x)$ (which is equal to  the single Schubert polynomial $ \mathfrak{S}_{w_\lambda}(x)$) is a polynomial representative of $[Y(\lambda)]$.
So \eqref{YHGTRE} is  the same as the classical Pieri formula for Schur polynomials. 
\end{Rmk}

\subsection{Classes of ideal sheaves and structure sheaves}

% * We shall use $\ftail{}\Leftrightarrow\xfhead{}$ (tail weighted).
% Recall that by Example \ref{SchRep2:Gr}
% $$\mathcal{Y}_{\lambda}=[\mathcal{I}_{Y(\lambda)}],$$
% and $[\mathcal{O}_{Y(\lambda)}]$ is its dual basis. 

For the classes  of ideal sheaves, we set
$(p,q,\hbar)=(1,0,1)$. In this case, the ribbons added in each step are restricted to those with width one, namely, {\it connected} vertical strips.
Formally, for $\lambda\subseteq (n-k)^k$ and $1\leq i\leq k$,
\begin{equation*}\label{IUYE} 
\fhead{i}\to [\mathcal{I}_{\partial Y({\lambda})}] = 
t_{c}\cdot [\mathcal{I}_{\partial Y({\lambda})}]+
\sum_{\mu}
\left(1-t_{\mathtt{h}(\mu/\lambda)}\right)\cdot[\mathcal{I}_{\partial Y({\mu})}],
\end{equation*}
ranging  over $\mu\subseteq (n-k)^k$ such that $\mu/\lambda$ is a connected vertical strip with head in row $i$.

\begin{Coro}\label{BGRFV} 
Let $\lambda\subseteq (n-k)^k$ and  $0\leq r\leq k$. Over $K_T(\Gr(k,n))$, we have  
\begin{align*} 
c_r(\mathcal{V}^\vee)\cdot [\mathcal{I}_{\partial Y({\lambda})}]
=\sum_{1\leq i_1<\cdots<i_r\leq k}
\fhead{i_r}\to\cdots \to
\fhead{i_1}\to
[\mathcal{I}_{\partial Y({\lambda})}].
\end{align*}
\end{Coro}

By Remark \ref{RRR-N}, we have $\MC_y(Y(\lambda)^\circ)|_{y=0} = [\mathcal{I}_{\partial Y(\lambda)}]$, and so Corollary \ref{BGRFV} can also be obtained  from   Theorem \ref{ThmA}   by setting $y=0$. 

Still, it follows from  Remark \ref{RRR-N} that 
$\SMC_y(Y(\lambda)^\circ)|_{y=0} = [\mathcal{O}_{Y(\lambda)}]$, and so we obtain a Pieri formula 
for the classes of structure sheaves from Theorem \ref{ThmB} by putting $y=0$. Let 
\begin{equation}\label{IJGTY}
\xfhead{i}\to [\mathcal{O}_{Y(\lambda)}]= 
t_{c}\cdot [\mathcal{O}_{Y(\lambda)}]+
\sum_{\mu}
\left(1-t_{\mathtt{t}(\mu/\lambda)}\right)\cdot [\mathcal{O}_{Y(\mu)}],    
\end{equation}
ranging    over $\mu\subseteq (n-k)^k$ such that $\mu/\lambda$ is a connected vertical strip with head in row $i$.

\begin{Coro} 
Let $\lambda\subseteq (n-k)^k$ and  $0\leq r\leq k$. Over $K_T(\Gr(k,n))$, we have  
\begin{align}\label{ORWP}
c_r(\mathcal{V}^\vee)\cdot [\mathcal{O}_{Y(\lambda)}]
=\sum_{1\leq i_1<\cdots<i_r\leq k}
\xfhead{i_r}\to\cdots \to
\xfhead{i_1}\to [\mathcal{O}_{Y(\lambda)}].
\end{align}
\end{Coro}

\begin{Rmk}\label{JHG678}
Consider the non-equivariant version of \eqref{ORWP}. When $t=0$, \eqref{IJGTY} becomes
\begin{equation*} 
\hhead{i}\to [\mathcal{O}_{Y(\lambda)}]= 
\sum_{\mu}
 [\mathcal{O}_{Y(\mu)}].    
\end{equation*}
So, over $K(\Gr(k,n))$, we have 
\begin{align}\label{ORWPPP}
c_r(\mathcal{V}^\vee)\cdot [\mathcal{O}_{Y(\lambda)}]
=\sum_{1\leq i_1<\cdots<i_r\leq k}
\hhead{i_r}\to\cdots \to
\hhead{i_1}\to [\mathcal{O}_{Y(\lambda)}].
\end{align}
Hence  $[\mathcal{O}_{Y(\mu)}]$ appears on the right-hand side of \eqref{ORWPPP} if and only $\mu$ is obtained from $\lambda$ by adding consecutively $r$ connected vertical strips from up to down.
Equivalently, $\mu$ is obtained from $\lambda$ by adding  a vertical strip, satisfying  that $\mu/\lambda$ can be partitioned  into $r$ connected vertical strips. 
%Given a vertical strip $\mu/\lambda$, there are ${|\mu/\lambda|-\wdd(\mu/\lambda) \choose  r-\wdd(\mu/\lambda)}$ ways
%to express    $\mu/\lambda$ as a union of  $r$ connected vertical strips, where $\wdd(\mu/\lambda)$ denotes the number of columns of $\mu/\lambda$. Therefore, \eqref{ORWPPP}
%can be rewritten as 
%\begin{align}\label{ORWPYPP}
%c_r(\mathcal{V}^\vee)\cdot [\mathcal{O}_{Y(\lambda)}]
%=\sum_{\mu} {|\mu/\lambda|-\wdd(\mu/\lambda) \choose  r-\wdd(\mu/\lambda)}
% [\mathcal{O}_{Y(\mu)}],
%\end{align}
%where the sum is over $\mu\subseteq (n-k)^k$ such that $\mu/\lambda$ is a vertical strip. 
Just like the role that Schur polynomials play in cohomology,
the symmetric Grothendieck polynomial $\mathfrak{G}_{\lambda}(x)=\mathfrak{G}_{w_\lambda}(x)$
represents the class $[\mathcal{O}_{Y(\lambda)}]$.
Hence \eqref{ORWPPP} is equivalent to 
the formula \cite[(3.1)]{Lenart1} 
by Lenart. 
\end{Rmk}

% are consistent with [REF]. 

% Let us state the operator for $\epsilon_r$. 
% $$\xftail{i}\to [\mathcal{I}_{\partial Y(\lambda)}]
% =
% \tau_c\cdot [\mathcal{I}_{\partial Y({\lambda})}] + 
% \begin{cases}
%     -\tau_{\mathtt{h}(\mu/\lambda)}[\mathcal{I}_{\partial Y(\mu)}] & \text{$\mu=\lambda\cup
% \fbox{\mbox{$\rule{0pc}{1pc}_i$}}$ is inside $(n-k)^k$}\\
%     0 & \text{otherwise}
% \end{cases}
% $$
% $$\ftail{i}\to [\mathcal{O}_{Y(\lambda)}]
% =
% \tau_c\cdot [\mathcal{O}_{Y({\lambda})}] + 
% \begin{cases}
%     -\tau_{\mathtt{t}(\mu/\lambda)}[\mathcal{O}_{Y(\mu)}] & \text{$\mu=\lambda\cup
% \fbox{\mbox{$\rule{0pc}{1pc}_i$}}$ is inside $(n-k)^k$}\\
%     0 & \text{otherwise}
% \end{cases}
% $$

\subsection{CSM classes and Segre--MacPherson  classes}

This time we take  
$(p,q,\hbar)=(1,1,1).$
Then 
\[
\fhead{i}\to \CSM(Y(\lambda)^\circ) 
=t_c\cdot\CSM(Y(\lambda)^\circ) + 
\sum_{\mu} \CSM(Y(\mu)^\circ),
\]
where the sum is over $\mu\subseteq (n-k)^k$ such that $\mu/\lambda$ is a ribbon with head in row $i$.

\begin{Coro}\label{09IJH} 
Let $\lambda\subseteq (n-k)^k$ and  $0\leq r\leq k$. Over $\mathbb{H}_T(\Gr(k,n))$, we have 
\begin{align}\label{9IUYT}
c_r(\mathcal{V}^\vee)\cdot \CSM(Y(\lambda)^\circ) 
=\sum_{1\leq i_1<\cdots<i_r\leq k}
\fhead{i_r}\to\cdots \to
\fhead{i_1}\to   
\CSM(Y(\lambda)^\circ) .
\end{align}
\end{Coro}

\begin{Rmk}
A Pieri formula for non-equivariant CSM classes was first given in  \cite[Theorem 7.3]{FGX}, 
which can be obtained  from  \eqref{9IUYT} by  replacing $\fhead{i}$ by $\xftail{i}$ (by Theorem \ref{sect8-2}) and then letting $t=0$.
\end{Rmk}

Let 
\[
\xfhead{i}\to \SM(Y(\lambda)^\circ) 
=t_c\cdot\CSM(Y(\lambda)^\circ) + 
\sum_{\mu} \CSM(Y(\mu)^\circ),
\]
where the sum is over $\mu\subseteq (n-k)^k$ such that $\mu/\lambda$ is a ribbon with head in row $i$.
Note that in the case $p=q=1$,  the two operators $\fhead{i}$ and  $\xfhead{i}$ 
have no difference since they are independent of   $\mathtt{h}(\mu/\lambda)$ and  $\mathtt{t}(\mu/\lambda)$.

Using Corollary \ref{09IJH}  and the facts (ii), (iv) in Proposition \ref{SchRep3}, we obtain the following Pieri formula for Segre--MacPherson  classes.

\begin{Coro} 
Let $\lambda\subseteq (n-k)^k$ and  $0\leq r\leq k$. Over $\mathbb{H}_T(\Gr(k,n))$, we have 
\begin{align*}
c_r(\mathcal{V}^\vee)\cdot \SM(Y(\lambda)^\circ) 
=\sum_{1\leq i_1<\cdots<i_r\leq k}
\xfhead{i_r}\to\cdots \to
\xfhead{i_1}\to   
\SM(Y(\lambda)^\circ) .
\end{align*}
\end{Coro}

\appendix

\section{Equivalences between ribbon Schubert operators}
\label{sec:Ribboper}

This appendix gives   equivalent formulations concerning the eight ribbon Schubert operators we defined in this paper.

\begin{Th}\label{sect8-2}
For $\lambda\subseteq (n-k)^k$, we have 
\begin{align}
\sum_{1\leq  i_1<\cdots < i_r\leq k}
\fhead{i_r}\to\cdots \to
\fhead{i_1}\to {\lambda}&=\sum_{1\leq  i_1<\cdots < i_r\leq k}
\xftail{i_r}\to\cdots \to
\xftail{i_1}\to {\lambda},\label{OMJU}\\[5pt]
\sum_{1\leq  i_1<\cdots < i_r\leq k}
\xfhead{i_r}\to\cdots \to
\xfhead{i_1}\to {\lambda}&=\sum_{1\leq  i_1<\cdots < i_r\leq k}
\ftail{i_r}\to\cdots \to
\ftail{i_1}\to {\lambda}.\label{OMJUE}
\end{align}
Dually, for $\mu\subseteq (n-k)^k$, we have 
\begin{align}
\sum_{1\leq  i_1<\cdots < i_r\leq k}{\mu}\to 
\fhead{i_r}\to\cdots \to
\fhead{i_1}&=\sum_{1\leq  i_1<\cdots < i_r\leq k} {\mu}\to
\xftail{i_r}\to\cdots \to
\xftail{i_1},\label{OMJU-D}\\[5pt]
\sum_{1\leq  i_1<\cdots < i_r\leq k}\mu\to
\xfhead{i_r}\to\cdots \to
\xfhead{i_1}&=\sum_{1\leq  i_1<\cdots < i_r\leq k}
\mu\to\ftail{i_r}\to\cdots \to
\ftail{i_1}.\label{OMJUE-D}
\end{align}

%Suppose that 
%\begin{equation*} 
%e_r(x_{[k]})\cdot \mathbf{b}_{\lambda}
%=\sum_{1\leq  i_1<\cdots < i_r\leq k}
%\fhead{i_r}\to\cdots \to
%\fhead{i_1}\to \mathbf{b}_{\lambda}.
%\end{equation*}
%Then
%\begin{equation}\label{WETY-12355}
%e_r(x_{[k]})\cdot \mathbf{b}_{\lambda}
%=\sum_{1\leq  i_1<\cdots < i_r\leq k}
%\xftail{i_r}\to\cdots \to
%\xftail{i_1}\to %\mathbf{b}_{\lambda}.
%\end{equation}
\end{Th}

We will present a proof of the equality in \eqref{OMJU}, and the remaining equalities  in Theorem \ref{sect8-2} can be concluded along the same line. To this end,  we defined  the \emph{refined ribbon Schubert operators $v_{ab}$} with $a\leq b$ by specifying   the head and tail values.
Let $\lambda\subseteq (n-k)^k$.  
\begin{itemize}
    \item[(i)] If $a<b$, then set
$$\upsilon_{ab}\cdot\lambda = 
\begin{cases}
\left(\hbar-(p-q)t_{b}\right)p^{\htt(\eta)-1}q^{\wdd(\eta)-1}\cdot\mu,&\hfill (*),\hfill\\[5pt]
0, & \text{otherwise},
\end{cases}$$
where  $(*)$  holds if there exists a (unique)  $\mu\subseteq (n-k)^k$ such that $\mu/\lambda=\eta$ is a ribbon with $\mathtt{h}(\mu/\lambda)=b$   and $\mathtt{t}(\mu/\lambda)=a$  (the verification of the uniqueness is left to the reader).  

\item[(ii)] If $a=b$, then set
$$\upsilon_{aa}\cdot \lambda = 
\begin{cases}
t_{a}\cdot \lambda, & \hfill a\in \{\lambda_i+k+1-i\colon 1\leq i\leq k\},\hfill\\[5pt]
0,  & \text{otherwise}.
\end{cases}$$
\end{itemize}
Comparing with the definitions of $\fhead{i}$ and $\xftail{i}$, we notice that  
\begin{align}\label{eq:upsilon1}
\sum_{1\leq i_1<\cdots<i_r\leq k} \fhead{i_r}\to\cdots\to\fhead{i_1}\to \lambda
& = \sum_{
\begin{subarray}{c}
    b_r<\cdots<b_1\\
    a_r,\ldots,a_1
\end{subarray}} \upsilon_{a_r b_r}\cdots\upsilon_{a_1 b_1}\cdot \lambda,
\end{align}
and that  
% there exsit integers $a_1',\ldots,a_r'$ and $b_1',\ldots,b_r'$ such that
\begin{align}\label{eq:upsilon2}
\sum_{1\leq i_1<\cdots<i_r\leq k} \xftail{i_r}\to\cdots\to\xftail{i_1}\to \lambda
& = \sum_{
\begin{subarray}{c}
    a'_r<\cdots<a'_1\\
    b'_r,\ldots,b'_1
\end{subarray}} \upsilon_{a'_r b'_r}\cdots\upsilon_{a'_1 b'_1}\cdot \lambda.
\end{align}

\begin{proof}[Proof of Theorem \ref{sect8-2}]

It suffices  to show that the right-hand sides of   \eqref{eq:upsilon1} and \eqref{eq:upsilon2} are the same.  We do this by
establishing an explicit bijection 
\[
 \phi\colon\quad \left\{\upsilon_{a_r b_r}\cdots\upsilon_{a_1 b_1}\cdot \lambda\neq 0\colon b_r<\cdots<b_1\right\}\longrightarrow \left\{\upsilon_{a'_r b'_r}\cdots\upsilon_{a'_1 b'_1}\cdot \lambda\neq 0\colon a'_r<\cdots<a'_1\right\}.
\]

To describe $\phi$, we need the following two claims.

\medbreak
\noindent
{\bf Claim 1.} 
For any given $\lambda\subseteq (n-k)^k$, assume that $v_{a'b'}\cdot v_{a b}\cdot \lambda\neq 0$ with $b'<b$. Then (1) $a\neq a'$, and (2) if $a<a'$, 
then $v_{a'b'}\cdot v_{a b}\cdot \lambda=v_{ab}\cdot v_{a' b'}\cdot \lambda$.

\begin{proof}
The arguments are divided into two cases.

\begin{itemize}
    \item[(1)]  $a=b$ or $a'=b'$. This case has three subcases.
    \begin{itemize}
        \item[(i)] $a=b$ and $a'=b'$. In this situation, the assertions in (1) and (2) are trivial.
        
        \item[(ii)] $a=b$ and $a'<b'$. In this situation, the assertion in (1) is clear. Assume that $b=\lambda_i+k+1-i$, and let $\mu\subseteq (n-k)^k$ be such that $\mu/\lambda$ is the unique ribbon with  $\mathtt{h}(\eta_1)=b'$ and  $\mathtt{t}(\eta_1)=a'$. Since $b'<b$, we see that the head of  $\mu/\lambda$ is  strictly below row $i$. This implies that $v_{ab}$ and $v_{a'b'}$ commute, and so we have the assertion in (2). 
        
        \item[(iii)] $a<b$ and $a'=b'$. In this situation, let $\mu\subseteq (n-k)^k$ be such that $\mu/\lambda$ is the unique ribbon with  $\mathtt{h}(\eta_1)=b$ and  $\mathtt{t}(\eta_1)=a$.   Assume that the ribbon $\mu/\lambda$ has head in row $i$ and tail in row $j$, and       
        that $b'=\mu_{i'}+k+1-i'$.   
        Since $b'<b$,   we have $i'>i$. If $i'>j$, then we see that $a>a'$, and $v_{ab}$ and $v_{a'b'}$ commute, and so  the assertions in (1) and (2) are   checked.

        We now consider the case of $i< i'\leq j$. Now it is clear that $a<a'$, and so the assertion in (1) is true. Moreover, we notice  that $b'=\lambda_{i'-1}+k+1-(i'-1)$. Thus both $v_{a'b'}\cdot v_{a b}\cdot \lambda$ and $v_{ab}\cdot v_{a' b'}\cdot \lambda$ are equal to 
        \[
        t_{b'}\,\left(\hbar-(p-q)t_{b}\right)p^{\htt(\mu/\lambda)-1}q^{\wdd(\mu/\lambda)-1}\cdot\mu.
        \]
        This verifies the assertion in (2).

    \end{itemize}

    \item[(2)] $a<b$ and $a'<b'$. This means there exist $\lambda\subset \mu\subset \nu\subseteq (n-k)^k$ such that $\eta_1=\mu/\lambda$ (resp., $\eta_2=\nu/\mu$) is a ribbon with $\mathtt{h}(\eta_1)=b$ and  $\mathtt{t}(\eta_1)=a$ (resp.,  $\mathtt{h}(\eta_2)=b'$ and $\mathtt{t}(\eta_2)=a'$). 
Let us first check $a\neq a'$. 
If the tail of $\eta_2$ is strictly  below the tail of $\eta_1$, then we have $a'<a$, and otherwise we have  $a'>a$. This verifies $a\neq a'$. 

It remains to prove the assertion in (2). Since $a<a'$ and $b'<b$, we see that the head of $\eta_2$ is strictly below the head of $\eta_2$, and the tail of $\eta_2$ is weakly above the tail of $\eta_1$. Let $\eta_2'$ be the ribbon contained  in $\eta_1$ which is obtained  by  sliding  $\eta_2$ upwards  along the north-west to south-east diagonal by one unit, as illustrated below. 
% $$\Young[0.8pc]{
% []&[]&[]&[]&A|&u|&u|\\
% []&[]&F|&=|&J|\\
% F|&I|\scriptstyle\eta_1&J|&A|\\
% H|&F|&I|\scriptstyle\eta_2&|J|\\
% V|&|V|\\
% |l|}
% \longleftrightarrow 
% \Young[0.8pc]{
% []&[]&[]&[]&A|&u|&u|\\
% []&[]&A|&F|&J|\\
% |F|&I|\scriptstyle\eta_2'&J|&H|\\
% V|&F|&I|\scriptstyle\eta_1'&J|\\
% <|&J|\\
% |l|}$$
$$\Young[0.8pc]{
[]&[]&[]&[]&[]&[]&[]&lu|\\
[]&[\scriptstyle\lambda]&[]&[]&[]&A|&u|\\
[]&[]&[]&F|&=|&J|\\
[]&F|&I|\scriptstyle\eta_1&J|&A|\\
[]&H|&F|&I|\scriptstyle\eta_2&J|\\
[]&V|&V|\\
dr|}
\hspace{.5cm}
\Longleftrightarrow 
\hspace{.5cm}
\Young[0.8pc]{
[]&[]&[]&[]&[]&[]&[]&lu|\\
[]&[\scriptstyle\lambda]&[]&[]&[]&A|&u|\\
[]&[]&[]&A|&F|&J|\\
[]&F|&I|\scriptstyle\eta_2'&J|&H|\\
[]&V|&F|&I|\scriptstyle\eta_1'&J|\\
[]&<|&J|\\
dr|}$$
Notice that $\mathtt{h}(\eta_2')=b'$ and  $\mathtt{t}(\eta_2')=a'$. 
Set $\mu'=\lambda\cup \eta_2'$.
Clearly, $\eta_1'=\nu/\mu'$ is a ribbon with $\mathtt{h}(\eta_1')=b$ and  $\mathtt{t}(\eta_1')=a$. Moreover, notice that $\eta_1'$ (resp., $\eta_2'$) has the same height and width as $\eta_1$ (resp,. $\eta_1$). Thus we have $v_{a'b'}\cdot v_{a b}\cdot \lambda=v_{ab}\cdot v_{a' b'}\cdot \lambda$, concluding the assertion in (2). \qedhere
\end{itemize} 
\end{proof}

\medbreak
\noindent
{\bf Claim 2.} 
For any given $\lambda\subseteq (n-k)^k$, if $\upsilon_{a_r b_r}\cdots\upsilon_{a_1 b_1}\cdot \lambda\neq 0$ with $b_r<\cdots<b_1$, then the values $a_1,\ldots, a_r$ are distinct. 

\begin{proof}
The proof is  by induction on $r$.
This is trivial when $r=1$. Assume now that $r>1$. Let $m=\max\{a_1,\ldots, a_r\}$. If there were two indices $1\leq i<j\leq r$ such that $a_i=a_j=m$ (with the assumption  that $a_k<m$ for $i<k<j$), then we could repeatedly apply (2) in Claim 1 to move $\upsilon_{a_j b_j}$ right to obtain that 
\[
\upsilon_{a_r b_r}\cdots\upsilon_{a_j b_j}\cdots\upsilon_{a_i b_i}\cdots\upsilon_{a_1 b_1}\cdot \lambda=\upsilon_{a_r b_r}\cdots\widehat{\upsilon_{a_j b_j}}\cdots\upsilon_{a_j b_j}\cdot\upsilon_{a_i b_i}\cdots\upsilon_{a_1 b_1}\cdot \lambda,
\]
which would vanish because $\upsilon_{a_j b_j}\upsilon_{a_i b_i}=0$ by (1) in Claim 1. So there is only one index, say $t_1$,   such that $a_{t_1}=m$. We now again  use (2) in Claim 1 to move $\upsilon_{a_{t_1} b_{t_1}}$ to the rightmost, yielding  that 
\begin{equation}\label{WERt}
\upsilon_{a_r b_r}\cdots\upsilon_{a_{t_1} b_{t_1}}\cdots\upsilon_{a_1 b_1}\cdot \lambda=\upsilon_{a_r b_r}\cdots\widehat{\upsilon_{a_{t_1} b_{t_1}}}\cdots\upsilon_{a_1 b_1}\cdot\upsilon_{a_{t_1} b_{t_1}}\cdot \lambda\neq 0.    
\end{equation}  
Letting  $\lambda'=\upsilon_{a_{t_1} b_{t_1}}\cdot \lambda\subseteq (n-k)^k$, 
we apply induction to $\upsilon_{a_r b_r}\cdots\widehat{\upsilon_{a_{t_1} b_{t_1}}}\cdots\upsilon_{a_1 b_1}\cdot \lambda'$, concluding that the elements in $\{a_1,\ldots, a_r\}\setminus \{a_{t_1}=m\}$  are distinct. So  $a_1,\ldots, a_r$ are distinct.
\end{proof}

We can now describe the map $\phi$. In fact, the arguments in the proof of  Claim 2 already imply the construction. Given $\upsilon_{a_r b_r}\cdots\upsilon_{a_1 b_1}\cdot \lambda\neq 0$  with $b_r<\cdots<b_1$, it follows from Claim 2 that we can rearrange $a_1,\ldots, a_r$ in increasing order, say   $a_{t_r}<\cdots<a_{t_1}$. Using (2) in Claim 1, we are able to move $v_{a_{t_1} b_{t_1}}$ to the rightmost, as given in   \eqref{WERt}. 
Implementing the same procedure to $\upsilon_{a_r b_r}\cdots\widehat{\upsilon_{a_{t_1} b_{t_1}}}\cdots\upsilon_{a_1 b_1}$, we can eventually rewrite 
$\upsilon_{a_r b_r}\cdots\upsilon_{a_1 b_1}\cdot \lambda$ as
\[
\upsilon_{a_{t_r}b_{t_r}}\cdots\upsilon_{a_{t_1} b_{t_1}}\cdot \lambda,
\]
which is defined as the image of 
$\upsilon_{a_r b_r}\cdots\upsilon_{a_1 b_1}\cdot \lambda$ under $\phi$.

Imitating the above construction, the inverse  of $\phi$ can be easily obtained  based on the following two facts, and  we omit the details here.

\medbreak
\noindent
{\bf Claim 1'.} 
For any given $\lambda\subseteq (n-k)^k$, assume that $v_{a'b'}\cdot v_{a b}\cdot \lambda\neq 0$ with $a'<a$. Then (1) $b\neq b'$, and (2) if $b<b'$, 
then $v_{a'b'}\cdot v_{a b}\cdot \lambda=v_{ab}\cdot v_{a' b'}\cdot \lambda$.

\medbreak
\noindent
{\bf Claim 2'.} For any given $\lambda\subseteq (n-k)^k$, if $\upsilon_{a_r' b_r'}\cdots\upsilon_{a_1' b_1'}\cdot \lambda\neq 0$ with $a_r'<\cdots<a_1'$, then the values $b_1',\ldots, b_r'$ are distinct.
\end{proof}


\begin{thebibliography}{99}
\bibitem{AM}
P. Aluffi and  L. Mihalcea,  {\em Chern--Schwartz--MacPherson classes for Schubert cells in flag manifolds},
Compos. Math. 152 (2016), 2603--2625.


\bibitem{AMSS17}
P. Aluffi, L. Mihalcea, J. Sch\"{u}rmann and C. Su, {\em Shadows of characteristic
cycles, Verma modules, and positivity of Chern--Schwartz--MacPherson classes of Schubert
cells}, Duke Math. J. 172 (17) (2023), 3257--3320. 


\bibitem{AMSS19}
P. Aluffi, L. Mihalcea, J. Sch\"{u}rmann and C. Su, {\em Motivic Chern classes of Schubert cells, Hecke algebras, and applications to Casselman's problem}, to appear in Ann. Sci. \'Ec. Norm. Sup\'er, 	arXiv:1902.10101v3. 

\bibitem{AMSS22}
P. Aluffi, L. Mihalcea, J. Sch\"{u}rmann and C. Su, {\em From motivic Chern classes of Schubert cells to their Hirzebruch and CSM classes},  arXiv:2212.12509.

\bibitem{AMSS}
P. Aluffi, L. Mihalcea, J. Sch\"{u}rmann and C. Su, {\em Equivariant motivic Chern classes via mixed Hodge modules on the cotangent bundle}, in preparation.


\bibitem{AF}
D. Anderson and W. Fulton,
{\em Equivariant Cohomology in Algebraic Geometry}, Cambridge
Studies in Advanced Mathematics, Cambridge Univ. Press, Cambridge, 2023 (A preliminary version is available at
https://people.math.osu.edu/anderson.2804/ecag/index.html).



% \bibitem{Brion2002}
% M. Brion. Positivity in the Grothendieck group of complex flag varieties. J. Algebra, 258:137–159, 2002.


\bibitem{BGG}
I.  Bern\u{s}te\u{i}n, I.  Gel'fand and S.  Gel'fand, {\em Schubert cells, and the cohomology of the
spaces $G/P$}, Uspekhi Mat. Nauk 28 (1973), 3--26. 


\bibitem{Borel}
A. Borel, {\em Sur la cohomologie des espaces fibr\'{e}s principaux et des espaces homog\'{e}nes de
groupes de Lie compacts}, Ann.  Math.   57 (1953), 115--207. 

\bibitem{BSY}
J.  Brasselet, J. Sch\"urmann and S. Yokura, {\em Hirzebruch classes and motivic Chern classes for singular spaces}, J. Topol.  Anal.  2 (2010), 1--55.


\bibitem{Brion1997}
M. Brion, {\em Equivariant cohomology and equivariant intersection theory}, Representation theories and algebraic geometry (Montreal, PQ, 1997), NATO Adv. Sci. Inst. Ser. C: Math. Phys. Sci., vol. 514, Kluwer
Acad. Publ., Dordrecht, 1998, pp. 1--37.


\bibitem{Buch}
A. Buch, {\em A Littlewood--Richardson rule for the K-theory of Grassmannians}, Acta Math. 189 (2002), 37--78.

\bibitem{BCMP}
A. Buch, P. Chaput, L. Mihalcea and N. Perrin,
{\em A Chevalley formula for the equivariant quantum K-theory of cominuscule varieties},  Algebr. Geom. 5 (2018), 568--595. 

\bibitem{CS}
S. Cappell and J. Shaneson, {\em Stratifiable maps and topological in variants}, J. Amer. Math. Soc. 4 (1991), 521--551.

\bibitem{CG}
N. Chriss and V. Ginzburg, {\em Representation theory and complex geometry}, Boston: Birkh\"auser, 1997.


\bibitem{Demazure}
 M. Demazure, {\em D\'esingularisation des vari\'et\'es de Schubert g\'en\'eralis\'ees}, Annales scientifiques de l'\'Ecole Normale Sup\'erieure 7 (1) (1974), 53--88.

\bibitem{Deodhar}
V. Deodhar, {\em On some geometric aspects of Bruhat orderings II, The parabolic analogue of Kazhdan-Lusztig
polynomials}, J. Algebra 111 (1987), 483--506.


\bibitem{Drin}
V. Drinfeld,  {\em Degenerate affine Hecke algebras and Yangians}, Funktsional'nyi Analiz i ego Prilozheniya,  20(1) (1986), 69--70.

\bibitem{3264}
D. Eisenbud and J. Harris, {\em 3264 and all that: A second course in algebraic geometry}, Cambridge University Press, 2016.

\bibitem{FGX}
N. Fan, P. Guo and R. Xiong,
{\em Pieri and Murnaghan--Nakayama type rules for Chern classes of Schubert cells},  arXiv:2211.06802v1.

\bibitem{FRW}
L.  Feh\'er,  R.  Rim\'anyi and
 A. Weber,  {\em Motivic Chern classes and K-theoretic stable envelopes}, Proc. London Math. Soc. 122  (2021), 153--189.

\bibitem{Fomin}
S. Fomin,  {\em Schur operators and Knuth correspondences}, J. Combin.  Theory Ser. A 72  (1995), 277--292.

\bibitem{FG}
S. Fomin and C. Greene, {\em Noncommutative Schur functions and their applications}, Discrete Math. 193 (1998), 179--200.

\bibitem{FK-1}
S. Fomin and A. Kirillov, {\em The Yang--Baxter equation, symmetric functions, and Schubert
polynomials},  Discrete Math. 153 (1996), 123--43.


\bibitem{GK}
W. Graham and S. Kumar, {\em On positivity in T-equivariant K-theory of flag varieties}, Int.
Math. Res. Not. IMRN (2008), rnn093.


\bibitem{Hump}
J. Humphreys, {\em Reflection Groups and Coxeter Groups}, Cambridge University Press, 1992.

\bibitem{Iwahori}
N. Iwahori, {\em On the structure of a Hecke ring of a Chevalley group over a finite field}, J. Fac. Sci. Univ. Tokyo Sect. I 10 (1964), 215--236. 

\bibitem{IM95}
N. Iwahori and H. Matsumoto, {\em On some Bruhat decomposition and the structure of the Hecke rings of $ p $-adic Chevalley groups}, Publications Math\'{e}matiques de l'IH\'ES 25 (1965), 5--48.

\bibitem{KL}
D.  Kazhdan and G. Lusztig,  {\em Representations of Coxeter groups and Hecke algebras}, Invent. Math. 53 (1979),  165--184.

\bibitem{K22} J. Koncki, {\em Comparison of motivic Chern classes and stable envelopes for cotangent bundles}, J. Topology 15 (2022), 168--203.

\bibitem{KW}
J. Koncki and A. Weber, {\em Twisted motivic Chern class and stable envelopes}, Advances in Mathematics, 404, Part A, 2022

%\bibitem{KZJ}
%A. Knutson and P. Zinn-Justin, {\em Schubert puzzles and integrability II: multiplying motivic Segre classes},arXiv:2102.00563. 

\bibitem{KK}
B. Kostant and S. Kumar, {\em T-equivariant K-theory of generalized flag varieties}, J. Diff. Geom. 32 (1990), 549--603.


\bibitem{Lam}
T. Lam,  {\em Ribbon Schur operators},  Europ. J. Combin. 29 (2008), 343--359.

\bibitem{LP}
T. Lam and P. Pylyavskyy, {\em Combinatorial Hopf algebras and K-homology of Grassmanians},  Int.
Math. Res. Not. IMRN 9 (2007), rnm125.

\bibitem{LS} 
A. Lascoux and M.-P. Sch\"utzenberger, {\em Polyn\^omes de Schubert}, C. R. Math. Acad. Sci. Paris, S\'er. I Math.  294 (1982), 447--450.

\bibitem{LS-2}
A. Lascoux and M.-P. Sch\"utzenberger,  {\em Structure de Hopf de l'anneau de cohomologie et de l'anneau de Grothendieck d'une variété de drapeaux}, C. R. Acad. Sci. Paris S\'er. I Math. 295 (1982),   629--633.



\bibitem{Lenart1}
C. Lenart, {\em Combinatorial aspects of the K-theory of Grassmannians}, Ann. Combin. 4 (2000), 67--82.

%\bibitem{Lenart2}
%C. Lenart and F. Sottile, {\em A Pieri-type formula for the K-theory of a flag manifold},  Trans. Amer. Math. Soc., 359 (2007), 2317--2342. 

\bibitem{LSZZ}
C. Lenart,  C. Su, K. Zainoulline and C. Zhong, {\em Geometric properties of the Kazhdan--Lusztig Schubert basis}, Algebra \& Number Theory  17 (2) (2023), 435--464.

\bibitem{LM}
M. Levine and F. Morel, {\em Algebraic cobordism}, Springer Science \& Business Media, 2007.

\bibitem{Lusztig1}
G. Lusztig, {\em Some examples of square integrable representations of semisimple $p$-adic groups}, Trans. Amer. Math. Soc.  277 (1983),  623--653.

\bibitem{Lusztig2}
G. Lusztig, {\em Affine Hecke algebras and their graded version}, J. Amer. Math. Soc. 2 (3) (1989), 599--635.

\bibitem{Mac}
I. Macdonald, {\em Symmetric functions and Hall polynomials}, second edition, Oxford Math.
Mon., 1995.


\bibitem{MacPherson}
R. MacPherson,  {\em Chern classes for singular algebraic varieties}, Ann. Math.  100 (1974), 423--432.

\bibitem{MNS}
L. Mihalcea, H. Naruse and C. Su, 
{\em Left Demazure--Lusztig operators on equivariant (quantum) cohomology and K-theory}, Int. Math. Res. Not. IMRN 16 (2022), 12096--12147.


\bibitem{MNS2}
L. Mihalcea, H. Naruse and C. Su, 
{\em Chevalley formulae for the motivic Chern classes of Schubert cells and for the stable envelopes}, 	arXiv:2312.17200.

\bibitem{MSA}
L. Mihalcea,  C. Su and D. Anderson, {\em Whittaker functions from motivic Chern classes}, Transform. Groups 27 (3) (2022), 1045--1067.

\bibitem{Ohmoto}
T. Ohmoto, {\em Equivariant Chern classes of singular algebraic varieties with group actions}, Math.
Proc. Camb. Philos. Soc. 140 (2006), 115--134.

\bibitem{Okounkov}
A. Okounkov,  {\em Lectures on K-theoretic computations in enumerative geometry}, Geometry of moduli spaces and representation theory, 251-V380, IAS/Park City Math. Ser., 24, AMS, 2017.

%\bibitem{PY}
%O. Pechenik and A. Yong,
%{\em Equivariant K-theory of Grassmannians}. Forum Math. Pi,  e3 (2017),  128pp.


\bibitem{Rimanyi}
R. Rim\'anyi,  {\em $\hbar$-Deformed Schubert Calculus in equivariant cohomology, K-theory, and elliptic cohomology,} Singularities and their interaction with geometry and low dimensional topology: In honor of András Némethi (2021), 73--96.


\bibitem{Sch65a} 
M. Schwartz, {\em Classes caract\'eristiques d\'efinies par une stratification
d’une vari\'et\'e analytique complexe, I}, C. R. Acad. Sci. Paris 260 (1965), 3262--3264.

\bibitem{Sch65b} 
M. Schwartz, {\em Classes caract\'eristiques d\'efinies par une stratification
d’une vari\'et\'e analytique complexe, II}, C. R. Acad. Sci. Paris, 260 (1965), 3535--3537.

\bibitem{Seo}
W. Soergel, {\em Kazhdan--Lusztig polynomials and a combinatoric for tilting modules}, Represent. Theory  1 (6) (1997), 83--114.


%\bibitem{Sottile1}
%F. Sottile, {\em Pieri’s formula for flag manifolds and Schubert polynomials}, Ann. de l’Institut Fourier,  46 (1996), 89--110.

\bibitem{Stanley2}
R. Stanley, {\em Enumerative Combinatorics, Vol. 2},  Cambridge University Press, Cambridge, 1999.

% \bibitem{SZZ}
% C. Su, G. Zhao and C. Zhong, 
% {\em On the K-theory stable bases of the Springer resolution}, Ann. Sci. \'Ec. Norm. Sup\'er., 53(3) (2020), 663--711.

% \bibitem{WZ19}
% M. Wheeler and P. Zinn-Justin, Littlewood--Richardson coefficients for Grothendieck polynomials from integrability, J. Reine Angew. Math.  757 (2019), 159--195.

\bibitem{Yokura}
S. Yokura, {\em  A singular Riemann--Roch theorem for Hirzebruch characteristics},  Banach Center Publ. 44 (1998), 257--268. 

\end{thebibliography}
\end{document}